\documentclass[a4paper]{amsart}
 
\usepackage{amsthm}
\usepackage{mathabx}
\usepackage{amsmath}
\usepackage{amssymb}
\usepackage{latexsym}
\usepackage{graphicx}
\usepackage[utf8]{inputenc}
\usepackage{epsfig} 
\usepackage{pgf}
\usepackage{tikz}
\usepackage{float}
\usepackage{dsfont}
\usepackage{hyperref}
\usepackage [normalem] {ulem}

\newtheorem*{theorem*}{Theorem}
 
\newtheorem*{corollary*}{Corollary}

\renewcommand{\thetheoremName}

\newtheorem{theorem}{Theorem}[section]

\newtheorem{lemma}[theorem]{Lemma}
\newtheorem{proposition}[theorem]{Proposition}
\newtheorem{corollary}[theorem]{Corollary}

\newtheorem{definition}[theorem]{Definition}
\newtheorem{example}[theorem]{Example}

\newtheorem{remark}[theorem]{Remark}

\numberwithin{equation}{section}

\definecolor{alert}{rgb}{0.8,0,0.3}

\newcommand{\alert}[1]{%
	\marginpar{%
		\ifodd\value{page} \raggedright \else \raggedleft \fi
		\footnotesize{\textcolor{alert}{#1}}
	}
}

\newcommand{\Hess}{\operatorname{Hess}}
\newcommand{\dist}{\operatorname{dist}}
\newcommand{\Vol}{\operatorname{Vol}}

\newcommand{\Div}{\operatorname{div}}
\newcommand{\C}{\operatorname{Cap}}

\newcommand{\Id}{\operatorname{Id}}

\newcommand{\WCap}{\operatorname{Cap_{\W}}}

\newcommand{\Reff}{\operatorname{R_{eff}}}

\newcommand{\W}{\mathcal{W}}

\begin{document}

\title[Criteria for parabolicity and hyperbolicity]{Criteria for parabolicity and hyperbolicity\\
of conductive Riemannian manifolds}

\author[V. Gimeno-Garcia]{Vicent Gimeno i Garcia}
\address{Departament de Matem\`{a}tiques- IMAC,
Universitat Jaume I, Castell\'{o}, Spain.}
\email{gimenov@uji.es}
\author[A. Hurtado]{Ana Hurtado}
\address{Departamento de Geometr\'{\i}a y Topolog\'{\i}a and Excellence Research Unit ``Modeling Nature'' (MNat), Universidad de Granada, E-18071,
Spain.}
\email{ahurtado@ugr.es}
\author[S. Markvorsen]{Steen Markvorsen}
\address{
 DTU Compute, Technical University of Denmark, Lyngby, Denmark
}
\email{stema@dtu.dk}
\author[V. Palmer]{Vicente Palmer}
\address{Departament de Matem\`{a}tiques- INIT,
Universitat Jaume I, Castell\'{o}, Spain.}
\email{palmer@uji.es}
\thanks{V. Gimeno and V. Palmer were partially supported by the Research grant  PID2020-115930GA-100 funded by MCIN/ AEI /10.13039/501100011033, and  by the project AICO/2023/035 funded by Conselleria de Educaci\'o, Cultura, Universitats i Ocupaci\'o.\\
A. Hurtado was partially supported by the grant PID2023-151060NB-I00 funded by MCIN/AEI/10.13039/501100011033 and by ERDF/EU, and by the grant PID2020-118180GB-I00 funded by MCIN/AEI/10.13039/501100011033.}

\maketitle
\begin{abstract}
Motivated by the physics of anisotropic conductive materials we consider a linear elliptic operator $\Delta_{\W}$ of divergence type on a Riemannian manifold $(M^{n}, g)$. The operator is determined by the metric $g$ and by a given \emph{conductivity}, which is modeled by a smooth self adjoint tensor field $\W$ of type $(1,1)$. We establish new conditions for a conductive manifold $(M, g, \W)$ to be $\W$-parabolic or $\W$-hyperbolic. Here, by definition, a $\W$-hyperbolic manifold (as opposed to a $\W$-parabolic manifold) admits an effective electric current $J$, i.e. a bounded potential function  $u$, which is a solution to the $\W$-Laplace equation $\Delta_{\W}(u) = 0$ with a finite flux of the current $J = -\W(\nabla u)$ to infinity. We prove a number of intrinsic conditions on $g$ and $\W$, that tell the type ( $\W$-hyperbolic or  $\W$-parabolic) of conductive Riemannian manifolds. And we prove similar extrinsic conditions for submanifolds (involving also naturally the second fundamental form), that give the type of the submanifolds when they are endowed with the inherited conductivities from the ambient conductive space. Our results are furthermore illustrated by corresponding families of examples, which emphasize how the present setting and results generalize previous findings concerning the usual Laplacian for Riemannian manifolds (with homogeneous, constant, conductivity) as well as similar recent results for weighted manifolds and submanifolds -- see e.g. \cite{Grigoryan1999BAMS}, \cite{GAFA}, \cite{Soliton2021} and references therein. We also present novel examples of $\W$-hyperbolic manifolds where the conductivity tensor is 'extracted' from the curvature tensor of the manifold itself, such as e.g. the metric equivalents of the Einstein tensor and the Schouten tensor.
\end{abstract}
\tableofcontents

\section{Introduction} \label{sec:Intro}

Consider a cylindrical surface $S$ of length $L$ and cross sectional radius $\alpha$ in $3$D --  like the one depicted in figure \ref{fig:Cyl1} below -- and suppose that the cylinder is made of a thin homogeneous electrically conducting material. Now endow the two circular ends of the cylinder with different potentials -- using a battery -- as in the middle display. At the left hand end, $A$, of the cylinder we impose the potential $u = 1$ and at the right hand end, $B$, we let $u=0$. In terms of the distance $r$ from $A$ along the cylinder this means $u(0) = 1$ and $u(L) = 0$. What is then the induced intermediate potential function $u(r)$, $0 < r < L$? \\

\begin{figure}[h]
    \centering
    \includegraphics[width=37mm]{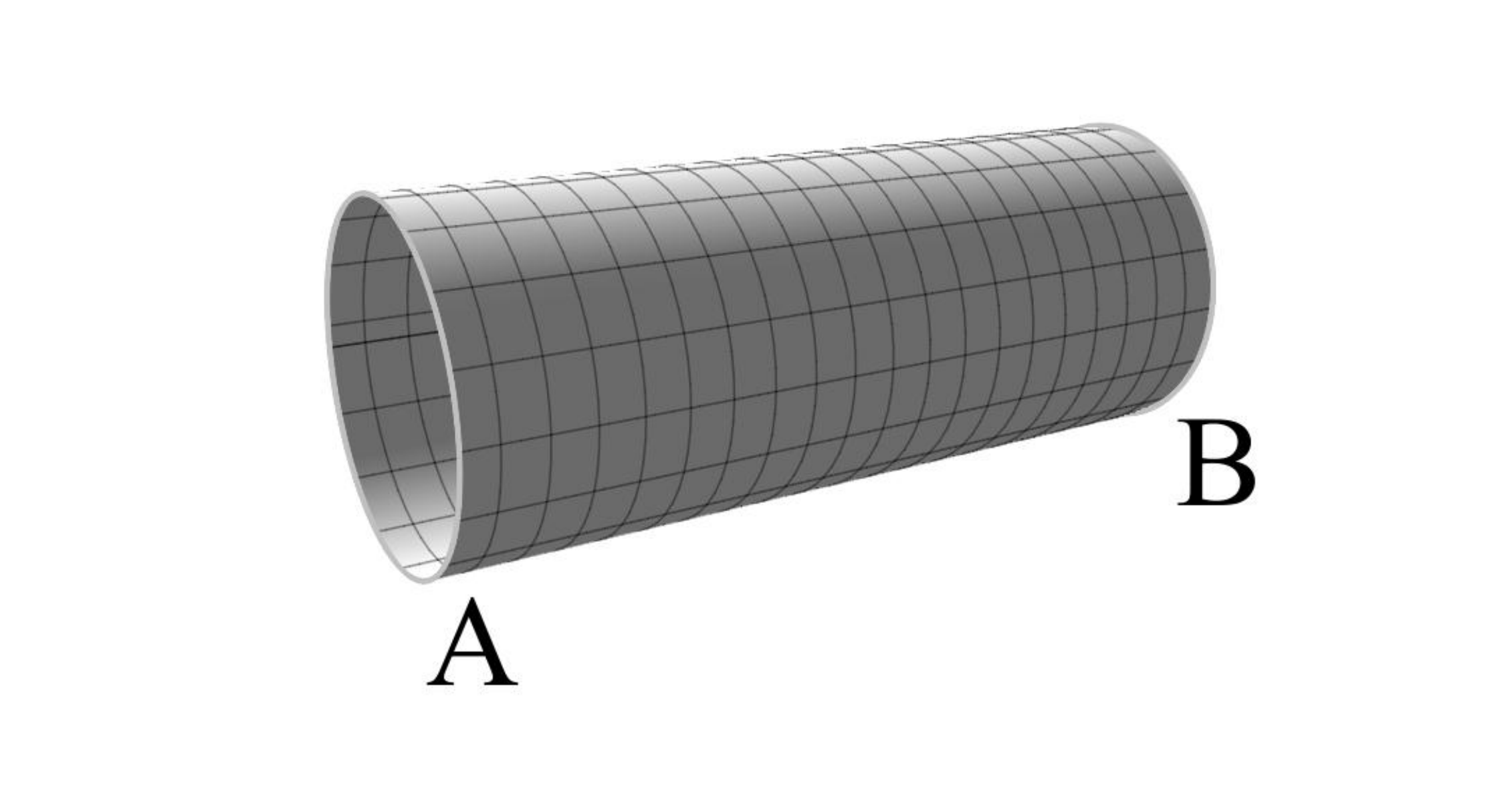}\qquad\includegraphics[width=37mm]{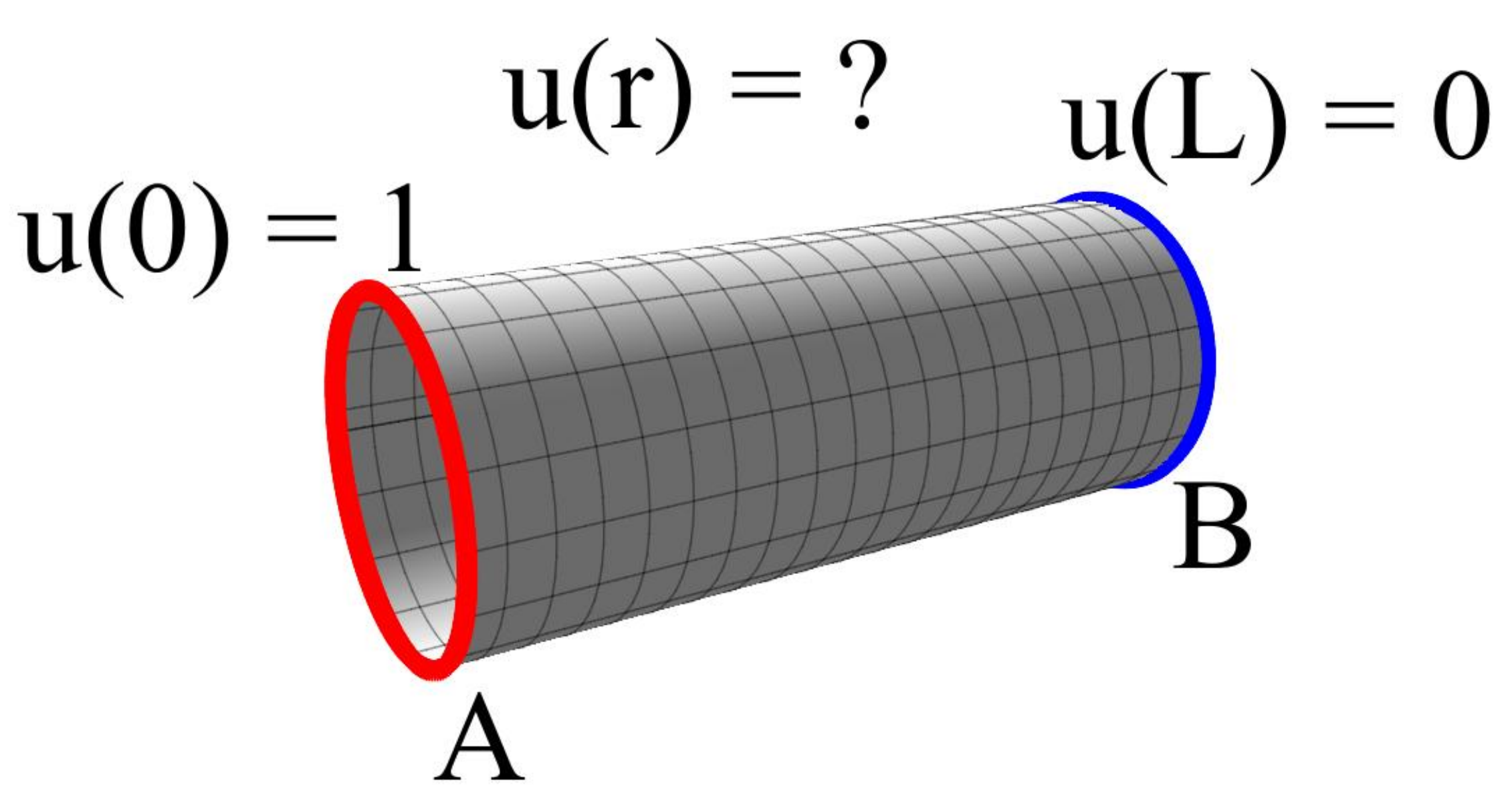}
    \qquad \includegraphics[width=37mm]{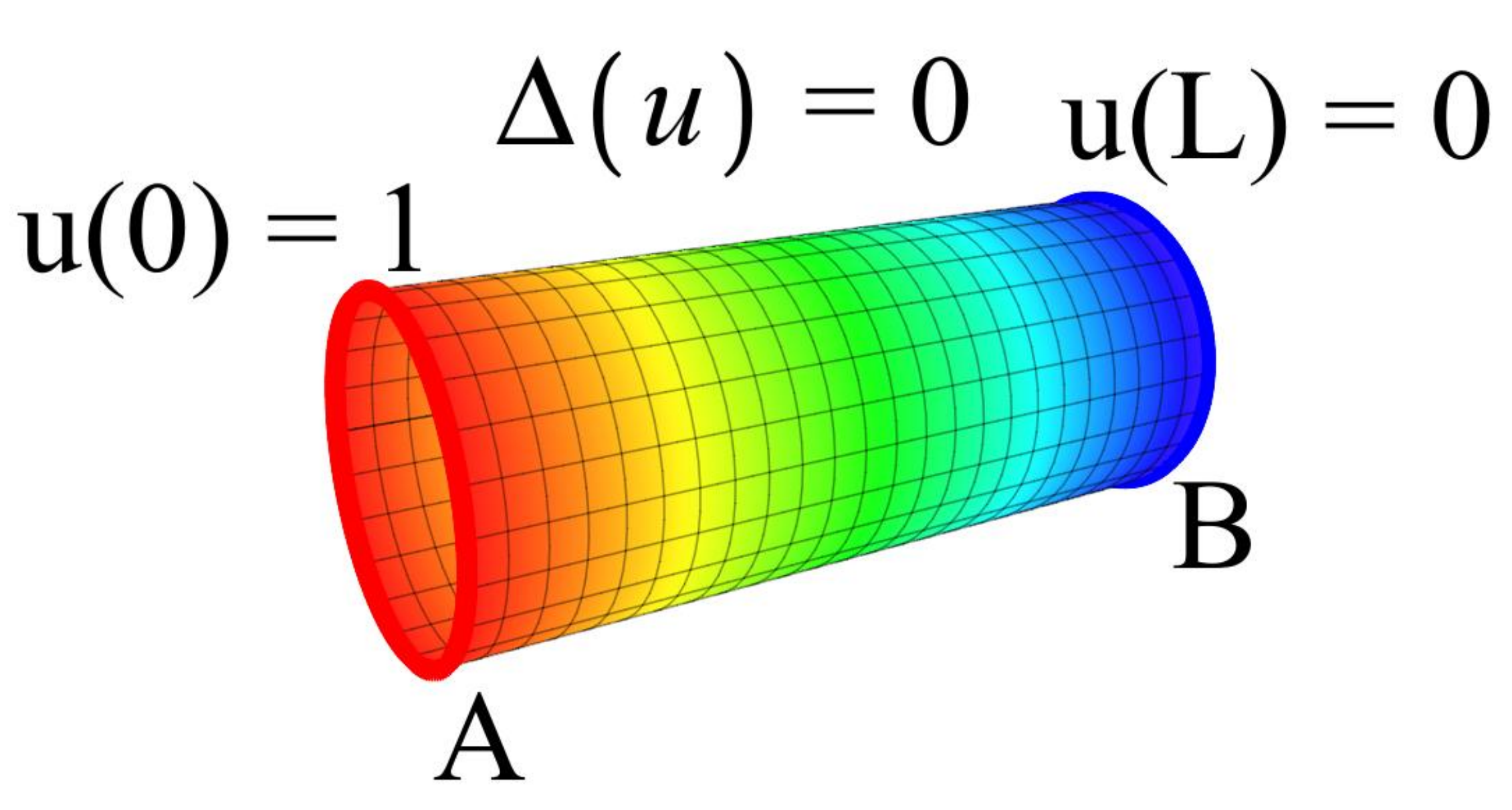}\caption{A finite cylinder with linear (colored) potential function $u(r) = (L-r)/L$ generated by $u(0)=1$ at $A$, $u(L) = 0$ at $B$ and $\Delta\, u =0$ in between.}
    \label{fig:Cyl1}
\end{figure}

The conductivity of the material has the following consequence, which will give the unique answer to that question: The (apriorily unknown) potential function $u$ drives the positive charges away from $A$ towards $B$ in such a way that the flow vector field (the electric current) satisfies the local (infinitesimal) Ohm's law: $J = -c\cdot \nabla u$, where the constant $c>0$ denotes the scalar homogeneous conductivity. But, since the mass of the positive charges is preserved by the flow, the local (and global) flux of $J$  out of any compact subset of $S$ must satisfy $\Div(J) = 0$. In consequence therefore, $\Delta\, u = \Div(\nabla u) = -(1/c)\cdot \Div (J) = 0$, which means that $u$ is a harmonic function and thence uniquely determined by the given boundary conditions.\\

In this particular case (due to the rotational symmetry of the cylinder and the specific boundary conditions) we get the simple solution $u = (L-r)/L$, and thence $J = c/L$ and the total influx current at $A$ (equal to the total outflux current at $B$): $I = 2\pi\rho\cdot c / L$. In total we thus arrive at the well known global Ohm's law for the cylinder: $U = \Reff(S)\cdot I$ with $U = 1$, i.e. the effective resistance of the cylinder is $\Reff(S) = L/ (2\pi\alpha\cdot c)$. The capacity of the cylinder is then, by definition, $\C(S)= 1/\Reff = 2\pi\alpha\cdot c/L$. In this work we shall be particularly interested in the capacities of complete surfaces and manifolds that extend to infinity, like the cylinder for $L \to \infty$. In this limit, in this case, we obviously have $\C(S) = 0$, or equivalently $\Reff(S) = \infty$, so that it is not possible to drive a current from $A$ to the other end at infinity of the homogeneous cylinder. In such cases we say that the surface is of \emph{parabolic type}, a notion that will be elaborated upon in detail below. \\

However, before getting into further technical details concerning the type problem for more general manifolds we must mention that the problem naturally has two possible answers, depending on both the conductivity and on the geometry of the manifold. For example, by choosing other non-constant, non-homogeneous, tensorial conductivities on the cylinder we \emph{can} drive a finite current from $A$ to $\infty$ on the cylinder considered above, still with global potential difference $U = u(0) - u(\infty) = 1$, and thus obtain a surface that is (with this new conductivity) of \emph{hyperbolic type}, i.e. a surface with $\C(S) > 0$. The effect of choosing a non-homogeneous conductivity is illustrated on a finite portion of the cylinder in figure \ref{fig:Cyl2} below.  Also, we may change \emph{the  geometry} of the cylinder into a surface of revolution with a more general profile curve,  as exemplified in figure \ref{fig:Cyl3}, where the surface is again endowed with a constant homogeneous conductivity.

\begin{figure}[h]
    \centering
    \includegraphics[width=45mm]{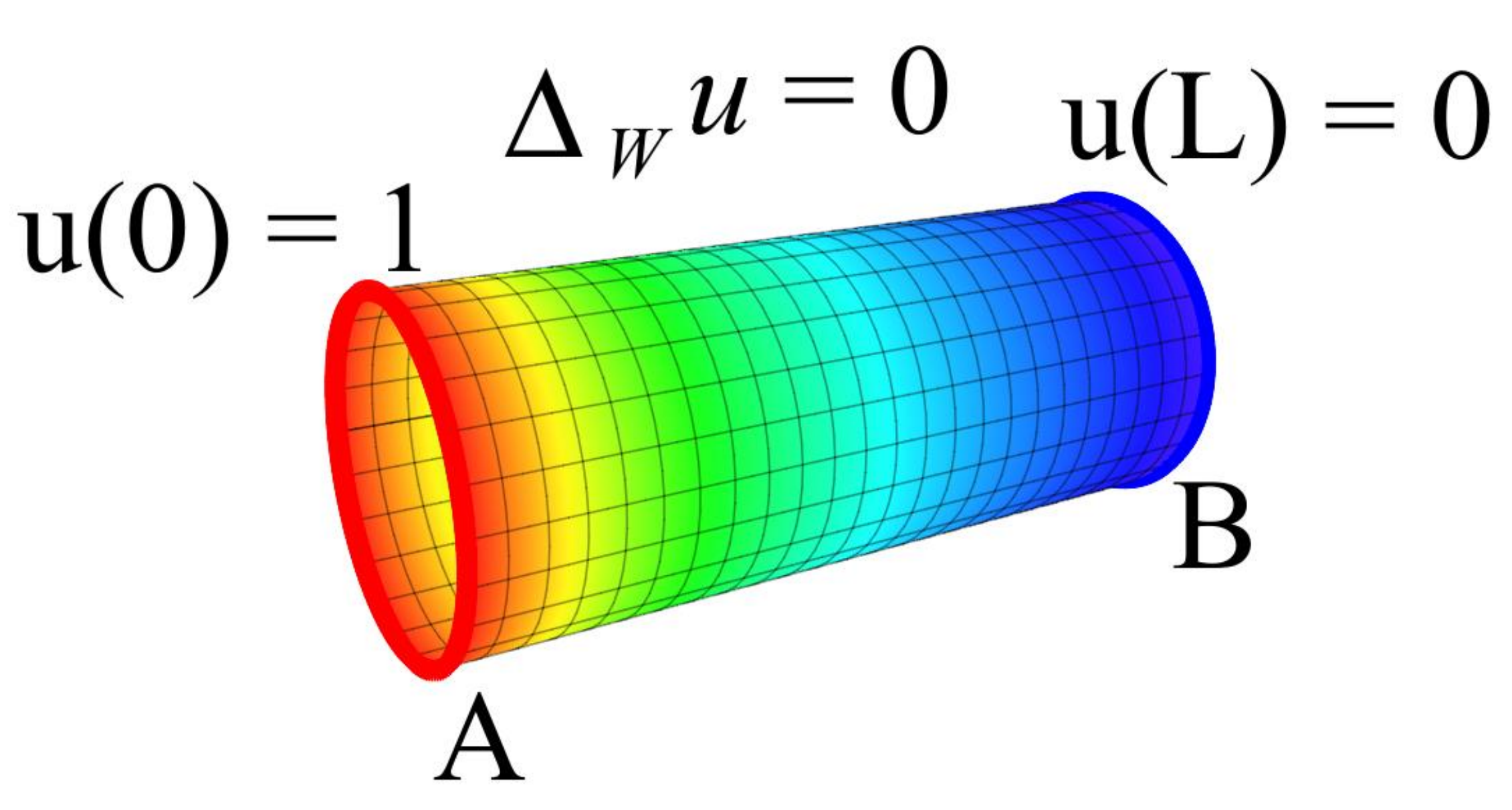}\qquad\includegraphics[width=45mm]{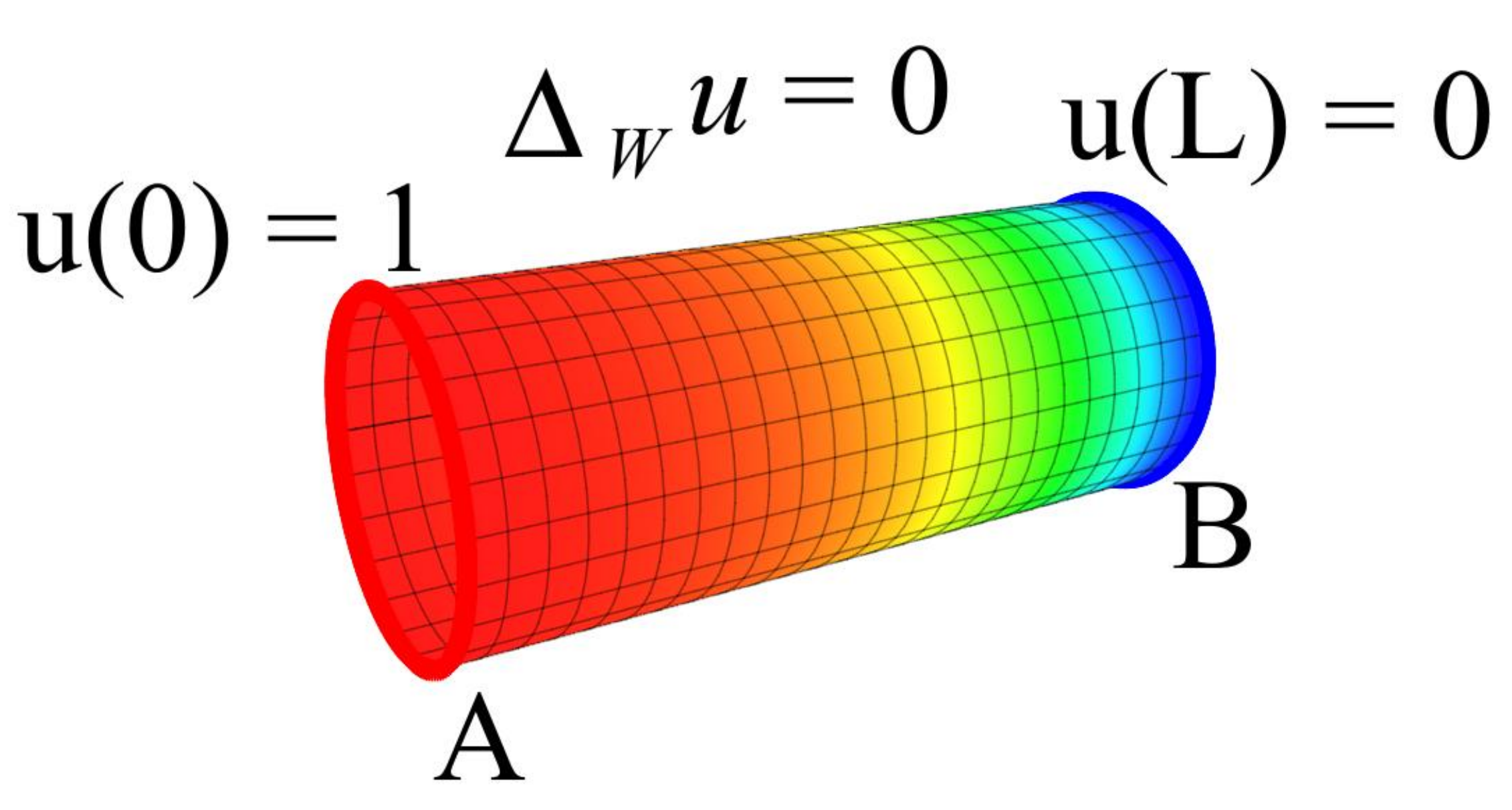}
    \caption{A finite cylinder endowed with two different non-constant $r$-dependent conductivities $\W$. They have corresponding non-linear (colored) potential functions $u(r)$, again generated by $u(0)=1$ at $A$, $u(L) = 0$ at $B$ and $\Delta_{\W}\,u =0$ in between, where now $\Delta_{\W}$ is the $\W$-dependent elliptic operator replacing the usual Laplacian in each of the cases.}
    \label{fig:Cyl2}
\end{figure}

\section{First glimpses of our results and examples} \label{sec:Glimpses}
As already hinted in the introduction, the most general setting in which we will consider the type problem has both nontrivial geometry and nontrivial conductivity. The conductivity $\W$ is given as a smooth positive definite and self adjoint tensor field of type $(1,1)$ on $M$,  $\W \in \mathcal{T}_{1}^{1}(M)$, which molds the key generalized Laplacian operator on $M$ as follows:  

\begin{equation} \label{eq:LaplaceWeq}
\Delta_{\W}\, u = \Div (\W(\nabla u)) \,\, \textrm{for all smooth functions $u \in C^{\infty}(M).$}
\end{equation}

\begin{definition}
Such a Laplacian $\Delta_{\W}$ will be called a \emph{conductive Laplacian}, or simply the $\W$-Laplacian, on $(M, g)$, and  $\W$ itself will be called a \emph{conductivity} on the \emph{conductive Riemannian manifold} $(M, g, \W)$. 
\end{definition}

To elaborate a bit further on the physical meaning and intuition of this otherwise purely geometric definition, we now consider an abstract generalization of the finite $2$-dimensional conductive flat cylinder considered above. Technically, and as a warm up for our extended use of so-called model spaces below (see Definition \ref{model}) we warp the cylinder into a warped product
 $[\rho, \infty) \times_{w} \mathbb{S}^{1}$, $\rho > 0$ with a positive warping function $w$ and projection $\pi$ onto the first factor. The metric $g$ is given by $ds^{2} = dr^{2} + w(r)d\theta^{2}$, where $r\in [\rho, \infty)$ and $\theta \in \mathbb{S}^{1}$.  The manifold $(M_{w}^{2}, g)$ is thence a generalized surface of revolution with boundary $\pi^{-1}(\rho)$.  (Such a $2$-manifold can be isometrically embedded into Euclidean 3-space if and only if  $w'(r) \leq 1$ for all $r$,  since $w(r)$ in that case measures the distance to the axis of revolution.) As mentioned, we now imagine that this surface, this abstract $2$-manifold, is again physically made of a conductive material with conductivity field $\W$, which will produce an (electric) current vector field $J$ on the surface whenever we impose fixed (voltage) potentials on suitable subsets of the surface. For example, as in the cylinder case, we can assume that we endow the circle fiber $\pi^{-1}(\rho)$ with the constant potential $u_{\rho} = 1$ and that we endow the fiber $\pi^{-1}(R)$, $\rho < R$, with the constant potential $u_{R} = 0$. Those given boundary values then induce a potential function $u$ and a corresponding current $J = -\W(\nabla u)$, where $\W$ denotes the conductivity endomorphism operator on each tangent space $T_{p}M$ for the given material, and $\nabla u$ denotes the $g$-gradient of $u$. Again, by preservation of particles, $\Div (J) = 0$, so  the potential $u$ is determined by the $\W$-Laplace equation as follows:  
\begin{equation} \label{1:eq:WLaplace}
 \Delta_{\W}\, u =\Div (\W(\nabla u)) = 0 \, \,  , \, \, u(\pi^{-1}(\rho)) = 1 \,\, , \,\, u(\pi^{-1}(R)) = 0 \quad .
\end{equation}

The simplest possible case is again where the conductive material alluded to above is homogeneous so that the conductivity is constant in all directions, i.e. $\W(X) = c\cdot X$ for all tangent vectors $X \in TM$ and some positive constant $c$. Then the unique solution to the Laplacian boundary value problem \eqref{1:eq:WLaplace} is constant on each fiber and given by:
\begin{equation} \label{eqCylPot}
u(\pi^{-1}(r)) =  \frac{\int_{r}^{R}(1/w(t))\, dt}{\int_{\rho}^{R}(1/w(t))\, dt} \,\,, \,\, \textrm{for all $r\in [\rho, R]$} \quad .
\end{equation}

This is but a classical result, which has recently been generalized much further to weighted warped products, see \cite{hurtado2020a}, and Proposition \ref{capacityweighted}, and Remark \ref{Nomodel} below. Moreover, with constant $w(r) = \alpha$ for $\rho \leq r \leq R$ and a slight change of notation, we recover the linear potential function shown in figure \ref{fig:Cyl1}. \\

As before, the flow vector field $J$ for the positive charges in the material stemming from the potential solution $u$ in \eqref{eqCylPot} has a well-defined in-flux at $\pi^{-1}(\rho)$, which is the same as its out-flux at $\pi^{-1}(R)$. This common flux value is again called the \emph{capacity} of the cylinder $\pi^{-1}([\rho, R])$:
\begin{equation}
\C(\pi^{-1}([\rho, R])) = 2\pi\cdot c \cdot\frac{d}{dr}u(\pi^{-1}(r))  = \frac{2\pi c}{\int_{\rho}^{R}(1/w(t))\, dt} \quad .
    \end{equation}

Further, as before, the out-flux of $J$ at infinity, in the above case $\C(\pi^{-1}([\rho, \infty])$, tells us whether or not the surface allows a current to escape to infinity. \\If  $\C(\pi^{-1}([\rho, \infty]) = 0$ there is no current to infinity with the given boundary conditions -- the conductive surface is \emph{parabolic}. If  $\C(\pi^{-1}([\rho, \infty]) > 0$ there is a finite current to infinity  -- the surface is then \emph{hyperbolic}. As promised, we shall generalize these very simple settings below and define the type ($\W$-parabolic or $\W$-hyperbolic) of any given conductive Riemannian manifold $(M, g, W)$. \\

Let us note immediately, that in the simple $2$-dimensional examples discussed above with constant $\W$, the warped product surface $(M, g)$ is then parabolic if and only if $\int_{\rho}^{\infty}(1/w(t))\, dt = \infty$ and consequently hyperbolic otherwise, i.e. if and only if $\int_{\rho}^{\infty}(1/w(t))\, dt < \infty$.  For consistency, and for comparison with hyperbolic geometry, we must also note that already Milnor observed in \cite{Milnor1977} that $\pi^{-1}([\rho, \infty]$ is hyperbolic in the above sense if and only if it is hyperbolic in the potential-geometric sense, namely if the warped product cylinder is conformally equivalent to a $2$-disk (with an interior disk removed). This particular result is then an illustration of the equivalence between zero capacity of $M$ and the non-existence of non-constant bounded subharmonic functions on $M$ -- see \cite[Theorem 5.1]{Grigoryan1999BAMS}, Definition \ref{1:def:type} and Theorem \ref{1:thm:EquivType} below.\\

Moreover, since $w'(r) \leq 1$ for surfaces of revolution in 3-space, all such surfaces must be of parabolic type, when they are endowed with a constant conductivity. In particular, any plane (where $w(r) = r$) with constant conductivity is parabolic. 

\begin{figure}[h]
    \centering
    \includegraphics[width=70mm]{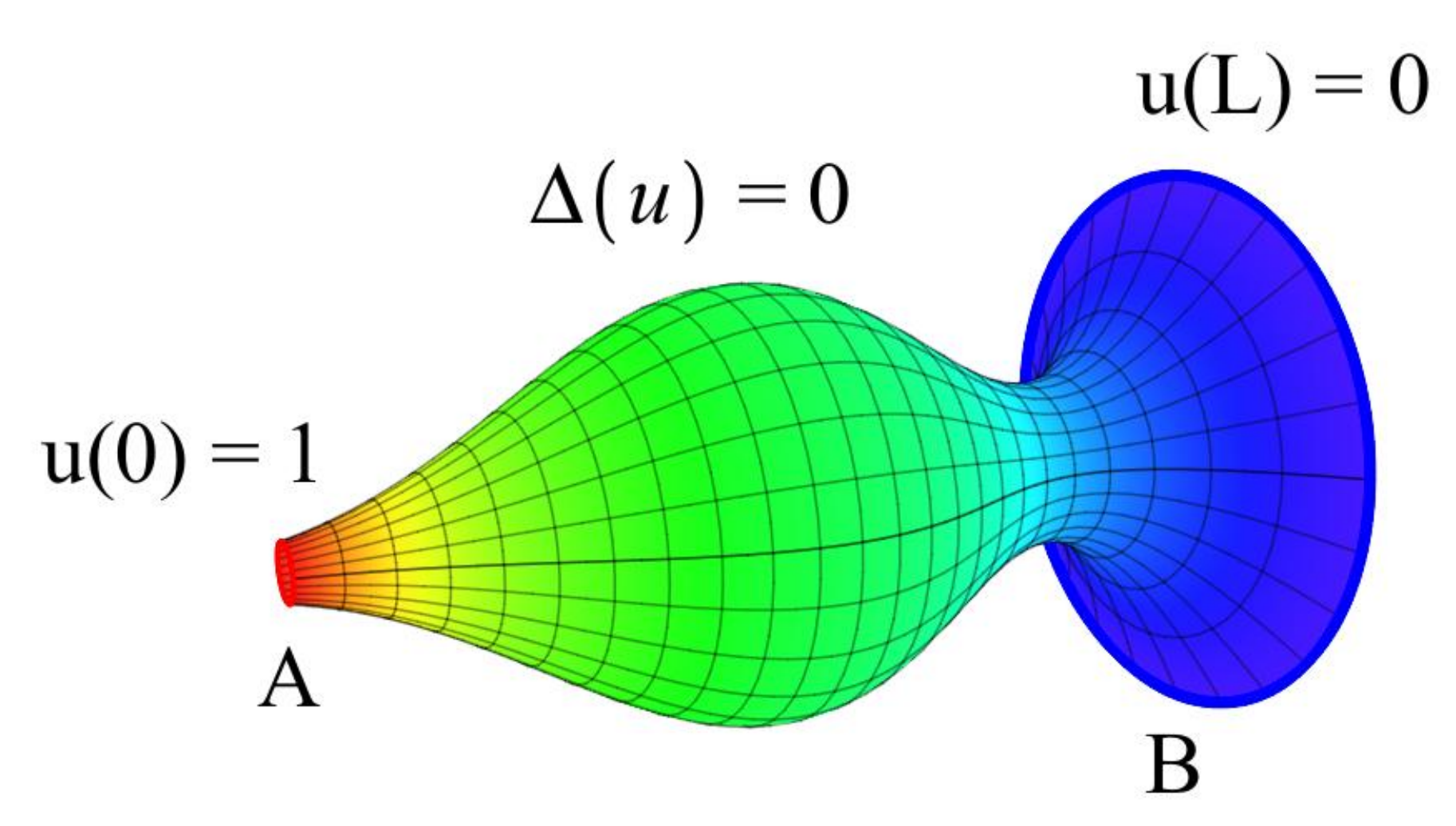}\caption{A finite surface of revolution (warped product with a warping function $w(r)$) endowed with a homogenous constant conductivity. It has a non-linear (colored) potential function $u(r)$ generated by $u(0)=1$ at $A$, $u(L) = 0$ at $B$ and $\Delta\,u =0$ in between. The general solution $u$ in such a setting is presented in equation \eqref{eqCylPot}. Note that the color bar for the values of $u$ is actually directly shown in the rightmost subfigure in figure \ref{fig:Cyl1}. }
    \label{fig:Cyl3}
\end{figure}

A more elaborate $2$-parameter family of examples -- with non-constant conductivities $\W$  modeled on Euclidean  $\mathbb{R}^{2}$ -- is contained in the following Proposition, which follows from our Theorem \ref{thm:Rp}: \\

\begin{proposition} \label{1:prop:2DVicent}
Let $(M^{2}, g, \W) = (\mathbb{R}^{2}, g_{\rm can}, \W_{\lambda,\alpha})$ denote a conductive Riemannian manifold where $g_{\rm can}$ denotes the Euclidean metric in $\mathbb{R}^{2}$, and the conductivity operator $\W_{\lambda,\alpha}$ is given by the following smooth field of tangent space endomorphisms at each point $(x,y) \in \mathbb{R}^{2}$:
\begin{equation}
\begin{aligned}
 \W_{\lambda,\alpha}(\partial x) &= e^{\alpha\cdot(x^2 + y^2)}\left((\lambda+1)\partial x + (\lambda-1)\partial y \right)  \\
 \W_{\lambda,\alpha}(\partial y) &=  e^{\alpha\cdot(x^2 + y^2)}\left( (\lambda-1)\partial x + (\lambda+1)\partial y \right)  \quad ,
 \end{aligned}
\end{equation}
with $\lambda>0$ and $\alpha \in \mathbb{R}$. Then, for any choice of $\lambda>0$, $(\mathbb{R}^{2}, g_{\rm can}, \W_{\lambda,\alpha})$ is $\W_{\lambda,\alpha}$-parabolic for $\alpha \leq 0$ and  $\W_{\lambda,\alpha}$-hyperbolic for $\alpha > 0$. 
\end{proposition}

With this background it should now be fairly clear that the capacity $\C_\W(M^{n}, g)$ of a  general conductive Riemannian manifold should depend not only on the geometry of the manifold, but also on the properties of the conductivity operator $\W$. For instance -- with a suitable geometry at infinity -- a 'small' $\W$ at infinity should be able to block the current from escaping to infinity and give $\W$-parabolic type of the conductive manifold, and vice versa, a 'large' $\W$ should be able to push the current to infinity and give $\W$-hyperbolic type.\\

In recent years a great many -- and quite different -- applications of conductive Laplacians have been studied -- not just in mathematical physics but also in differential geometry and in global geometric analysis. 
We mention here only a few previous works in these directions in order to indicate further the spectrum of applications --  some of which are directly related to the present work.\\

In the vein of the physical interpretation of $\W$, the $\W$-Laplace equation $\Delta_{\W}\, u = 0$ together with suitable boundary conditions is also fundamental for the \emph{Riemannian Calder\'{o}n problem}, which is concerned with the difficult inverse task of determining the conductivity $\W$ from suitable knowledge about the solutions $u$ on the boundary of the domains in question, see \cite{sylvester1987a, lee1989a}.\\

To be precise, suppose that $\Omega$ is a compact domain in $(M, g, \W)$ with smooth hypersurface boundary $\partial \Omega$ and a smooth outward pointing normal field $n_{\partial \Omega}$. Suppose that $\W$ is unknown, but that $b_{i} \in C^{\infty}(\Omega)$, $i = 1, \cdots, k \leq \infty$, are a number of known smooth functions on the boundary of $\Omega$ for each one of which  $u_{i}$ is the corresponding unique solution to the $\W$-Laplace equation in $\Omega$ such that $u_{i} |_{\partial \Omega} = b_{i}$. The Riemannian Calder\'{o}n problem is then: Suppose that you also know (measure) the flux integrand $g(\W(\nabla u_{i}), n_{\partial\Omega})$ for each $i$ along the boundary $\partial \Omega$; Is it then possible to construct the conductivity $\W$ in all of $\Omega$? \\

We must note, however, that this Riemannian Calder\'{o}n problem is not directly equivalent to the so-called \emph{geometric Calder\'{o}n problem}, which is the one that is mainly studied -- successfully, and with many applications -- in the inverse problems literature, see \cite{lee1989a} and \cite{uhlmann2009a}. The latter geometric problem is concerned with the reconstruction of a metric $G$, which is metrically equivalent to $\W$ while using the same setting and boundary conditions on $\Omega$ as above, but now with a usual $G$-Laplacian $\Delta_{G}$ for $(M, G)$ in place of $\Delta_{\W}$ for $(M, g, \W)$. For example, in dimensions $n \geq 3$  equivalence of the two problems is obtained if $G = \det(\W)^{\frac{1}{n-2}} \cdot g\cdot \W^{-1}$, see Section \ref{subsec:EquivMetrics}; in particular it follows that $\Delta_{G}\, u = 0$ is then equivalent to $\Delta_{\W} \,u = 0$. Thus, in order to reconstruct $\W$ via solving the \emph{geometric} inverse problem for $G$ we also need information about $\det(\W)$ in the domain $\Omega$ as well as information about $\W$ itself on the boundary $\partial \Omega$.\\

 Following the classical geometric analysis of the ordinary Laplacian $\Delta$ on $(M, g)$, a set of natural questions for $\Delta_{\W}$ is concerned with estimations of its Dirichlet spectrum and its corresponding mean exit time moment spectrum on compact domains in $(M, g, \W)$ -- and, moreover, not least, with the intriguing interdependence between the two spectra, see \cite{dryden2017a, McDonald2002, bessa2006a, gomes2018a, Filho2022}. Even earlier, in \cite{Cheng1977}  Cheng and  Yau take $\W$ to be the (metric equivalent of) the second fundamental form of a non-compact hypersurface with constant scalar curvature (and non-negative sectional curvature) in Euclidean space and use a corresponding second order differential operator to show, among other results, a splitting result to the effect, that such a hypersurface is a generalized cylinder over a sphere.\\ 

In the present paper we consider a related setting for obtaining new results concerning the geometric interplay between $\W$, $g$, and $\Delta_{\W}$. Specifically, we will be concerned with the aforementioned \emph{type problem} for any given conductive Riemannian manifold $(M, g, \W)$: Is it  $\W$-hyperbolic, or is it $\W$-parabolic? \\

In passing we must note already, that we obtain new results for the type problem when the conductivity is already determined by the curvatures of the manifold itself -- as in the vein of \cite{Cheng1977} -- e.g. the Schouten tensor in Section \ref{sec:Schouten}, the Einstein tensor in Section \ref{sec:Einstein}, and the first Newton transform (defined by the mean curvature) in Section \ref{sec:Newton}.\\

As promised above we now give a precise definition of type, which complies with the previous examples of types of conductive $2$D surfaces: 

\begin{definition}\label{1:def:type}
A conductive Riemannian manifold $(M, g, \W)$ is $\W$-parabolic if every $\W$-subharmonic function that is bounded from above must be constant. Otherwise, we say that $M$ is $\W$-hyperbolic.  \\
\end{definition}

However, as already illustrated by the $2$D cases, a useful and equivalent decision criterion for the type problem is obtained as follows:  
Consider special annular domains in $M$, i.e. compact domains $\Omega = \Lambda_{\rho, R}$ with two disjoint boundary components $\partial \Omega_{\rho}$ and $\partial \Omega_{R}$ and potential functions $u$ on $\Lambda_{\rho, R}$ which satisfy $\Delta_{\W}(u) = 0$ in $\Lambda_{\rho, R}$ and having $u_{\partial \Omega_{\rho}} = 1$, $u_{\partial \Omega_{R}} = 0$ on the respective boundaries. It is well known that $u$ is then the minimizer of the total $\W$ energy in $\Omega$ among all functions satisfying the boundary conditions. And that this energy can be read off as the total flux of $\W(\nabla u)$ through any one of the two boundary components.
This energy, this flux, is the $\W$-capacity of $\Omega$ and is then still denoted $\C_{\W}(\Omega) = \C_{\W}(\Lambda_{\rho, R})$.\\

We may specialize the domains $\Omega$ even further and assume, that $\Omega$ is the difference between two topological balls in the manifold $M$, i.e. $\Omega = \Lambda_{\rho, R} = B_{R} - \bar{B_{\rho}}$  -- possibly (geodesic) metric balls with the indicated radii and with the same center point, as will become relevant and employed below.\\

This then gives rise to the following equivalent capacity definition of the type, (see Subsection \ref{subsec:WLap} and Section \ref{sec:CapCondit}):

\begin{theorem} \label{1:thm:EquivType} When $B_{R_{i}}$ is an exhaustion of a non-compact complete conductive manifold $(M, g, \W)$ -- obtained e.g. by letting $R_{i} \to \infty$ for $i \to \infty$ --  then one of two consequences occurs: 
Either $\C_{\W}(\Lambda_{\rho, \infty}) = 0$ in which case $(M, g, \W)$ is $\W$-\emph{parabolic}, or $\C_{\W}(\Lambda_{\rho, \infty}) > 0$ in which case $(M, g, \W)$ is $\W$-\emph{hyperbolic}.
\end{theorem}

These concepts and notions have previously been well studied for Riemannian manifolds and submanifolds (in these cases based on the classical Laplacian) as well as for weighted Riemannian manifolds and submanifolds (based on corresponding gradient-drifted Laplacians and weighted volume forms). We refer to \cite {Grigoryan1999BAMS}, \cite{hurtado2020a, hpr2020a} and \cite{McDonald2002, dryden2017a} for classical as well as recent results in this direction.  \\

The initial step for obtaining our new results for the type problem in the more general conductive setting is to show $\W$-Hessian (and thence $\W$-Laplacian) inequalities for modified distance functions $F(r)$, where $r = \dist_{g}(o, \cdot)$ in $(M, g)$. The specific modifier functions $F$ are obtained by transplantation from suitable comparison conductive model spaces whose radial sectional curvatures bound the corresponding curvatures in $(M, g)$. The maximum principle for $\W$-harmonic functions then eventually gives the type-conclusive bounds on the $\W$-capacity of a given $(M, g, \W)$ in terms of an explicit conductive capacity of the corresponding model space. \\

It is but natural to consider also the type problem for submanifolds as well, i.e. for properly immersed manifolds $\Sigma^{m}$ in a given conductive Riemannian manifold $(M^{n}, g, \W)$.
For a submanifold $\Sigma$ in $(M, g, \W)$ the metric, the connection, the $\W$-Laplacian, and the conductivity can all be inherited from the ambient manifold. Thus $(\Sigma, g_{\Sigma}, \W^{\Sigma})$ becomes a conductive Riemannian manifold in its own right, while at the same time supporting the fruitful interplay -- for example via the second fundamental form -- with the geometry of the ambient space. The ensuing external viewpoint has been a 
key strategy in obtaining several estimates and comparison results for the capacity, and thence also for finding the type, parabolicity/hyperbolicity, of weighted submanifolds, i.e for isotropic conductivities, $\W = e^f\cdot{\rm Id}$, see \cite{hurtado2020a, hpr2020a}. \\

If we are given a conductivity field $\W^{\Sigma}$ on $(\Sigma, g_{\Sigma})$ -- which is a physically reasonable starting point -- then the exponential map from $\Sigma$ provides smooth extensions $\W$ of  $\W_{\Sigma}$ into a tubular neighborhood of $\Sigma$ in $M$ so that $\W^{\Sigma}$ becomes the restriction of $\W$ in the following local sense at each point $p \in \Sigma$: $\W(V) = \W^{\Sigma}(V)$ for all $V \in T_{p}\Sigma$ and $\W(U) \in (T_{p}\Sigma)^\perp$ for all $U \in (T_{p}\Sigma)^\perp$. Note that we cannot in general just restrict $\W$ pointwise to $\Sigma$ and let $\W^{\Sigma}(p) = \W(p)$ for $p \in \Sigma$, because $\W_{\Sigma}$ must physically 'drive' the gradient vector field of any potential function $u \in C^{\infty}(\Sigma)$, so that the corresponding current vector field stays in the tangent bundle of $\Sigma$.\\

In Example \ref{examp:Paraboloid} we consider a classical paraboloid of revolution $(\Sigma, h= g_{|\Sigma})$ in $(\mathbb{R}^{3}, g)$. It is well-known that $(\Sigma, h, c\cdot Id)$ is $c\cdot Id$-parabolic, but we show that $\Sigma$ can also be endowed with an other conductivity $\W^{\Sigma}$ inherited from the ambient space $(\mathbb{R}^{3}, g, \W)$ so that $(\Sigma, h, \W^{\Sigma})$ becomes $\W^{\Sigma}$-hyperbolic. \\

Below we now present an outline of the paper, including brief descriptions of some of the further novelties in, and highlights from, this work.

\subsection{Outline and highlights}
{\bf{Section \ref{sec:Prelim}}} is concerned with setting up the preliminaries for the paper. This includes in particular the precise definitions and expressions related to the conductivity tensor field $\W$, its divergence, and its influence in the definitions of the $\W$-Hessian and the $\W$-Laplacian, as well as its influence on the ensuing notion of $\W$-capacity. 

Our working conditions for the type problem is expressed in terms of the $\W$-capacities in {\bf{Section \ref{sec:CapCondit}}}. Model spaces and weighted model spaces are introduced as our fundamental comparison background tools, and the expression for their (comparison) weighted capacities are reviewed in Proposition \ref{capacityweighted}.

In {\bf{Section \ref{sec:BoundingLapDist}}} we then set up our main general background assumption on the manifold $(M, g)$, which will be instrumental for our comparisons with the (weighted) model spaces. As mentioned above, this hinges in particular on the existence of a pole in $(M, g)$ and on a bound (upper or lower) on the radial sectional curvatures in the manifold -- as seen from the pole -- in terms of the corresponding curvatures of the chosen model space. In effect, this gives first a well-known (lower or upper) bound on the classical Hessian of the distance function from the pole and then -- as stated and proved in Proposition \ref{ineq1intrinsic} -- a new corresponding bound (lower or upper) of the $\W$-Laplacian of a suitably modified distance function on the conductive Riemannian manifold $(M, g, \W)$. These bounds depend naturally on $\W$ itself, the trace and the divergence of $\W$, as well as on the warping function $w$ and the choice of distance modifier function. 

The intermezzo in {\bf{Section \ref{sec:Intermezzo}}} addresses the fact that the 'pole condition' on the manifolds can be avoided, so that the main conditions for the $\W$-capacity estimates -- and therefore the results of the paper -- can instead be expressed directly in terms of inequalities for the Hessian of a suitable distance-like exhaustion function on the manifold.

In {\bf{Section \ref{sec:IntrinsicCrit}}} we show our main comparison result, Theorem \ref{thm:Rp}, which states upper and lower bounds on $\W$-capacities in terms of bounds on the radial sectional curvatures from the pole of $M$ together with inequalities involving $\W$, the trace of $\W$, and the divergence of $\W$. Naturally, these inequalities are then used to express the corresponding consequences for the type problem. In the same setting we show how the pointwise maximal and minimal eigenvalues of $\W$ can be applied as shortcuts to produce similar results. A similar shortcut is obtained from the assumption of $0$ divergence of $\W$ everywhere. Moreover, in this section we also give a proof of Proposition \ref{1:prop:2DVicent} that was already presented in Section \ref{sec:Glimpses}.

In {\bf{Section \ref{sec:ExtrinsicCrit}}} we then consider submanifolds $\Sigma^{m}$ in a conductive Riemannian manifold $(M^{n}, g, \W)$,  and define the notion of $\W$-compatible submanifolds (with inherited conductivity from $M$) and a corresponding notion of $\W$-mean curvature. We then show capacity inequalities and the ensuing type-statements for $(\Sigma, h= g_{\Sigma}, \W^{\Sigma})$. Much in the same way as for the intrinsic setting, but now with the modified ambient distance functions restricted to $\Sigma$ and with an intricate application of what we call the $\W$-mean curvature of $\Sigma$. Together these ingredients then give useful $\W^{\Sigma}$-Laplace inequalities for the distance modifiers as expressed in Proposition \ref{prop:dua}. 

In the following {\bf{Section \ref{sec:BV}}} we show a generalization of the well-known fact that every minimal submanifold of an $n$-dimensional Cartan-Hadamard space is hyperbolic when $n > 2$. In fact, with a 'radial' $\W$-mean curvature vector field, $(\Sigma, h= g_{\Sigma}, \W^{\Sigma})$ is shown to be $W^{\Sigma}$-hyperbolic in such a conductive Cartan-Hadamard manifold $(M^{n}, g, \W)$ if only the induced capacity $\W^{\Sigma}$ has bounded variation, i.e. if it has a specific upper bound (in terms of $m$) for its coefficient of variation. Moreover, for the intrinsic equivalent of this result, the corresponding $n$-dependent upper bound is shown to be sufficient to give $\W$-hyperbolicity of the $\W$-conductive Cartan-Hadamard manifold if just the divergence of $\W$ vanishes everywhere.

In the two last sections, {\bf{Section \ref{sec:Applications}}} and {\bf{Section \ref{sec:GeometricConduc}}} we present a number of new examples and corollaries that show how our findings can be used to establish the type of conductive Riemannian manifolds and submanifolds in some more special settings than those considered previously. We obtain new criteria for the type of submanifolds with inherited isotropic conductivities from their ambient spaces. We use the notion of $\W$-induced equivalent metrics to construct a specific $6$-dimensional $\W$-hyperbolic manifold. And we show how the Schouten tensor, the Einstein tensor, the shape operator of a hypersurface, and the stress-energy tensor of a compressible gas, respectively, can decide the type of a conductive Riemannian manifold when the conductivity is given directly by one of these tensor fields.  Finally, in {\bf{Subsection \ref{subsec:ExtrinExamps}}}, we present concrete examples of $\W^{\Sigma}$-hyperbolic surfaces, cylinders and paraboloids, in ambient Euclidean space $\mathbb{R}^{3}$ endowed with specific ambient conductivities $\W$. These 'extrinsic' examples thus vividly complement the 'intrinsic' examples (mentioned in Proposition \ref{1:prop:2DVicent} in Section \ref{sec:Glimpses}): An inherited conductivity can also change the type of a surface from parabolic to hyperbolic.

\section{Preliminaries} \label{sec:Prelim}

We let $(M^{n}, g=\langle \cdot,\cdot\rangle, \nabla)$ denote a complete Riemannian manifold with metric $g$ and induced Levi-Civita connection $\nabla$. We equip $M$ with a tensor field $\W$ of type $(1,1)$ with two important features, which define a conductivity tensor and thence a conductive Riemannian manifold as already alluded to in the Introduction:

\begin{definition} \label{defConducRiem}
Suppose that {$\W$ is a smooth $(1,1)$-tensor field positive definite and selfadjoint}  with respect to $g$, i.e. at all points $p \in M$ and for all tangent vectors $X$ and $Y$ in $T_{p}M$:
\begin{equation} \label{eqWconditions}
\begin{aligned}
\langle\W(X), Y\rangle &= \langle X, \W(Y)\rangle \quad \textrm{and} \\
\mu(p)\cdot \Vert X \Vert^{2} &\leq  \langle\W(X), X\rangle \leq \kappa(p)\cdot \Vert X\Vert^{2}
\end{aligned}
\end{equation}
for some positive functions  $\kappa$ and $\mu$ in {$C^\infty(M)$}.\\

In particular, the $n$ positive eigenvalues $\lambda_{i}(p)$ of $\W$ at $p$ satisfy these bounds:
\begin{equation} \label{eq:Wspectrum}
0 < \mu(p) \leq \lambda_{1}(p) \leq \lambda_{2}(p) \leq \cdots \leq \lambda_{n}(p) \leq \kappa(p) \, \, \,\, \textrm{at all points $p \in M$}.
\end{equation}

Then the triple $(M^{n}, g,\W)$ is called a \emph{conductive Riemannian manifold} with \emph{conductivity} $\W$. If $\W$ is everywhere proportional to the identity endomorphism, $\W = f\cdot \Id$, for some smooth positive function $f \in C^{\infty}(M)$, then the conductivity is said to be \emph{isotropic} (or \emph{homogeneous} if $f$ is constant); otherwise, it is \emph{anisotropic}.
\end{definition}

In an arbitrary coordinate frame field $\{e_{i} = \partial / \partial x^{i}\}$ with dual one-form basis $\{\theta^{j} \}$, the corresponding coordinates of $\W$ are
\begin{equation}
\W = \sum_{i\, j}\W^{i}_{j}\cdot (e_{i} \otimes \theta^{j}).
\end{equation}

The divergence of the $(1,1)$ tensor field $\W$  is the vector field
\begin{equation} \label{eqDivergence}
\begin{aligned}
\Div(\W) &= \sum_{i\, j}\left( \nabla_{e_{i}}\left(\W(e_{j}) \right) - \W\left(\nabla_{e_{i}} e_{j}\right) \right) \cdot g^{i\, j} \\
&= \sum_{i\,j\,k}\W_{i\,; \,j}^{k}\cdot g^{i\, j}\cdot e_{k},
\end{aligned}
\end{equation}
where we have used the standard $\{e_{i}\}$-indexed coordinates for the covariant derivatives of $\W$ as follows:
\begin{equation} \label{eqDivergenceCoord}
\W_{i \,; \, j}^{k} = \left(\frac{\partial}{\partial x^{j}}\W_{i}^{k}\right) + \sum_{\ell} \Gamma_{\ell \, j}^{k} \cdot \W_{i}^{\ell} - \sum_{\ell} \Gamma_{i \, j}^{\ell}\cdot \W_{\ell}^{k}.
\end{equation}
Indeed, following \cite{Lee2018} the divergence is obtained as the trace on the last two indices of the total covariant derivative of $\W$. To be even more precise:
For any one-form $\omega$ and vector fields $X$ and $Y$:
\begin{equation}
\begin{aligned}
\left(\nabla \W \right)\left(\omega, Y, X \right) &= \left(\nabla_{X} \W \right)\left(\omega, Y\right) \\
&= X\left(\W(\omega, Y) \right) - \W\left(\nabla_{X}\omega, Y \right) - \W(\omega, \nabla_{X}Y),
\end{aligned}
\end{equation}
with coordinates:
\begin{equation}
\left(\nabla \W \right)\left(\theta^{k}, e_{i}, e_{j} \right) = \W_{i \,; \, j}^{k}.
\end{equation}

Here, the last two indices are $i$ and $j$, so the trace is exactly as in \eqref{eqDivergence}.\\

\subsection{The $\W$-Laplacian and $\W$-capacities} \label{subsec:WLap}

The conductivity defines in a natural way the following divergence type Laplacian, see, e.g.
\cite{gomes2018a}.

\begin{definition}
The $\W$-Laplacian is defined by its action on functions {$u \in C^\infty(M)$} as follows:
\begin{equation}
\Delta_{\W}\,u = \Div (\W(\nabla u)).
\end{equation}
\end{definition}

The $\W$-Laplacian of a function {$u \in C^\infty(M)$} can be written in terms of the divergence of the conductivity $\W$ and a modified Hessian operator as follows:

\begin{proposition} \label{prop:LapWexpression} Let $(M^n,g,\W)$ a conductive Riemannian manifold, and {$u \in C^\infty(M)$}. Then,
\begin{equation}
\Delta_{\W}\,u = {\rm tr}(\Hess_{\W}\,u) + \left<\nabla u\, , \,\, \Div(\W) \right>,
\end{equation}
where the $\W$-Hessian of $u$ is then defined as:
\begin{equation}
\left(\Hess_{\W}u\right)(X, Y) = \left< \nabla_{X}\nabla u\, , \,\, \W(Y) \right>.
\end{equation}
\end{proposition}

\begin{proof}
{Let us consider $\{E_{i}\}_{i=1}^n$ a   local orthonormal frame defined on an open set $U \subset M$. Then, at any point $p \in U$},

\begin{equation*} \label{eqExpand}
\begin{aligned}
\Delta_{\W}\,u =& \Div(\W(\nabla u))= \sum_{i=1}^n\left<\nabla_{E_{i}}\left(\W(\nabla u) \right)\, , \,\, E_{i}  \right> \\
=& \sum_{i=1}^n E_{i}\left<\W(\nabla u)\, , \,\, E_{i}  \right> -  \sum_{i=1}^n\left<\W(\nabla u)\, , \,\,\nabla_{E_{i}} E_{i}  \right> \\
=& \sum_{i=1}^n E_{i}\left<\nabla u\, , \,\, \W(E_{i})  \right> -  \sum_{i=1}^n\left<\nabla u\, , \,\,\W\left(\nabla_{E_{i}} E_{i}\right)  \right> \\
=& \sum_{i=1}^n\left<\nabla_{E_{i}}\nabla u\, , \,\, \W(E_{i}) \right> +
\left<\nabla u\, , \,\, \sum_{i=1}^n\nabla_{E_{i}}(\W(E_{i})) \right> \\ &-  \left<\nabla u\, , \,\,\sum_{i=1}^n\W\left(\nabla_{E_{i}} E_{i}\right)  \right> \\
=& {\rm tr}(\Hess_{\W}\,u) + \left<\nabla u\, , \,\, \Div(\W) \right>.
\end{aligned}
\end{equation*}
\end{proof}

\begin{definition} \label{def:WeightedMnf}
A weighted Riemannian manifold $(M, g, e^{f})$ is a Riemannian manifold $(M,g)$ with a smooth positive function $e^{f}$, that is used to endow $M$ with a measure that has the smooth positive density $e^{f}dV$ with respect to the canonical Riemannian volume element, see \cite{Grigoryan2009Heat}, \cite{hurtado2020a}. We define the weighted Laplacian or $f$-Laplacian of a smooth function $u$  by 
\begin{equation}
\Delta^{f}\,u = \Delta\,u + \left< \nabla u\,, \,\, \nabla f\right>.
\end{equation}

\end{definition}

The weighted Laplacian and the conductive $\W$-Laplacian are directly comparable for isotropic conductivities $\W = e^{f}\cdot \Id$. In this particular case, we have

\begin{corollary} \label{cor:LapWfoisotropic}
When $\W$ is isotropic, i.e. $\W = e^{f}\cdot \Id$, then the expressions for the $\W$-Hessian and the $\W$-Laplacian reduce to:
\begin{equation}
\begin{aligned}
\Hess_{\W}\,u (X, Y) &= \left< \nabla_{X}\nabla u\, , \,\, e^{f}\cdot Y \right> \\
&= e^{f}\cdot \Hess \,u (X,Y),
\end{aligned}
\end{equation}
and 
\begin{equation}
\Delta_{\W}\,u = {\rm tr}(\Hess_{\W}\,u) + \left<\nabla u\, , \,\, \nabla (e^{f}) \right> = e^f\, \Delta^f\, u.
\end{equation}

\end{corollary}

Given a domain (open connected set) $\Omega$ in $M$, we will say that a smooth function $u$ is $\mathcal{W}$-harmonic (resp. $\W$-subharmonic) if $\Delta_{\W} u=0$ (resp. $\Delta_{\W} u\geq 0$) on $\Omega$. In particular, by the above corollary, a function $u$ is $\W$-harmonic (resp. $\W$-subharmonic) for an isotropic conductivity $\W=e^f\cdot \Id$ if and only if it is $f$-harmonic (resp. $f$-subharmonic) in the weighted manifold $(M,g,e^f)$.\\

The $\W$-Laplacian is an elliptic operator since $\W$ is strictly positive definite. In particular, we can apply the maximum principle for $\Delta_{\W}$, see \cite{jost2013a}:

\begin{theorem}\label{Hopf}
Let $\Omega$ be a smooth domain in $M$ and let $u$ denote a $\W$-subharmonic function in $\widebar{\Omega}$. If $u$ achieves its maximum in $\Omega$, then $u$ is constant. As a consequence, if $v$ is $\W$-harmonic with $v(q)= u(q)$ for all $q \in \partial \Omega$, then
\begin{equation}
\begin{aligned}
v &\geq u \quad \textrm{everywhere in $\Omega$, and } \\
\left<\W(\nabla v), n_{\partial \Omega}\right> &\leq \left<\W(\nabla u), n_{\partial \Omega}\right>  \quad  \textrm{everywhere along $\partial \Omega$},
\end{aligned}
\end{equation}
where $n_{\partial \Omega}$ denotes the outwards pointing unit vector field along  $\partial \Omega$.
\end{theorem}

From the maximum principle it is clear that any $\W$-subharmonic function on a compact manifold $M$ must be constant. In general, as already mentioned in Section \ref{sec:Glimpses} we have:
{
\begin{definition} \label{defTypeRep}
A conductive Riemannian manifold $M$ is $\W$-parabolic if any $\W$-subharmonic function bounded from above must be constant. Otherwise, we say that $M$ is $\W$-hyperbolic.    
\end{definition}
\begin{remark} Notice that the problem of determining whether a conductive Riemannian manifold $M$ is $\W$-parabolic or $\W$-hyperbolic for an isotropic conductivity $\W=e^f\cdot \Id$ is the same as determining whether $M$ is weighted parabolic or weighted hyperbolic for the density $e^f$.
\end{remark}
}
As in the classical setting, the $\W$-parabolicity of manifolds can be characterized by means of $\W$-capacities (see \cite{Littman}). For more details on the classical versions of the definitions and results below, we refer to \cite{Grigoryan1999BAMS,grigoryan2013a} for homogeneous conductivities $\W = Id$   and to \cite{hpr2020a, hurtado2020a}  for weighted ``conductivities'' $\W = e^{f}Id$. 

Let $(M,g,\W)$ be a conductive Riemannian manifold. Given an open set $\Omega\subset M$ and a compact set $K\subset\Omega$, the $\W$-capacity of the capacitor $(K,\Omega)$ is the non-negative number given by
\begin{equation}
\label{eq:capacity}
\WCap(K,\Omega):=\inf\left\{\int_\Omega\left<\W(\nabla \phi), \nabla \phi \right>\,dV\,;\,\phi\in \mathcal{L}(K, \Omega) \right\},
\end{equation}
when $\mathcal{L}(K,\Omega)$ is the set of functions on $M$ with compact support on $\widebar{\Omega}$ such that $0\leq\phi\leq 1$ and $\phi=1$ on $K$. For a precompact open set $D$ with $\widebar{D}\subset \Omega$, we denote $\WCap(D,\Omega):= \WCap(\widebar{D},\Omega)$. \\

When $\Omega$ is a smooth precompact open set and $K$ has smooth boundary, it is known that the infimum in \eqref{eq:capacity} is attained by the unique solution to the following Dirichlet problem for the  $\W$-Laplace equation
\begin{equation}
\label{eqWLap}
\begin{cases}
\Delta_\W\,u=0\,\,\,&\text{in\, $\Omega - K$},\\
\phantom{\Delta_\W }u=1\,\,\,&\text{in\, $\partial K$}, \\
\phantom{\Delta_\W }u=0\,\,\,&\text{in\, $\partial \Omega$},
\end{cases}
\end{equation}
and it can be proved (see \cite{Littman}) that

\begin{equation}
\WCap(K,\Omega) = \int_{\Omega- K}\left<\W(\nabla u), \nabla u \right>\,dV=\int_{\partial K}  \left<\W(\nabla u)\, ,\,\, n_{\partial K}\right>  \, d\mu_{\partial K},
\end{equation}
where $u$ is the solution to \eqref{eqWLap} and $n_{\partial K}$ is the outer unit normal along $\partial(\Omega- K)$, i.e., the unit normal along $\partial K$ pointing into $K$.\\

The $\W$-capacity of the whole manifold $M$ can be obtained  by considering any exhaustion
sequence $\{\Omega_i\}_{i=1}^\infty$ covering $M$ from precompact open sets such that $K\subset \Omega_i \subset \widebar{\Omega}_i \subset \Omega_{i+1}$ for all $i\in \mathbb{N}$. Then,
\begin{equation*}
\WCap(K)=\lim_{i\to\infty}\WCap (K,\Omega_i).
\end{equation*}

\begin{remark} For an isotropic conductivity $\W= e^f\cdot \Id$, the $\W$-harmonic function $u$ the solution to \eqref{eqWLap} in $\Omega- K$ is the same as the solution to the corresponding problem for the weighted Laplacian $\Delta^f$ in the weighted manifold $(M,g, e^f)$ (since $\Delta_\W\,u= e^f\, \Delta^f\, u$). As a consequence we therefore have
\begin{equation*}
\C_{\W}(K,\Omega)=\C^f (K,\Omega)=\int_{\partial K} \Vert\nabla u\Vert  \,e^f\, d\mu_{\partial _{K}}. \\
\end{equation*}
\end{remark}

\subsection{Comparison with the classical capacity}\label{sec:com:cap}
%

{Let $u$ be the solution to the classical Laplace problem, namely when we consider  $\W=Id$ in \eqref{eqWLap}, and $\phi$ the solution to \eqref{eqWLap} for a general $\W$}. Then
$$
\begin{aligned}
{\rm Cap}_\mathcal{W}(K,\Omega)\leq&\int_{\Omega- K}\langle \mathcal{W}(\nabla
u),\nabla u\rangle dV\leq \int_{\Omega- K}\lambda^*(p)\Vert\nabla u\Vert^2 dV\\
\leq& \max_{p\in \overline{\Omega}}\lambda^*(p){\rm Cap}(K,\Omega),
\end{aligned}
$$
where $\lambda^* (p)$ is the highest eigenvalue of $\mathcal{W}$ at $p$. Similarly, 

$$
\begin{aligned}
{\rm Cap}_\mathcal{W}(K,\Omega)=&\int_{\Omega- K}\langle \mathcal{W}(\nabla 
\phi),\nabla \phi\rangle dV\geq \int_{\Omega- K}\lambda_*(p)\Vert\nabla \phi\Vert^2 dV\\
\geq& \min_{p\in \overline{\Omega}}\lambda_*(p){\rm Cap}(K,\Omega),
\end{aligned}
$$
where here $\lambda_* (p)$ is the lowest eigenvalue of $\mathcal{W}$ at $p$ and ${\rm Cap}(K,\Omega)$ is the classical Riemannian capacity, defined as in \eqref{eq:capacity} with $\W={\rm Id}$. Thence,
$$
\min_{p\in \overline{\Omega}}\lambda_*(p)\leq \frac{{\rm Cap}_\mathcal{W}(K,\Omega)}{{\rm Cap}(K,\Omega)}\leq \max_{p\in \overline{\Omega}}\lambda^*(p).
$$

{
\begin{remark}\label{boundedeigen}
    Observe that by the above inequality, and taking $\Omega=M$, if $\W$ has bounded eigenvalues, then the conductive Riemannian manifold $(M,g, \W)$ is $\W$-parabolic if and only if $(M,g)$ is parabolic in the Riemannian sense.
\end{remark}
}

Since
$$
\begin{aligned}
&\liminf_{R\to 0}\min_{p\in B_R(o)}\lambda_*(p)=\limsup_{R\to 0}\min_{p\in B_R(o)}\lambda_*(p)=\lim_{R\to 0}\min_{p\in B_R(o)}\lambda_*(p)=\lambda_*(o),\\
&\liminf_{R\to 0}\min_{p\in B_R(o)}\lambda^*(p)=\limsup_{R\to 0}\min_{p\in B_R(o)}\lambda^*(p)=\lim_{R\to 0}\min_{p\in B_R(o)}\lambda^*(p)=\lambda^*(o),
\end{aligned}
$$
where $B_t (o)$ denotes the open metric ball of radius $t$ and center at the point $o\in M$, we can state the following 
\begin{proposition}
Let $(M,g,\mathcal{W})$ be a conductive Riemannian manifold. Let $o$ be any point in $M$. Then
$$
\begin{aligned}
\lambda_*(o)\leq \liminf_{R\to 0}\frac{{\rm Cap}_\mathcal{W}\left(B_{\frac{R}{2}}(o),B_R(o)\right)}{{\rm Cap}\left(B_{\frac{R}{2}}(o),B_R(o)\right)}\leq\lambda^*(o) ,
\end{aligned}
$$
and
$$
\begin{aligned}
\lambda_*(o)\leq \limsup_{R\to 0}\frac{{\rm Cap}_\mathcal{W}\left(B_{\frac{R}{2}}(o),B_R(o)\right)}{{\rm Cap}\left(B_{\frac{R}{2}}(o),B_R(o)\right)}\leq\lambda^*(o) 
\end{aligned}.
$$

\end{proposition}


\section{The capacity conditions for the $\W$ type problem} \label{sec:CapCondit}
The relation between $\W$-capacities and the $\W$-parabolicity (resp. $\W$-hyperboli\-city) of a conductive Riemannian manifold is shown in the next result. See also \cite{Gri3} and \cite{Grigoryan1999BAMS} for the classical background.

\begin{theorem}
\label{theorGrig}
Let $(M,g,\W)$ be a conductive Riemannian manifold. Then, $M$ is
$\W$-parabolic $($resp. $\W$-hyperbolic$)$ if and only if $M$ has null $($resp. positive$)$ $\W$-capacity, i.e., there exists a precompact open set $D\subseteq M$ such that $\WCap(D)=0$ $($resp. $\WCap(D) >0$$)$.
\end{theorem}

In view of Theorem \ref{theorGrig}, in order to determine if a conductive Riemannian manifold is $\W$-parabolic or $\W$-hyperbolic it therefore suffices to find bounds on $\WCap(D,\Omega_i)$ for some precompact open set $D \subset M$ and some exhaustion $\{\Omega_i\}_{i=1}^\infty$. We establish such bounds via comparison with suitable model manifolds with radial isotropic conductivity. 
\subsection{Model manifolds}

Here we introduce the model spaces that we will use to establish our comparison
theorems. They will be weighted warped products with radial weights $e^{h(r)}$, where $r$ denotes the distance to the well defined center of the model. We then utilize the fact that the $\W$-Laplacian with $\W$ isotropic is up to a positive function the weighted Laplacian. 

\begin{definition}(see {\cite[Ch.~2]{GreeneWu}, \cite[Sect.~3]{Grigoryan1999BAMS}, \cite[Ch.~3]{Petersen2016}}).
\label{model}
A model space is a smooth warped product $(M^n_w,g_w):=B^1\times_w F^{n-1}$ with base $B^{1}:=[0,\Lambda[\,\subset\mathbb{R}$ (where $0 < \Lambda \leq \infty$), fiber $F^{n-1}:=\mathbb{S}^{n-1}_{1}$ (the unit $(n-1)$-sphere with standard metric), and warping function $w:[0,\Lambda[\,\,\to [0,\infty[$ such that $w(r)>0$ for all $0<r\leq \Lambda$, whereas $w(0) = 0$, $w'(0) = 1$, and $w^{(k)}(0) = 0$ for all even derivation orders. The point $o_w:=\pi^{-1}(0)$, where $\pi$ denotes the projection onto $B^1$, is called the {\em{center point}} of the
model space. If $\Lambda = \infty$, then $o_w$ is a pole of the manifold.
\end{definition}

\begin{remark}
The strong analytical conditions for $w$ at $0$ ensure that the warped metric is $C^\infty$ at the pole. However, for most of our comparison results, it suffices to require $w(0)=0$, $w'(0)=1$, and $w(s)>0$ for $s>0$. In this case, we get a model space with less regularity at the pole that we still denote $(M^n_w,g_w)$. As illustrated by the warped product example in Figure \ref{fig:Cyl3} in Section \ref{sec:Glimpses} -- the conditions on the warping function as a function on the distance $r$ from the pole can be even further circumvented while upholding the other necessary conditions (on the level of the Hessian of  more general distance functions) for solving the type problem. The generality of this possibility is discussed in Section \ref{sec:Intermezzo} below. 
\end{remark}

\begin{example}
\label{propSpaceForm}
The simply connected space forms $\mathbb{K}^n(b)$ of constant sectional curvature $b$
can be constructed as $w$-models with any given point as center
point using the warping functions
\begin{equation*}
w_{b}(r):=
\begin{cases} 
\frac{1}{\sqrt{b}}\sin(\sqrt{b}\, r) &\text{if $b>0$, and $r\leq \pi/\sqrt{b}$}
\\
\phantom{\frac{1}{\sqrt{b}}} r &\text{if $b=0$}
\\
\frac{1}{\sqrt{-b}}\sinh(\sqrt{-b}\,r) &\text{if $b<0$}.
\end{cases}
\end{equation*}
Note that, for $b > 0$, the function $w_{b}(r)$ admits a smooth
extension to  $r = \pi/\sqrt{b}$. For $\, b \leq 0\,$ any center
point is a pole.
\end{example}

A weighted model space is a triple $(M^n_w, g_w, e^{h(r)})$ where $e^{h(r)}$ is a radial weight in the model ($M^n_w,g_w)$. In this situation, the weighted volumes of the open metric ball $B^w_R$ of radius $R>0$ centered at $o_w$, and of the sphere $\partial B^w_R$ are computed as follows
\begin{equation*}
\begin{split}
\Vol_h(B^w_R)&=V_0\,\int_0^R w^{n-1}(t)\,e^{h(t)}\,dt,
\\
\Vol_h(\partial B^w_R)&=V_0\,w^{n-1}(R)\,e^{h(R)},
\end{split}
\end{equation*}
where $V_0$ is the Riemannian volume of the unit sphere $\mathbb{S}^{n-1}_1$. The
sectional curvatures of $M^n_w$ for planes containing the radial direction vector, i.e. the gradient $\nabla r$, are determined by the radial function $-w''(r)/w(r)$, and the mean curvature of $\partial B^w_r$ is $\eta_{w}(r)=w'(r)/w(r)$.\\

The weighted capacity $\C^h(B^w_\rho,B^w_R)$ for any two radii $\rho,\,R \in \mathbb{R}$ with $\rho< R$ is computed in \cite{hurtado2020a}. It is determined from the solution to the following weighted Dirichlet problem

\begin{equation}
\label{eqfDirModel2}
\begin{cases}
\Delta_{M_w}^h u = 0\,\,\,&\text{in\, $B^w_R- \widebar{B}^w_\rho$},\\
\phantom{\Delta }u = 1\,\,\,&\text{in\, $\partial B^w_\rho$}, \\
\phantom{\Delta }u = 0\,\,\,&\text{in\, $\partial B^w_R$}.
\end{cases}
\end{equation}

For later use we must mention that a radial function $\phi(r)$ defined on the metric annulus 
 $B^w_R- \widebar{B}^w_\rho$ satisfies the first equation in \eqref{eqfDirModel2} if and only if
\begin{equation}
\label{eq:sabika}
\phi''(r)+\phi'(r)\,\bigg((n-1)\,\frac{w'(r)}{w(r)}+h'(r)\bigg)=0.
\end{equation}

\begin{proposition}[\cite{hurtado2020a}]
\label{capacityweighted} 
In the weighted model space $(M^n_w,g_w,e^{h(r)})$ the solution  to the weighted Dirichlet problem \eqref{eqfDirModel2} is given by the radial function:
\begin{equation}
\label{fsolmodel}
\phi_{\rho,R,h,n}(r):=\left(\int_r^R w^{1-n}
(s)\,e^{-h(s)}\,ds\right)\,\left(\int_\rho^R w^{1-n}
(s)\,e^{-h(s)}\,ds\right)^{-1}.
\end{equation}
Therefore
\begin{equation}
    \phi_{\rho,R,h,n}'(\rho) = - w^{1-n}(\rho)\cdot e^{-h(\rho)} \,\left(\int_\rho^R w^{1-n}(s)\,e^{-h(s)}\,ds\right)^{-1},
\end{equation}

and thence, 
\begin{equation}
\begin{split}
\label{fcapacitymodel}
{\C}^h(B^w_\rho,B^w_R)&=
-\phi_{\rho,R,h,n}'(\rho) \cdot e^{h(\rho)}\cdot w^{n-1}(\rho) \cdot V_{0}\\
&=V_0\,\left(\int_\rho^R w^{1-n}(s)\,e^{-h(s)}\,ds\right)^{-1}.
\end{split}
\end{equation}
\end{proposition}
\begin{remark}\label{Nomodel} If we define the function $\phi_{\rho,R,h,n}(r)$ in \eqref{fsolmodel} for $n\in \mathbb{R}_+$ (i.e. not only $n\in \mathbb{N}$), this function is the unique solution of the differential equation \eqref{eq:sabika} with boundary conditions $\phi_{\rho,R,h,n}(\rho)=1$, and $\phi_{\rho,R,h,n}(R)=0$. Moreover,
$$
-\phi_{\rho,R,h,n}'(\rho) \cdot e^{h(\rho)}\cdot w^{n-1}(\rho) 
=\left(\int_\rho^R w^{1-n}(s)\,e^{-h(s)}\,ds\right)^{-1}.
$$

In this case, we cannot interpret this expression as the weighted capacity of the capacitor $(\widebar{B}^w_\rho,B^w_R)$ in a weighted model space of dimension $n$, but we can still use the analytical expression for comparison purposes. 
\end{remark}

We now introduce the following concept, the $\W$-capacity volume of metric spheres, which we will use in the sequel.

\begin{definition} Let $(M,g,\W)$ be a conductive Riemannian manifold with a pole $o$, and let $B_R$ be the metric open ball centered at $o$. We will denote
\begin{equation}
{\rm Vol}_\W (\partial B_R)=\int_{\partial B_R} \langle \W(\nabla r), \nabla r\rangle\,d\mu_{\partial B_R}.
\end{equation}
\end{definition}

Notice that if $\W=\Id$, ${\rm Vol}_\W (\partial B_R)$ coincides with the Riemannian volume of the metric sphere and, in the isotropic case, $\W=e^f\cdot {\rm Id}$, ${\rm Vol}_\W (\partial B_R)$ is the weighted volume of $\partial B_R$ in the weighted space $(M,g,e^f)$. This definition is also used in \cite{Azami21,Azami23}.

\section{The conductive Laplacian of modified distance functions} \label{sec:BoundingLapDist}
Let $(M,g,\W)$ be a conductive Riemannian manifold such that $M$ is complete and noncompact. Denote by $r: M \to \mathbb{R}$ the distance function from a point $o\in M$. In this section, we will derive estimates away from the cut locus $\rm{cut}(o)$ of $o$ for the $\W$-Hessian and $\W$-Laplacian of modified radial functions $F\circ r$ for a smooth decreasing function $F:(0,+\infty)\to \mathbb{R}$. In order to do this, we use the following well-known result:

\begin{theorem}[\cite{GreeneWu}]\label{teo:GreeneWu} Let $(M,g)$  be a Riemannian manifold. Let $r:M\to \mathbb{R}$ the distance function from a point $o\in M$. Suppose that for any $p\in M-(\rm{cut}(o)\cup \{o\})$ and any plane $\sigma_p \subset T_p M$ containing $(\nabla r)_p$, the sectional curvatures  are bounded by
\begin{equation}\label{ineq:sec}
{\rm sec} (\sigma_p)\leq (\geq) \frac{-w''(r(p))}{w(r(p))}
\end{equation}
for a smooth function $w(s)$ such that $w(0)=0$, $w'(0)=1$ and $w(s)>0$ for all $s>0$. Then for any $p\in M-({\rm cut}(o) \cup\{o\})$ and any $v\in T_p M$
$$
{\rm Hess}\,r(v,v)\geq (\leq)\,\frac{w'(r)}{w(r)}\left(\Vert v\Vert^2-\langle \nabla r,v\rangle^2\right).
$$
 \end{theorem}

 \begin{proposition}\label{ineq1intrinsic} Let $(M^n,g,\W)$ be a conductive Riemannian manifold, $r:M\to \mathbb{R}$ denote the distance function from a point $o\in M$, and $w(s)$ a smooth function such that $w(0)=0$, $w'(0)=1$ and $w(s)>0$ for all $s>0$. Suppose that for any $p\in M-(\rm{cut}(o)\cup \{o\})$ and any plane $\sigma_p \subset T_p M$ containing the gradient $(\nabla r)_p$,
\begin{displaymath}
{\rm sec}(\sigma_p) \leq (\geq)\,- \frac{w''(r(p))}{w(r(p))}.
\end{displaymath}
Then, for every smooth function $F:(0,+\infty)\to \mathbb{R}$ with $F'\leq 0$ we have that:
\begin{equation} \label{ineq:WLaplacian}
\begin{aligned}
    \Delta_\mathcal{W}(F\circ r) &\leq (\geq)\left(F''(r)-F'(r)\frac{w'(r)}{w(r)}\right) \langle\mathcal{W}(\nabla r),\nabla r\rangle\\
    &+{\rm tr} (\mathcal{W})\, F'(r)\frac{w'(r)}{w(r)}
     +F'(r)\,\langle {\rm div} (\mathcal{W}),\nabla r\rangle,
\end{aligned}
\end{equation}
on $M-(\rm{cut}(o)\cup \{o\})$.
\end{proposition}

\begin{proof} By Proposition \ref{prop:LapWexpression}, it is easy to see that

\begin{equation*}
\Delta_\mathcal{W}(F\circ r)=F''(r)\,\langle\mathcal{W}(\nabla r),\nabla r\rangle +F'(r)\, {\rm tr}({\rm Hess}_\W\,r)+F'(r)\,\langle {\rm div} (\mathcal{W}),\nabla r\rangle.
\end{equation*}
 
Then, to bound the Laplacian of $F\circ r$, we will need to bound
${\rm Hess}\, r(e_i,\mathcal{W}(e_i))$ for an orthonormal basis $\{ e_i\}_{i=1}^n$ of $T_p M$. We will use an orthonormal frame $\{e_i\}_{i=1}^n$ of $T_p M$ which diagonalizes $\mathcal{W}$, namely
$$
\mathcal{W}(e_i)=\lambda_i\,e_i.
$$
Applying Theorem \ref{teo:GreeneWu}, we conclude that
$$
\begin{aligned}
 \sum_{i=1}^n{\rm Hess}\, r(e_i,\mathcal{W}(e_i))=&\sum_{i=1}^n\,\lambda_i{\rm Hess}\,r(e_i,e_i)
   \geq (\leq)\sum_{i=1}^n\lambda_i{\frac{w'(r)}{w(r)}\left(1-\langle \nabla r,e_i\rangle^2\right)}\\
   =&{\rm tr}(\mathcal{W})\frac{w'(r)}{w(r)}-{\frac{w'(r)}{w(r)}}\sum_{i=1}^n\langle \nabla r,\lambda_i\,e_i\rangle\langle \nabla r,e_i\rangle\\
   =&{\rm tr} (\mathcal{W})\frac{w'(r)}{w(r)}-{\frac{w'(r)}{w(r)}}\sum_{i=1}^n\langle \nabla r,\W(e_i)\rangle\langle \nabla r,e_i\rangle\\
   =&{\rm tr} (\mathcal{W})\frac{w'(r)}{w(r)}-{\frac{w'(r)}{w(r)}}\sum_{i=1}^n\langle \W(\nabla r),e_i\rangle\langle \nabla r,e_i\rangle\\
   =&{\rm tr} (\mathcal{W})\frac{w'(r)}{w(r)}-{\frac{w'(r)}{w(r)}}\langle \W(\nabla r),\nabla r\rangle,
\end{aligned}
$$
and the result follows from $F'(r)\leq 0$.
\end{proof}

\begin{proposition}
Let $(M,g,\mathcal{W})$ be a conductive Riemannian manifold, with a pole $o \in M$ and denote $r: M \to \mathbb{R}$ the distance function with respect to $o$. Suppose that $\nabla r$ is an eigenvector of $\mathcal{W}$ in $(M,g,\W)$, that is, $\mathcal{W}(\nabla r)= e^{\lambda}\, \nabla r$ (the isotropic conductivity is a particular case). Then, 
\begin{eqnarray*}
\Delta_{\mathcal{W}}(F\circ r)&=&e^\lambda\, \Delta^\lambda (F\circ r),
\end{eqnarray*}
where $\Delta^\lambda$ is the weighted Laplacian for the density 
$e^\lambda$ on $M$. Hence, $M$ is $\mathcal{W}$-parabolic if and only if $(M,g, e^\lambda)$ is $\lambda$-parabolic in the weighted sense.\\    
\end{proposition}
\begin{proof}
\begin{eqnarray*}
\Delta_{\mathcal{W}}(F\circ r)&=&{\rm div}(\mathcal{W}(F'(r)\nabla r))={\rm div}(e^\lambda \,F'(r)\, \nabla r))\\
&=& e^\lambda\,F''(r)  + e^\lambda\,F'(r)\,\langle \nabla \lambda,\nabla r \rangle +  e^\lambda\,F'(r)\, {\rm div}(\nabla r)\\
&=& e^\lambda\, \left(\Delta(F\circ r)+ \langle \nabla \lambda, \nabla (F\circ r)\rangle\right)=e^\lambda\, \Delta^\lambda (F\circ r).
\end{eqnarray*}
\end{proof}

\section{Intermezzo concerning a possible comparison extension} \label{sec:Intermezzo}

The proof of the above proposition holds true also in the following more general setting:
Let $\tilde{\rho}: M \rightarrow \mathbb{R}$ be a continuous exhaustion function on $M$. That means that the level sets $B_t=\{p\in M\,:\, \tilde{\rho}(p)< t\}$ are precompact and $\tilde{\rho}(p) \to +\infty$ when $p$ is leaving any compact set. Assume moreover that $\tilde{\rho}$ is smooth and that $\Vert\nabla 
\tilde{\rho}\Vert=1$, in both cases out  of a compact set, so  $\partial B_t$  is a smooth hypersurface for all $t$ large enough, namely, for $t\geq \rho_0$ for some number $\rho_0>0$. If $M$ is a manifold with a pole $o$, the distance function $r$ from the pole is an example of such a function. 

 Given such an exhaustion $\tilde{\rho}: M \rightarrow \mathbb{R}$, we consider the following definition: we say that the \emph{Hessian of $\tilde{\rho}$ is $\eta$-bounded from below (resp. from above)} if there exists a number $\rho_0>0$ and a continuous function $\eta:[\rho_0,+\infty)\to \mathbb{R}$ such that for any $p\in M-B_{\rho_0}$ and any $v\in T_pM$
\begin{equation}
    {\rm Hess}\, \tilde{\rho}(v,v)\geq (\leq)\eta(\tilde{\rho})\,(\Vert v\Vert^2-\langle \nabla \tilde{\rho},v\rangle).
\end{equation}

Associated with the bounding function $\eta$, we can define the function 
\begin{equation}
    w(\tilde{\rho})= C\,e^{\int_{\rho_0}^{\tilde{\rho}} \eta(s)\,ds}
\end{equation}
for some positive constant $C$. This function is solution to the second order differential equation 
\begin{equation}
w'(\tilde{\rho})-\eta(\tilde{\rho})\, w(\tilde{\rho})=0.
\end{equation}

With these considerations, we can recover that the inequalities 
\begin{equation}
\begin{aligned}
    \Delta_\mathcal{W}\,(F\circ \tilde{\rho}) &\leq (\geq)\left(F''(\tilde{\rho})-F'(\tilde{\rho})\eta(\tilde{\rho})\right) \langle\mathcal{W}(\nabla \tilde{\rho}),\nabla \tilde{\rho}\rangle+{\rm tr} (\mathcal{W})\, F'(\tilde{\rho})\eta(\tilde{\rho})\\
    & +F'(\tilde{\rho})\langle {\rm div} (\mathcal{W}),\nabla \tilde{\rho}\rangle
\end{aligned}
\end{equation}
\noindent are satisfied on $M-B_{\rho_0}$, by any smooth $F:(0,+\infty)\to \mathbb{R}$ with $F'\leq 0$ when we have a smooth exhaustion function $\tilde{\rho}$ defined on $M-B_{\rho_0}$, satisfying $\Vert\nabla 
\tilde{\rho}\Vert=1$ and such that the Hessian of $\tilde{\rho}$ is $\eta$-bounded from below (resp. from above), for some function $\eta$. 

Therefore, the comparison results for the capacity established in Sections \ref{sec:IntrinsicCrit} and \ref{sec:ExtrinsicCrit}  can be stated for manifolds where  such an exhaustion function $\tilde{\rho}$ is defined on $M-B_{\rho_0}$, with the Hessian bounded from below, (resp. from above). In particular, these results hold for manifolds with a pole such that the distance to the pole has Hessian bounded from below or from above.

On the other hand, notice that if $M$ has a pole $o$ and $r$ is the distance from the pole, Theorem \ref{teo:GreeneWu} asserts that when the sectional curvatures in the radial directions satisfies inequality \eqref{ineq:sec} for a smooth function $w(s)$ such that $w(0)=0$ and $w(s)>0$ for all $s>0$, then the Hessian of $r$ is $\eta$-bounded from below (resp. from above) for any $\rho_0>0$ and $\eta(r)=w'(r)/w(r)$.\\

Hence, for the sake of simplicity, in the following sections we will only work with manifolds with poles $o\in M$, and we will correspondingly assume the classical hypotheses on the radial sectional curvatures as viewed from these poles.

\section{Intrinsic criteria for $\mathcal{W}$-parabolicity and $\mathcal{W}$-hyperbolicity} \label{sec:IntrinsicCrit}

Let $(M,g,\mathcal{W})$ be a conductive Riemannian manifold with a pole $o \in M$ and denote $r: M \to \mathbb{R}$ the distance function with respect to $o$. In this section, we provide estimates for the conductive capacity of a metric ball centered at the pole, assuming lower or upper bounds for the radial sectional curvatures of the manifold. From the capacity estimates we will deduce conclusions about the type problem: Is the conductive Riemannian manifold $\W$-parabolic or is it $\W$-hyperbolic?\\

Throughout this section we will denote by $B_R$ the open metric ball of radius $R>0$ centered at $o$.

\begin{theorem}  \label{thm:Rp}
Let $(M^n,g,\mathcal{W})$ be a conductive Riemannian manifold with a pole $o \in M$, $r:M\to \mathbb{R}$ the distance function from $o$ and $w(s)$ a smooth function such that $w(0)=0$, $w'(0)=1$, and $w(s)>0$ for all $s>0$. Suppose that the following conditions are fulfilled:
\begin{itemize}
\item[(a) ] For any $p\in M- \{o\}$ and any plane $\sigma_p \subset T_p M$ containing $(\nabla r)_p$,
\begin{displaymath}
{\rm sec}(\sigma_p) \leq (\geq) - \frac{w''(r(p))}{w(r(p))}.
\end{displaymath}
\item[(b) ] There are numbers $\rho,\,q>0$, and a continuous function $\theta:[\rho,+\infty)\to \mathbb{R}$ such that the following inequalities hold in $M- B_{\rho}$ i.e. for all $r \geq \rho$:
\begin{enumerate}
\item $w'(r)\bigg({\rm tr}(\mathcal{W})- q\,\langle\mathcal{W}(\nabla r), \nabla r\rangle\bigg) \geq (\leq)\,\, 0$,
\item $\langle{\rm div}(\mathcal{W}),\nabla r\rangle \geq (\leq)\, \theta(r)\,\langle\mathcal{W}(\nabla r), \nabla r\rangle$.
\end{enumerate}
\end{itemize}
\medskip
Then, for all $R>\rho$,
\begin{equation}\label{eq:capacity_comp_q}
{\rm Cap}_\mathcal{W}(B_\rho, B_R)\geq (\leq)-\phi'_{\rho,R,h,q}(\rho) \,{\rm Vol}_\W(\partial B_\rho),
\end{equation}
where $\phi_{\rho,R,h,q}$ is the function obtained from \eqref{fsolmodel} by replacing $n$ by $q$, with $h(t)=\int_{\rho}^t\theta(s)\,ds$ for any $t\geq \rho$.\\
\medskip
Moreover, if 
\begin{equation}\label{condinfty}
\int_\rho^\infty w^{1-q}(s)\,e^{-h(s)}ds < (=)\,\,\infty,
\end{equation}
then $M$ is $\mathcal{W}$-hyperbolic (resp. $\mathcal{W}$-parabolic).
\end{theorem}

\begin{proof}
Given $q\in \mathbb{R}_+$, let $\phi_{\rho,R,h,q}$ be the function given in equation \eqref{fsolmodel} for $n=q$, with $h(t)=\int_{\rho}^t\theta(s)\,ds$ for any $t\geq \rho$. Consider the transplanted radial function $v:=\phi_{\rho,R,h,q}\circ r$ defined on {$\widebar{B}_R- B_\rho$}. For $r\geq \rho$, we have that  $\phi_{\rho,R,h,q}'(r)\leq 0$ and 
\begin{equation}
\phi_{\rho,R,h,q}''(r)-\phi_{\rho,R,h,q}'(r)\frac{w'(r)}{w(r)}=-\phi_{\rho,R,h,q}'(r)\left(q\frac{w'(r)}{w(r)}+\theta(r)\right).
\end{equation}

We apply Proposition \ref{ineq1intrinsic} and hypotheses (a) and (b) to obtain that
\begin{eqnarray*}
\begin{aligned}
\Delta_\mathcal{W}\,v&\leq (\geq)\bigg(\phi_{\rho,R,h,q}''(r)-\phi_{\rho,R,h,q}'(r)\frac{w'(r)}{w(r)}\bigg)\langle\mathcal{W}(\nabla r),\nabla r\rangle\\
&+\phi_{\rho,R,h,q}'(r)\,{\rm tr}(\mathcal{W})\,\frac{w'(r)}{w(r)}+ \phi_{\rho,R,h,q}'(r)\, \langle{\rm div}(\mathcal{W}),\nabla r\rangle\\
&=-\phi_{\rho,R,h,q}'(r)\left(q\frac{w'(r)}{w(r)}+\theta(r)\right)\langle\mathcal{W}(\nabla r),\nabla r\rangle \\
&+\phi_{\rho,R,h,q}'(r)\,{\rm tr}(\mathcal{W})\,\frac{w'(r)}{w(r)} +\phi_{\rho,R,h,q}'(r)\, \langle{\rm div}(\mathcal{W}),(\nabla r)\rangle\\
&=\phi_{\rho,R,h,q}'(r)\,\frac{w'(r)}{w(r)}\bigg({\rm tr}(\mathcal{W})-q\,\langle\mathcal{W}(\nabla r),\nabla r\rangle\bigg)\\&+\phi_{\rho,R,h,q}'(r) \big(\langle{\rm div}(\mathcal{W}),(\nabla r)\rangle-\theta(r)\langle\mathcal{W}(\nabla r),\nabla r\rangle\big) \leq (\geq) \,\,0=\Delta_\mathcal{W}\,u,
\end{aligned}
\end{eqnarray*}
where $u$ is the solution to the Dirichlet problem \eqref{eqWLap} in $B_R- \widebar{B}_\rho$.\\

We apply Theorem \ref{Hopf} to the $\mathcal{W}$-subharmonic function $-v$ (resp. $v$) and the $\W$-harmonic function $-u$ (resp. $u$) to obtain that $u \leq (\geq)\, v$ on $B_R- \widebar{B}_{\rho}$ and that 
\begin{equation}\label{inteqgrad1}
\langle \mathcal{W}(\nabla v), n_{\partial B_\rho}\rangle \leq (\geq) \langle \mathcal{W}(\nabla u), n_{\partial B_\rho}\rangle
\end{equation}
\noindent on $\partial B_\rho$, where $n_{\partial B_\rho}$ denotes the unit outward (pointing outward $B_R-\widebar{B}_\rho$) normal vector along $\partial B_\rho$. Hence, as $v$ is non-increasing with respect $r$ and $v=1$ on $\partial B_\rho$, $n_{\partial B_\rho}=\nabla v/\Vert\nabla v\Vert=-\nabla r$.\\

\noindent Therefore, by definition of $\mathcal{W}$-capacity, 

\begin{eqnarray*}
{\rm Cap}_\mathcal{W}(B_\rho,B_R)&=&\int_{\partial B_\rho} \langle \mathcal{W}(\nabla u), n_{\partial B_\rho}\rangle\, d\mu_{\partial B_\rho}\\
&\geq (\leq)& \int_{\partial B_\rho} \langle\mathcal{W}(\nabla v), \frac{\nabla v}{|\nabla v|}\rangle\,d\mu_{\partial B_\rho}\\
&=&-\phi'_{\rho,R,h,q}(\rho)\int_{\partial B_\rho} \langle \mathcal{W}(\nabla r), \nabla r\rangle\,d\mu_{\partial B_\rho}\\
&=&-\phi'_{\rho,R,h,q}(\rho)\,{\rm Vol}_\W(\partial B_\rho).
\end{eqnarray*}
Taking limits when $R$ goes to $\infty$ and taking into account that 
$$\int_{\partial B_\rho} \langle \mathcal{W}(\nabla r), \nabla r\rangle\,d\mu_{\partial B_\rho} >0,$$ 
if \eqref{condinfty} is fulfilled, we deduce that ${\rm Cap}_\W (B_\rho)> (=)\,\,0$ and then $M$ is $\W$-hyperbolic (resp. $\W$-parabolic).
\end{proof}

\begin{remark}
When $q\in \mathbb{N}$, we consider the $q$-dimensional weighted model space $(M^q_w,g_w,e^{h(r_w)})$ with $h(t)=\int_{\rho}^t\theta(s)\,ds$ for $t\geq \rho$. Then, using \eqref{fcapacitymodel}, equation \eqref{eq:capacity_comp_q} reads
\begin{equation} \label{eq:CapComp1}
\begin{aligned}
\frac{{\rm Cap}_\mathcal{W}(B_\rho,B_R)}{{\rm Vol}_{\mathcal{W}}(\partial B_\rho)}\geq (\leq)\frac{\C^h(B^{w}_\rho,B^w_R)}{{\rm Vol}_h(\partial B^w_\rho)}.
\end{aligned}
\end{equation}
Moreover, the integral condition \eqref{condinfty} is equivalent to the weighted hyperbolicity (resp. weighted parabolicity) of the weighted manifold $(M^q_w,g_w,e^{h(r_w)})$, where $r_w$ denotes here the distance from the pole in the model.\\
\end{remark}

\begin{remark}
For the isotropic case $\W=e^f\cdot {\rm Id}$ with $q=n$, we have that $\langle\mathcal{W}(\nabla r),\nabla r\rangle=e^f$, ${\rm tr}(\mathcal{W})=n\,e^f$, and ${\rm div}(\mathcal{W})=e^f\,\nabla f$. Then, hypothesis (1) in (b) is always fulfilled and condition (2) in (b) is equivalent to the existence of a continuous function $\theta:[\rho,+\infty)\to \mathbb{R}$ such that $\langle \nabla f,\nabla r \rangle \geq (\leq) \,\theta (r)$ for all $p\in M- B_{\rho}$. In this way, we recover Theorems 4.9 and 4.10 in \cite{hpr2020a}, since we can replace in Theorem \ref{thm:Rp}, ${\rm sec}(\sigma_p)\geq - \frac{w''(r(p))}{w(r(p))}$ by ${\rm Ric}((\nabla r)_p, (\nabla r)_p)\geq -(n-1) \frac{w''(r(p))}{w(r(p))}$ in the isotropic case and get the same conclusions. However, this idea cannot be implemented for a general $\mathcal{W}$. 
\end{remark}

The $2$-parameter family of examples in Proposition \ref{1:prop:2DVicent} follows from the above Theorem. For convenience we repeat the proposition here as a corollary to Theorem \ref{thm:Rp} with a proof as follows: 

\begin{corollary} \label{cor:2DVicent}
Let $(M^{2}, g, \W) = (\mathbb{R}^{2}, g_{\rm can}, \W_{\lambda,\alpha})$ denote a conductive Riemannian manifold where $g_{\rm can}$ denotes the Euclidean metric in $\mathbb{R}^{2}$, and the conductivity operator $\W_{\lambda,\alpha}$ is given by the following smooth field of tangent space endomorphisms at each point $(x,y) \in \mathbb{R}^{2}$:
\begin{equation}
\begin{aligned}
 \W_{\lambda,\alpha}(\partial x) &= e^{\alpha\cdot(x^2 + y^2)}\left((\lambda+1)\partial x + (\lambda-1)\partial y \right)  \\
 \W_{\lambda,\alpha}(\partial y) &=  e^{\alpha\cdot(x^2 + y^2)}\left( (\lambda-1)\partial x + (\lambda+1)\partial y \right)  \quad ,
 \end{aligned}
\end{equation}
with $\lambda>0$ and $\alpha \in \mathbb{R}$. Then, for any choice of $\lambda>0$, $(\mathbb{R}^{2}, g_{\rm can}, \W_{\lambda,\alpha})$ is $\W_{\lambda,\alpha}$-parabolic for $\alpha \leq 0$ and  $\W_{\lambda,\alpha}$-hyperbolic for $\alpha > 0$. 
\end{corollary}

\begin{proof} Notice that the eigenvalues of $\W_{\lambda,\alpha}$ at each point $(x,y)$ are given by $\lambda_1(x,y)=2\,e^{\alpha \cdot r^{2}}$ and $\lambda_2(x,y)=2\,\lambda\,e^{\alpha \cdot r^{2}}$, where $r^{2} = x^{2} + y^{2}$. Then, for $\lambda>0$, $\W_{\lambda,\alpha}$ is positive definite.

We have that, for $r>0$,
$$
\begin{aligned}
    \langle \nabla r,\mathcal{W}\left(\nabla  r\right)\rangle=&e^{\alpha \cdot r^{2}}\left(\frac{x^2}{r^2}(\lambda+1)+\frac{y^2}{r^2}(\lambda+1)+\frac{2 x y}{r^2}(\lambda-1)\right)\\
    &=e^{\alpha \cdot r^{2}}\left(\lambda+1+\frac{2 x y}{r^2}(\lambda-1)\right).
\end{aligned}
$$
Taking into account that  $-1\leq \frac{2xy}{r^2}\leq 1$  , we can conclude that
$$
\frac{{\rm tr}(\mathcal{W})}{\tilde{q}}= 2e^{\alpha \cdot r^{2}}\,\min\{1,\lambda\}\leq  \langle \nabla r,\mathcal{W}\left(\nabla  r\right)\rangle\leq 2e^{\alpha \cdot r^{2}}\,\max\{1,\lambda\}\leq \frac{{\rm tr}(\mathcal{W})}{1},
$$
where 
\begin{equation*}
\tilde{q}=1+\max\{\lambda, \frac{1}{\lambda}\}.
\end{equation*}
Moreover,
$$
\begin{aligned}
\langle \nabla r, {\rm div} (\mathcal{W})\rangle=&2\alpha r  e^{\alpha \cdot r^{2}}\left(\lambda+1+\frac{2xy}{r^2}(\lambda-1)\right)=2\alpha r \langle \nabla r,\mathcal{W}\left(\nabla r\right)\rangle.
\end{aligned}
$$
Therefore for $\alpha < 0$ we can apply Theorem \ref{thm:Rp} with $w(t)=t$,  $\theta(t)=2\alpha t$ ,  and $q=\tilde{q}$ 
and since 
$$
\int_1^\infty \frac{1}{s^{\tilde{q}-1}}e^{-\alpha \cdot s^{2}}\, ds=\infty,
$$
we can conclude that $(\mathbb{R}^2,g_{\rm can}, \mathcal{W}_{\lambda,\alpha})$ is $\mathcal{W}_{\lambda,\alpha}$-parabolic for any $\lambda>0$ and $\alpha<0$.
For the case  $\alpha >0$, by using Theorem \ref{thm:Rp}  with $w(t)=t$,  $\theta(t)=2\alpha t$ , and $q=1$ we get:

$$
\int_1^\infty e^{-\alpha \cdot s^{2}}\, ds<\infty,
$$
and we can conclude that $(\mathbb{R}^2,g_{\rm can}, \mathcal{W}_{\lambda,\alpha})$ is $\mathcal{W}_{\lambda,\alpha}$-hyperbolic for any $\lambda>0$ and $\alpha > 0$. 

For $\alpha=0$, notice that $\W_{\lambda,0}$ is constant and therefore it has bounded eigenvalues. Then, by Remark \ref{boundedeigen}, $(\mathbb{R}^2,g_{\rm can}, \mathcal{W}_{\lambda,0})$ is $\mathcal{W}_{\lambda,0}$-parabolic since $(\mathbb{R}^2,g_{\rm can})$ is parabolic.

\end{proof}

As another consequence of Theorem \ref{thm:Rp}, for \emph{divergence-free conductivities} $\W$ and  $\theta = 0$ in (2) of (b), we can state the following result:

\begin{corollary} \label{corollaryC}
Let $(M^n,g,\mathcal{W})$ be a conductive Riemannian manifold with a pole $o \in M$, $r:M\to \mathbb{R}$ the distance function from $o$, and $w(s)$ a smooth function such that $w(0)=0$, $w'(0)=1$, and $w(s)>0$ for all $s>0$. Suppose that $\mathcal{W}$ is divergence-free and that the following conditions are fulfilled:
\begin{itemize}
\item[(a) ] For any $p\in M- \{o\}$ and any plane $\sigma_p \subset T_p M$ containing $(\nabla r)_p$,
\begin{displaymath}
{\rm sec}(\sigma_p) \leq (\geq) - \frac{w''(r(p))}{w(r(p))}.
\end{displaymath}\\
\item[(b) ] {There are  numbers $\rho,\, q>0$  such that in $M- B_{\rho}$:}
\begin{equation*}
 w'(r)\bigg({\rm tr}(\mathcal{W})- q\,\langle\mathcal{W}(\nabla r), \nabla r\rangle\bigg) \geq (\leq)\,0.
\end{equation*}
\end{itemize}
Then, for all $R \geq \rho$,
\begin{equation}
{{\rm Cap}_\mathcal{W}(B_\rho,B_R)\geq (\leq)-\phi'_{\rho,R,0,q}(\rho)\,{\rm Vol}_\W(\partial B_\rho)},
\end{equation}
where  $\phi_{\rho,R,0,q}$ is the function given in \eqref{fsolmodel} for $n=q$, with $h=0$.\\

Moreover, if 
\begin{equation}\label{condinftyfree}
\int_\rho^\infty w^{1-q}(s)\,ds < (=)\,\,\infty,
\end{equation}
then $M$ is $\mathcal{W}$-hyperbolic (resp. $\mathcal{W}$-parabolic).
\end{corollary}

Using the definition of conductivity, for all $p\in M$ and all $v\in T_p M$, $\mu(p)\,\Vert v\Vert^2 \leq \langle \W(v),v \rangle\leq \kappa(p)\,\Vert v\Vert^2$ for some smooth positive functions $\mu$ and $\kappa$, and we can now state the following results.

\begin{theorem} \label{propIntrinsicA2}
Let $(M^n,g,\mathcal{W})$ be a conductive Riemannian manifold with a pole $o \in M$, $r:M\to \mathbb{R}$ the distance function from $o$, and $w(s)$ a smooth function such that $w(0)=0$, $w'(0)=1$, and $w(s)>0$ for all $s>0$.  Let $\kappa$ be as in Definition \ref{defConducRiem} and suppose that the following conditions are fulfilled:
\begin{itemize}
\item[(a) ] For any $p\in M- \{o\}$ and any plane $\sigma_p \subset T_p M$ containing $(\nabla r)_p$,
\begin{displaymath}
{\rm sec}(\sigma_p) \geq (\leq) - \frac{w''(r(p))}{w(r(p))}.
\end{displaymath}\\

\item[(b) ] There is a number $\rho>0$ and a continuous functions $\theta:[\rho,+\infty)\to \mathbb{R}$ such that for all $p\in M- B_{\rho}$:
\medskip
\begin{enumerate}

\item $\langle{\rm div}(\mathcal{W}),(\nabla r)_p\rangle/ \kappa(p) \leq (\geq)\,\theta(r(p))$.
\medskip
\item $\theta(r(p))+n \frac{w'(r(p))}{w(r(p))}\leq (\geq)\,\,0$ and $w'(r(p))\geq (\leq)\,\,0$ (balance condition).
\end{enumerate}
\end{itemize}
Then, for all $R>\rho$,
\begin{eqnarray}\label{compIntrinsicA2}
    \frac{{\rm Cap}_\mathcal{W}(B_\rho,B_R)}{ {\rm Vol}_\W (\partial B_\rho)}&\leq (\geq)&  \frac{{\rm Cap}^h(B^w_\rho,B^w_R)}{{\rm Vol}_h(\partial B^w_\rho)},
\end{eqnarray}
where ${\rm Cap}^h(B^w_\rho,B^w_R)$ is the weighted capacity of a metric annulus in the  weighted model space $(M^n_w, g_w, e^{h(r_w)})$ with $h(t)=\int_{\rho}^t \theta(s)\,ds$ for any $t\geq \rho$, and ${\rm Vol}_h$ is the corresponding weighted volume in $M_w^n$.\\ 

Moreover, {since $\int_\rho^\infty w^{1-n}(s)\,e^{-h(s)}ds=\infty$, then} $M$ is $\mathcal{W}$-parabolic (resp. if $
\int_\rho^\infty w^{1-n}(s)\,e^{-h(s)}ds <\infty$, 
then $M$ is $\W$-hyperbolic).

\end{theorem}

\begin{proof}
Let $\phi_{\rho,R,h,n}$ be the function given in \eqref{fsolmodel}, with $h(t)=\int_{\rho}^t \theta(s)\,ds$ for any $t\geq \rho$. Consider the transplanted radial function $v:=\phi_{\rho,R,h,n}\circ r$ defined on $\widebar{B}_R- B_\rho$. For $r\geq \rho$, we have that  $\phi_{\rho,R,h,n}'(r)\leq 0$ and 
\begin{equation*}
\phi_{\rho,R,h,n}''(r)-\phi_{\rho,R,h,n}'(r)\frac{w'(r)}{w(r)}=-\phi_{\rho,R,h,n}'(r)\left(n\frac{w'(r)}{w(r)}+\theta(r)\right) \leq 0.
\end{equation*}

Taking into account that $\langle\mathcal{W}(\nabla r),\nabla r\rangle\leq \kappa$ and ${\rm tr}(\W)\leq n\,\kappa$, we apply Proposition \ref{ineq1intrinsic} and hypotheses (a) and (b) to obtain

\begin{eqnarray*}
\begin{aligned}
\Delta_\mathcal{W}\,v&\geq (\leq) \bigg(\phi_{\rho,R,h,n}''(r)-\phi_{\rho,R,h,n}'(r)\frac{w'(r)}{w(r)}\bigg)\langle\mathcal{W}(\nabla r),\nabla r\rangle\\
&+{\rm tr}(\mathcal{W})\, \phi_{\rho,R,h,n}'(r)\frac{w'(r)}{w(r)}+ \phi_{\rho,R,h,n}'(r) \langle{\rm div}(\mathcal{W}),\nabla r\rangle\\
& \geq (\leq) \kappa \bigg(\phi_{\rho,R,h,n}''(r)-\phi_{\rho,R,h,n}'(r)\frac{w'(r)}{w(r)}\bigg)+ n\,\kappa\, \phi_{\rho,R,h,n}'(r)+\kappa\,\phi_{\rho,R,h,n}'(r) \theta(r)\\
&=  \kappa \bigg( \phi_{\rho,R,h,n}''(r)-(n-1)\phi_{\rho,R,h,n}'(r)\frac{w'(r)}{w(r)} +\phi_{\rho,R,h,n}'(r) \theta(r)\bigg)=0=\Delta_\mathcal{W}\,u,
\end{aligned}
\end{eqnarray*}
where $u$ is the solution to the Dirichlet problem \eqref{eqWLap} in $B_R - \widebar{B}_\rho$.\\

\noindent Following the lines of the proof of Theorem \ref{thm:Rp}, we obtain the result. Notice that, under the hypothesis of ${\rm sec}(\sigma_p) \geq  - \frac{w''(r(p))}{w(r(p))}$, since $\theta(r)\leq -n w'(r)/w(r)$ and $w'(r)\geq 0$ on $M- B_{\rho}$, integrating this inequality, we obtain that
\begin{equation}
h(t)=\int_{\rho}^t \theta(s)ds\leq \log \left(\frac{w(\rho)^n}{w(t)^n}\right),\quad  t\geq \rho.
\end{equation}
Then,
\begin{equation}
\int_\rho^\infty w^{1-n}(s)\,e^{-h(s)}ds \geq \frac{1}{w(\rho)^n}\int_\rho^\infty w(s) ds=\infty,
\end{equation}

and then $M$ is $\W$-parabolic.

\end{proof}
Using the function $\mu$ instead of $\kappa$ in hypothesis (1) of (b) of the above theorem we have the following $\W$-parabolicity result:
\begin{theorem}  \label{propIntrinsicB2}
Let $(M,g,\mathcal{W})$ be a conductive Riemannian manifold with a pole $o \in M$, $r:M\to \mathbb{R}$ the distance function from $o$, and $w(s)$ a smooth function such that $w(0)=0$, $w'(0)=1$, and $w(s)>0$ for all $s>0$. Let $\mu$ be as in Definition \ref{defConducRiem} and suppose that the following conditions are fulfilled:
\begin{itemize}
\item[(a) ] For any $p\in M- \{o\}$ and any plane $\sigma_p \subset T_p M$ containing $(\nabla r)_p$,
\begin{displaymath}
{\rm sec}(\sigma_p) \geq - \frac{w''(r(p))}{w(r(p))}.
\end{displaymath}

\item[(b) ] There is a number $\rho>0$ and continuous function $\theta:[\rho,+\infty)\to \mathbb{R}$ such that for all $p\in M- B_{\rho}$:
\medskip
\begin{enumerate}

\item $\langle{\rm div}(\mathcal{W}),(\nabla r)_p\rangle/\mu(p) \leq \theta(r(p))$.
\medskip
\item $\theta(r(p))+n \frac{w'(r(p))}{w(r(p))}\geq 0$ and $w'(r(p))\leq 0$ (balance condition).
\end{enumerate}
\end{itemize}
Then, for all $R>\rho$,
\begin{equation}
\frac{{\rm Cap}_\mathcal{W}(B_\rho,B_R)}{{\rm Vol}_\W(\partial B_\rho)}\leq \frac{{\rm Cap}^h(B^w_\rho,B^w_R)}{{\rm Vol}_h(\partial B^w_\rho)},
\end{equation}
where ${\rm Cap}^h(B^w_\rho,B^w_R)$ is the weighted capacity of a metric annulus in the  weighted model space $(M^n_w,g_w, {\rm e}^{h(r_w)})$ with $h(t)=\int_{\rho}^t \theta(s) ds$ for any $t\geq \rho$, and ${\rm Vol}_h$ is the corresponding weighted volume in $M_w^n$.\\

Moreover, if 
\begin{equation*}
\int_\rho^\infty w^{1-n}(s)\,e^{-h(s)}ds=\infty,
\end{equation*}
then $M$ is $\mathcal{W}$-parabolic.
\end{theorem}
\begin{remark} For the isotropic case, $\W=e^f\cdot{\rm Id}$ and then $\langle\mathcal{W}(\nabla r),\nabla r\rangle=e^f=\kappa=\mu$, ${\rm tr}(\mathcal{W})=n\,e^f$, and ${\rm div}(\mathcal{W})=e^f\,\nabla f$. In this case, we do not need hypothesis $(2)$ and we recover Theorems 4.9 and 4.10 in \cite{hpr2020a}. As mentioned before, in the isotropic case (and then, in the Riemannian case also) Theorems \ref{propIntrinsicA2} and \ref{propIntrinsicB2} can be stated bounding the Ricci curvature instead of the sectional one when the sectional curvature is bounded from below. This cannot be done for a general $\mathcal{W}$.
\end{remark}


\section{The extrinsic type problem} \label{sec:ExtrinsicCrit}


Let  $(M^n,g,\mathcal{W})$  be a conductive Riemannian manifold. Let $\Sigma^m\subset M^n$ be a complete $m$-dimensional immersed submanifold of $M^{n}$, $m < n$. For each point $p\in \Sigma$, the tangent space $T_pM$ splits in two orthogonal components:
$$
T_pM=T_p\Sigma\oplus\left(T_p\Sigma\right)^\perp
$$
where $T_p\Sigma$ are the vectors of $T_pM$ tangent to $\Sigma$ and $\left(T_p\Sigma\right)^\perp$ is the orthogonal complement.

\begin{definition} \label{def:CompatibleConduc}
We will say that the immersed submanifold $\Sigma\subset M$ is a \emph{$\mathcal{W}$-compatible submanifold} if for every $p\in \Sigma$, the $(1,1)$-tensor $\mathcal{W}$ satisfies 
\begin{enumerate}
          \item $\mathcal{W}\left(T_p\Sigma\right)=T_p\Sigma$,
        \item $\mathcal{W}\left(\left(T_p\Sigma\right)^\perp\right)=\left(T_p\Sigma\right)^\perp.$
    \end{enumerate} 
\end{definition}
Observe that a $\mathcal{W}$-compatible submanifold $\Sigma\subset M$ with the restriction of the metric tensor $g$ and the restriction of $\mathcal{W}$ to $\Sigma$,  becomes a conductive Riemannian (sub)manifold $(\Sigma, g\vert_\Sigma, \W\vert_\Sigma)$. We denote the Riemannian connections in $M$ and $\Sigma$ by $\nabla$ and $\nabla^\Sigma$, respectively.
Here, for any two vectors $v,w\in T_q\Sigma$ and any extensions into smooth vector fields $V,W$ in $M$ (with $V(q)=v$, $W(q)=w$), we then have
$$
\nabla\mathcal{W}(v,w)=\left.\nabla_w\mathcal{W}(V)-\mathcal{W}(\nabla_w V)\right..
$$
Observe moreover that
$$
{\rm tr}_\Sigma(\nabla \mathcal{W})\neq {\rm tr}_\Sigma (\nabla^\Sigma \mathcal{W}).
$$
Indeed, the vector field ${\rm tr}_\Sigma (\nabla \mathcal{W})$ could be non-tangent to $\Sigma$.\\

Conversely, if we are given a conductivity field $\W^\Sigma$ on $(\Sigma, g\vert_\Sigma)$ then the exponential map from the normal bundle of $\Sigma$ provides smooth extensions $\W$ of $\W^\Sigma$ into a tubular neighborhood so that $\W^{\Sigma}$ becomes the restriction of $\W$ in the following local sense at each point $p \in \Sigma$: $\W(V) = \W^{\Sigma}(V)$ for all $V \in T_{p}\Sigma$ and $\W(U) \in (T_{p}\Sigma)^\perp$ for all $U \in (T_{p}\Sigma)^\perp$. We will prove below that the results obtained in this section are independent of the chosen extension of $\W^\Sigma$, so we can use as a starting point that $\Sigma$ is a $\W$-compatible submanifold of $(M,g,\W)$.  

\begin{definition}
\label{extball} 
If $\Sigma$ is a noncompact submanifold properly immersed in a manifold $M$ with a pole $o$, the \emph{extrinsic metric ball} of (sufficiently large) radius $R>0$ and center $o$ is denoted by $D_R$, and defined as any connected component of the set
\begin{displaymath}
B_R\cap \Sigma=\{p\in \Sigma:\,r(p)<R\},
\end{displaymath}
where $r$ denotes the distance function from the pole $o$.
\end{definition}

Since $\Sigma$ is properly immersed in $M$ the extrinsic balls are precompact open sets in $\Sigma$. As we assume that $\Sigma$ is noncompact, then $D_R\neq \Sigma$ for any $R>0$. Moreover, by Sard's Theorem we deduce that $\partial D_R$ is smooth for almost any $R>0$.

Now, we present another necessary ingredient to establish our results:

\begin{definition}\label{defWH} Let  $(M^n,g,\mathcal{W})$  be a conductive Riemannian manifold. Let $\Sigma\subset M$ be a $\mathcal{W}$-compatible immersed $m$-dimensional submanifold. The $\W$-mean curvature vector of $\Sigma$ is the vector field given by

$$
\Vec{H}_{\mathcal{W}}=\frac{1}{m}\left( \mathcal{W}({\rm tr}_\Sigma B)+{\rm tr}_\Sigma (\nabla\mathcal{W})\right)= \mathcal{W}(\Vec{H})+\frac{1}{m}{\rm tr}_\Sigma( \nabla\mathcal{W}),
$$
where $B$ is the second fundamental form of $\Sigma$ in $M$, namely, $B(X,Y)=(\nabla_X Y)^\bot$ for all $X,\,Y\,\in T \Sigma$, and $\Vec{H}$ is the usual mean curvature vector of $\Sigma$. See references \cite{hpr2020a,hurtado2020a}.
\end{definition}

\begin{proposition}
  Let  $(M^n,g,\mathcal{W})$  be a conductive Riemannian manifold. Let $\Sigma\subset M$ be $\mathcal{W}$-compatible immersed $m$-dimensional submanifold. Given a point $p\in \Sigma$, and an orthonormal basis $\{e_i\}_{i=1}^m$ of $T_p\Sigma$, then  
$$
   m\, \Vec{H}_{\mathcal{W}}(p)= \sum_{i=1}^mB(e_i,\mathcal{W}(e_i))+{\rm div}^\Sigma(\mathcal{W}),
$$  
where ${\rm div}^\Sigma(\mathcal{W}) ={\rm tr}_\Sigma(\nabla^\Sigma\mathcal{W})$ is the divergence on the (sub)manifold $\Sigma$ of the restriction of $\W$ to $\Sigma$. 
\end{proposition}
\begin{proof} Given $p\in \Sigma$ and $\{E_i\}_{i=1}^m$ a local orthonormal frame of $T\Sigma$ with $E_i(p)=e_i$,

\begin{equation}\label{def2HW}
\begin{aligned}
    m\,\Vec{H}_{\mathcal{W}}(p)=&\sum_{i=1}^m \left( \W(B(e_i,e_i))+(\nabla\mathcal{W})(e_i,e_i)\right)\\
    =& \sum_{i=1}^m \left(\W((\nabla_{e_i} E_i)^\bot)+ \nabla_{e_i} \W(E_i) -\W(\nabla_{e_i} E_i)\right)\\
    =& \sum_{i=1}^m \left( B(e_i,\W(e_i)) +\nabla^\Sigma_{e_i} \W(E_i)-\W(\nabla^\Sigma_{e_i} E_i)\right),
\end{aligned}
\end{equation}
and the result follows.
\end{proof}

\begin{remark}\label{independence_of_extension} We must remark here that given an immersed submanifold $\Sigma\subset M$ endowed with a conductivity $\mathcal{W}^\Sigma$, and given two independent extensions $\mathcal{W}_1, \mathcal{W}_2$ of $\mathcal{W}^\Sigma$ to $M$, the definition of $\W$-mean curvature vector given in Definition \ref{defWH} does not depend on the particular choice of $\mathcal{W}_1$ or $\mathcal{W}_2$ since $\vec{H}_{\mathcal{W}}$ can be written only in terms of $\mathcal{W}^\Sigma$ as it is proved in the above proposition.
\end{remark}

\begin{remark} If $\W={\rm Id}$, then $\Vec{H}_\W$ is the usual Riemannian mean curvature vector of $\Sigma$. In this vein, if $\W$ on the submanifold is just divergence-free, i.e.  ${\rm div}^\Sigma (\W)=0$, then $\Vec{H}_\W=\frac{1}{m}\sum_{i=1}^m  B(e_i,\W(e_i))$. This last vector, that is orthogonal to $\Sigma$, is used in \cite{Filho2023,grosjean2004a} and is called the generalized $\W$ mean curvature of $\Sigma$.\\

On the other hand, if $\Sigma^m$ is a $\W$-compatible submanifolds of $(M^n,g,\W)$ with $\W= e^f\cdot {\rm Id}$, then ${\rm tr}_\Sigma (\nabla \W)=e^f\,\nabla^\Sigma f$. Therefore,
\begin{equation}\label{rel:HW_weighted}
m\,\Vec{H}_\W=e^f(m\,\Vec{H}+\nabla^\Sigma f)= e^f(\Vec{H}_f+\nabla f),
\end{equation}
where $\Vec{H}_f=m \Vec{H}-(\nabla f)^\bot$ is the weighted mean curvature of $\Sigma$ in the weighted manifold $(M,g,e^f)$.
\end{remark}
\begin{remark} In order to obtain that the $\W$-Laplacian of the distance function coincides with the $\W$-mean curvature of the metric spheres (as it happens in the Riemannian and weighted settings), the definition of the $\W$-mean curvature vector $\Vec{H}_\W$ should be:
\begin{equation}\label{meanc}
m\,\Vec{H}_\W= \sum_{i=1}^m B(e_i,\mathcal{W}(e_i))+({\rm div}(\mathcal{W}))^\bot.
\end{equation}
But if we consider a conductivity $\W^\Sigma$ defined only on the submanifold $\Sigma$, then the expression \eqref{meanc}, depends on the particular extension of $\W^\Sigma$ to $M$. Notice that for a divergence-free conductivity, equation \eqref{meanc} coincides with our definition.
\end{remark}

\subsection{Bounding the $\W$-Laplacian of extrinsic modified distance functions} \label{subsec:HessianComp}
With the previous definition we can state the following expression for the $\W$-Laplacian of extrinsic modified radial functions $v=F\circ r$ restricted to $\Sigma$. {{As in Remark \ref{independence_of_extension}, if the submanifold $\Sigma\subset M$ is endowed with a conductivity $\mathcal{W}^\Sigma$, the following results are independent of the extension of $\W^\Sigma$ to $M$.}}

\begin{proposition}\label{prop:dua}
    Let  $(M^n,g,\mathcal{W})$  be a conductive Riemannian manifold, $\Sigma^m\subset M$ a $\mathcal{W}$-compatible immersed submanifold, and $r: M\to \mathbb{R}$ the distance from a point $o \in M$. Then, for any smooth function $F:(0,+\infty)\to \mathbb{R}$ and $p\in \Sigma-({\rm cut}(o)\cup\{o\})$ we have that
    $$
\begin{aligned}
    \Delta_\mathcal{W}^\Sigma (F\circ r)=&F''(r)\langle\mathcal{W}(\nabla^\Sigma r),\nabla^\Sigma r\rangle+F'(r)\sum_{i=1}^m{\rm Hess}\, r(e_i,\mathcal{W}(e_i))\\
    & +F'(r)\langle \nabla r,m\,\Vec{H}_\W)\rangle,
\end{aligned}
    $$
    where $\{e_i\}_{i=1}^m$ is an orthonormal basis of $T_p\Sigma$ and $\Vec{H}_\W$ is the $\W$-mean curvature vector field of $\Sigma$.
\end{proposition}
\begin{proof} Let $p\in \Sigma- \{o\}$ and $\{E_i\}_{i=1}^m$ a local orthonormal frame of $T\Sigma$ with $E_i(p)=e_i$,
$$
        \begin{aligned}
            \Delta_\mathcal{W}^\Sigma (F\circ r)=&\sum_{i=1}^m\langle \nabla^\Sigma_{e_i}\mathcal{W}(F'(r)\nabla^\Sigma r), e_i\rangle=\sum_{i=1}^m\langle \nabla^\Sigma_{e_i}(F'(r)\mathcal{W}(\nabla^\Sigma r)), e_i\rangle\\
            =&F''(r)\sum_{i=1}^m\langle \nabla^\Sigma r, e_i\rangle\langle \mathcal{W}(\nabla^\Sigma r), e_i\rangle+F'(r)\sum_{i=1}^m\langle \nabla^\Sigma_{e_i}\mathcal{W}(\nabla^\Sigma r), e_i\rangle\\
            =&F''(r)\langle\mathcal{W}(\nabla^\Sigma r),\nabla^\Sigma r\rangle+F'(r)\left(\sum_{i=1}^m\langle \nabla_{e_i}\mathcal{W}(\nabla r),e_i\rangle\right.\\
            &+\left.\sum_{i=1}^m\langle \mathcal{W}(\nabla r),B(e_i,e_i)\rangle\right),
        \end{aligned}
  $$
  where we have used that $\nabla^\Sigma r=\nabla r -(\nabla r)^\bot$, and that
  $$\langle \nabla_{e_i}\mathcal{W}((\nabla r)^\bot), e_i\rangle=-\langle \mathcal{W}((\nabla r)^\bot),\nabla_{e_i} E_i\rangle=-\langle \mathcal{W}(\nabla r),(\nabla_{e_i} E_i)^\bot\rangle.$$
  
 Moreover, 
    $$
    \begin{aligned}
\langle \nabla_{e_i}\mathcal{W}(\nabla r),e_i\rangle=&E_i\langle\mathcal{W}(\nabla r),E_i\rangle-\langle \mathcal{W}(\nabla r),\nabla_{e_i} E_i\rangle\\
=&E_i\langle\nabla r,\mathcal{W}(E_i)\rangle-\langle \nabla r,\mathcal{W}(\nabla_{e_i} E_i)\rangle\\
=&\langle \nabla_{e_i}\nabla r,\mathcal{W}(e_i)\rangle+\langle \nabla r,\nabla_{e_i}\mathcal{W}(E_i)\rangle\\
&-\langle \nabla r,\mathcal{W}(\nabla_{e_i} E_i)\rangle\\
=&\langle \nabla_{e_i}\nabla r,\mathcal{W}(e_i)\rangle+\langle \nabla r,\nabla_{e_i}\mathcal{W}(E_i)-\mathcal{W}(\nabla_{e_i} E_i)\rangle\\
=&\langle \nabla_{e_i}\nabla r,\mathcal{W}(e_i)\rangle+\langle \nabla r,\nabla\mathcal{W}(e_i,e_i)\rangle,
    \end{aligned}
    $$
    and the result then follows by Definition \ref{defWH}.
 \end{proof}

 \begin{theorem}\label{mainteo}Let $(M,g,\mathcal{W})$ be a conductive Riemannian manifold, $\Sigma^m \subset M$ a $\W$-compatible immersed submanifold, $r:M\to \mathbb{R}$ the distance function from a point $o\in M$, and $w(s)$ a smooth function such that $w(0)=0$, $w'(0)=1$ and $w(s)>0$ for all $s>0$. Suppose that for any $p\in \Sigma-({\rm cut}(o) \cup \{o\})$ and any plane $\sigma_p\subset T_p M$ containing $(\nabla r)_p$,
$$
{\rm sec}(\sigma_p)\leq (\geq) -\frac{w''(r(p))}{w(r(p))}.
$$
 Then, for every smooth function $F:(0,+\infty)\to \mathbb{R}$ with $F'\leq 0$ we have that the restriction of $F\circ r$ to $\Sigma$ satisfies the following corresponding inequalities:
$$
\begin{aligned}
    \Delta_\mathcal{W}^\Sigma (F\circ r)\leq (\geq)&\left(F''(r)-F'(r)\frac{w'(r)}{w(r)}\right) \langle\mathcal{W}(\nabla^\Sigma r),\nabla^\Sigma r\rangle+{\rm tr}_\Sigma (\mathcal{W})\, F'(r)\frac{w'(r)}{w(r)}\\
    & +m\,F'(r)\langle \Vec{H}_\mathcal{W},\nabla r\rangle,
\end{aligned}
$$
at every point in $\Sigma-({\rm cut}(o) \cup \{o\})$.
\end{theorem}
\begin{proof} Let $p \in \Sigma-({\rm cut}(o) \cup \{o\})$ and $\{e_i\}_{i=1}^m$ an orthonormal frame of $T_p \Sigma$ such that $\W(e_i)=\lambda_i\, e_i$. Then, as in the proof of Proposition \ref{ineq1intrinsic}, we obtain that

\begin{equation}
 \sum_{i=1}^m{\rm Hess}\, r(e_i,\mathcal{W}(e_i))\geq (\leq) \frac{w'(r)}{w(r)}\left({\rm tr}_\Sigma (\mathcal{W})- \langle \mathcal{W}(\nabla^\Sigma r),\nabla^\Sigma r\rangle \right).
\end{equation}

This expression, in combination with Proposition \ref{prop:dua}, gives us the result.
\end{proof}

\subsection{$\W$-parabolicity or $\W$-hyperbolicity of submanifolds}

In this subsection, given a conductive Riemannian manifold $(M^n,g,\W)$ with a pole $o$, and a noncompact properly immersed $\W$-compatible submanifold $\Sigma^m\subset M$, we derive bounds for the $\W$-capacity of extrinsic balls $D_R$ of radius $R$ using the inequalities for the conductivity Laplacian of modified distance functions given in Theorem \ref{mainteo}. As a consequence, we obtain $\W$-parabolicity and $\W$-hyperbolicity criteria for such submanifolds $\Sigma^m$.

\begin{theorem} \label{thm:extrinsic2}  Let $(M^n,g,\mathcal{W})$ be a conductive Riemannian manifold, $\Sigma^m$ a noncompact $\mathcal{W}$-compatible submanifold properly immersed in $M$, $r:M\to \mathbb{R}$ the distance function from the pole $o$ in $M$, and $w(s)$ a smooth function such that $w(0)=0$, $w'(0)=1$ and $w(s)>0$ for all $s>0$. Suppose that the following conditions are fulfilled:

\begin{itemize}

\item[(a) ] For any $p\in M- \{o\}$ and any plane $\sigma_p \subset T_p M$ containing $(\nabla r)_p$,
\begin{displaymath}
{\rm sec}(\sigma_p) \leq (\geq) - \frac{w''(r(p))}{w(r(p))}.
\end{displaymath}

\item[(b) ] There exist numbers $\rho,\,q>0$, and a continuous function $\theta:[\rho,+\infty)\to \mathbb{R}$ such that $\partial D_\rho$ is smooth and in $\Sigma - D_\rho$:
\begin{enumerate}
\item  $w'(r)\bigg({\rm tr}_\Sigma(\mathcal{W})- q\,\langle\mathcal{W}(\nabla^\Sigma r), \nabla^\Sigma r\rangle\bigg) \geq (\leq)\,\, 0$,
\item $m\langle\Vec{H}_{\mathcal{W}}, \nabla r\rangle \geq (\leq)\,\,\theta(r)\,\langle\mathcal{W}(\nabla^\Sigma r),\nabla^\Sigma r\rangle$.
\end{enumerate}
\end{itemize}

Then for all $R> \rho$ with $\partial D_R$ smooth,
\begin{equation}\label{comp_q_ext}
{\rm Cap}^\Sigma_\mathcal{W}(D_\rho,D_R)\geq (\leq) 
-\phi'_{\rho,R,h,q}(\rho)\int_{\partial D_\rho} \langle \mathcal{W}(\nabla^\Sigma r), \frac{\nabla^\Sigma r}{\Vert\nabla^\Sigma r\Vert}\rangle\,d\mu_{\partial D_\rho},
\end{equation}
\noindent where $\phi_{\rho,R,h,q}$ is the function given in equation \eqref{fsolmodel} for $n=q$:

\begin{equation}
\phi_{\rho,R,h,q}(r):=\left(\int_r^R w^{1-q}
(s)\,e^{-h(s)}\,ds\right)\,\left(\int_\rho^R w^{1-q}
(s)\,e^{-h(s)}\,ds\right)^{-1},
\end{equation}
with $h(t)=\int_{\rho}^t\theta(s)\,ds$ for $t\geq \rho$.\\
Moreover, if

\begin{equation}\label{modinfty}
\int_\rho^\infty w^{1-q}(s)\,{\rm e}^{-h(s)}ds < (=)\infty,
\end{equation}

then $\Sigma$ is $\mathcal{W}$-hyperbolic (resp. $\W$-parabolic).
\end{theorem}

\begin{proof}
Given $q \in \mathbb{R}_+$, let $\phi_{\rho,R,h,q}$ be the function given in equation \eqref{fsolmodel} for $n=q$, with $h(t)=\int_{\rho}^t\theta(s)ds$ for any $t\geq \rho$. Consider the transplanted radial function $v:=\phi_{\rho,R,h,q}\circ r$ defined on the extrinsic annulus $D_R- \widebar{D}_\rho \subseteq \Sigma$. For $r\geq \rho$, we have that  $\phi_{\rho,R,h,q}'(r)\leq 0$ and 
\begin{equation}\label{v''ext}
\begin{aligned}
\phi_{\rho,R,h,q}''(r)&=\phi_{\rho,R,h,q}'(r)\bigg(-\theta(r)+(1-q)\frac{w'(r)}{w(r)}\bigg),\\
\end{aligned}
\end{equation}
since $\phi_{\rho,R,h,q}$ is solution of the differential equation \eqref{eq:sabika}. We apply Theorem \ref{mainteo}, equation \eqref{v''ext} and hypotheses (a) and (b) to obtain

\begin{eqnarray*}
\begin{aligned}
\Delta^\Sigma_\mathcal{W}\,v&\leq (\geq) \bigg(\phi_{\rho,R,h,q}''(r)-\phi_{\rho,R,h,q}'(r)\frac{w'(r)}{w(r)}\bigg)\langle\mathcal{W}(\nabla^\Sigma r),\nabla^\Sigma r\rangle\\
&+{\rm tr}_\Sigma(\mathcal{W})\, \phi_{\rho,R,h,q}'(r)\frac{w'(r)}{w(r)}+m \,\phi_{\rho,R,h,q}'(r) \langle\Vec{H}_{\mathcal{W}}, \nabla r\rangle\\
&=\phi_{\rho,R,h,q}'(r)\frac{w'(r)}{w(r)}\bigg({\rm tr}_\Sigma(\mathcal{W}) -q\,\langle\mathcal{W}(\nabla^\Sigma r),\nabla^\Sigma r\rangle\bigg)+\\&\phi_{\rho,R,h,q}'(r)\bigg(m \langle\Vec{H}_{\mathcal{W}}, \nabla r\rangle-\theta(r)\langle\mathcal{W}(\nabla^\Sigma r),\nabla^\Sigma r\rangle\bigg) \leq (\geq)\, 0=\Delta^\Sigma_\mathcal{W}\,u,
\end{aligned}
\end{eqnarray*}
where $u$ is the solution to the Dirichlet problem \eqref{eqWLap} in $D_R -\widebar{D}_\rho$.\\ 
We apply Theorem \ref{Hopf} to the $\mathcal{W}$-subharmonic $-v$ (resp. $v$) and the $\W$-harmonic function $-u$ (resp. $u$) to obtain that $u\leq (\geq) v$ on $D_R-\widebar{D}_\rho$, and that 
\begin{equation}\label{eqgrad2}
\langle \mathcal{W}(\nabla^\Sigma v), n_{\partial D_\rho}\rangle \leq (\geq)\langle \mathcal{W}(\nabla^\Sigma u), n_{\partial D_\rho}\rangle,
\end{equation}
\noindent where $n_{\partial D_\rho}$ denotes the unit outward, (pointing outward $D_R-\widebar{D}_\rho$), normal vector along $\partial D_\rho$. Hence, as $v$ is non-increasing with respect $r$ and $v=1$ on $\partial D_\rho$, $n_{\partial D_\rho}=\nabla^\Sigma v/\Vert\nabla^\Sigma v\Vert$, and by the definition of $\mathcal{W}$-capacity, 
\begin{eqnarray*}
{\rm Cap}^\Sigma_\mathcal{W}(D_\rho, D_R)&=&\int_{\partial D_\rho} \langle \mathcal{W}(\nabla^\Sigma u), n_{\partial D_\rho}\rangle\, d\mu_{\partial D_\rho}\\
&\geq(\leq)& \int_{\partial D_\rho} \langle\mathcal{W}(\nabla^\Sigma v), \frac{\nabla^\Sigma v}{\Vert\nabla^\Sigma v\Vert}\rangle\,d\mu_{\partial D_\rho}\\
&=&-\phi'_{\rho,R,h,q}(\rho)\int_{\partial D_\rho} \langle \mathcal{W}(\nabla^{\Sigma} r), \frac{\nabla^{\Sigma} r}{\Vert\nabla^{\Sigma} r\Vert}\rangle\,d\mu_{\partial D_\rho}.\\
\end{eqnarray*}
By using Sard's Theorem we can suppose that $\Vert \nabla^\Sigma r\Vert \neq 0$ on $\partial D_\rho$, and for all $p\in \partial D_\rho$ we have that,
\begin{equation}\label{boundsWnablar}
 0<\mu(p)\,\Vert\nabla^{\Sigma} r\Vert \leq \langle \mathcal{W}(\nabla^{\Sigma} r), \frac{\nabla^{\Sigma} r}{\Vert\nabla^{\Sigma} r\Vert}\rangle \leq \kappa(p)\,\Vert\nabla^{\Sigma} r\Vert,
\end{equation}
for some positive functions $\mu$ and $\kappa$ in {$C^\infty(M)$}. Then, since $\partial D_\rho$ is compact, 
\begin{equation}
\int_{\partial D_\rho} \langle \mathcal{W}(\nabla^{\Sigma} r), \frac{\nabla^{\Sigma} r}{\Vert\nabla^{\Sigma} r\Vert}\rangle\,d\mu_{\partial D_\rho}>0.
\end{equation}
Taking limits when $R\to +\infty$, if \eqref{modinfty} is fulfilled then, we have that
\begin{equation}
    {\rm Cap}^\Sigma_\mathcal{W}(D_\rho)> (=)\,0,
\end{equation}
and $\Sigma$ is $\mathcal{W}$-hyperbolic (resp. $\W$-parabolic).
\end{proof}

\begin{remark}\label{qN}
When $q \in \mathbb{N}$, we consider the $q$-dimensional weighted model space $(M^q_w,g_w,e^{h(r_w)})$ with $h(t)=\int_{\rho}^t\theta(s)\,ds$ for $t\geq \rho$. Then, using \eqref{fcapacitymodel}, equation \eqref{comp_q_ext} reads 
\begin{equation} \label{eq:CapComp1ext}
\begin{aligned}
{\rm Cap}^\Sigma_\mathcal{W}(D_\rho,D_R)\geq (\leq) \frac{{\C}^h(B_\rho, B_R)}{{\rm Vol}_h(\partial B^w_\rho)} \int_{\partial D_\rho} \langle \mathcal{W}(\nabla^\Sigma r), \frac{\nabla^\Sigma r}{\Vert\nabla^\Sigma r\Vert}\rangle\,d\mu_{\partial D_\rho}.
\end{aligned}
\end{equation}

Moreover, if  $M^q_w$ is $h$-hyperbolic (resp. $h$-parabolic),
 then $\Sigma^m$ is $\mathcal{W}$-hyperbolic (resp. $\W$-parabolic).\\

 Using the inequalities in \eqref{boundsWnablar} and the fact that $\Vert\nabla^\Sigma r\Vert \leq 1$, the hypotheses in Theorem \ref{thm:extrinsic2} implies that,

 \begin{equation} \label{eq:CapComp1extmod}
\begin{aligned}
{\rm Cap}^\Sigma_\mathcal{W}(D_\rho,D_R)\geq &(\leq) \,C(\rho)\,\frac{{\C}^h(B_\rho, B_R)}{{\rm Vol}_h(\partial B^w_\rho)} \int_{\partial D_\rho} \Vert\nabla^\Sigma r\Vert\,d\mu_{\partial D_\rho}\\
&(\leq) \,C(\rho)\,\frac{{\C}^h(B_\rho, B_R)\,{\rm Vol}(\partial D_\rho)}{{\rm Vol}_h(\partial B^w_\rho)},
\end{aligned}
\end{equation}

where $C(\rho)={\rm min}\{\mu(p):\,p\in \partial D_\rho\}$ (resp. $C(\rho)={\rm max}\{\kappa(p):\,p\in \partial D_\rho\}$).
\end{remark}
\medskip

{As in the intrinsic case, using that for all $p\in \Sigma$ and all $v\in T_p \Sigma$, $\mu(p)\,\Vert v\Vert^2 \leq \langle \W(v),v \rangle\leq \kappa(p)\,\Vert v\Vert^2$ for some smooth positive functions $\mu$ and $\kappa$, we can state the following results}.

\begin{theorem}\label{thm:extrinsicr} Let $(M,g,\mathcal{W})$ be a conductive Riemannian manifold, $\Sigma^m$ a non-compact $\mathcal{W}$-compatible submanifold properly immersed in $M$, $r:M\to \mathbb{R}$ the distance function from the pole $o$ in $M$, and $w(s)$ a smooth function such that $w(0)=0$, $w'(0)=1$ and $w(s)>0$ for all $s>0$. Let $\kappa$ be as in Definition \ref{defConducRiem} and suppose that the following conditions are fulfilled:

\begin{itemize}

\item[a) ] For any $p\in M-\{o\}$ and any plane $\sigma_p \subset T_p M$ containing $(\nabla r)_p$,
\begin{displaymath}
{\rm sec}(\sigma_p) \geq (\leq) - \frac{w''(r(p))}{w(r(p))}.
\end{displaymath}

\item[b) ] There exists $\rho>0$ and continuous function $\theta:[\rho,+\infty)\to \mathbb{R}$ such that $\partial D_\rho$ is smooth and for all $p\in \Sigma - D_\rho$:
\begin{enumerate}

\item m\,$\langle \nabla r, \Vec{H}_\W \rangle \leq (\geq) \theta(r(p))\, \kappa(p)$ and $w(r(p))'\geq (\leq) \,\,0$,
\item $m\, w'(r(p))/w(r(p))+\theta(r(p)) \leq (\geq)\,\,0$ (balance condition).
\end{enumerate}
\end{itemize}

Then for all $R> \rho$ with $\partial D_R$ smooth,
\begin{equation} \label{eq:CapCompextrinsicr}
\begin{aligned}
{\rm Cap}^\Sigma_\mathcal{W}(D_\rho,D_R)\leq &(\geq) \,C(\rho)\,\frac{{\C}^h(B^w_\rho, B^w_R)}{{\rm Vol}_h(\partial B^w_\rho)} \int_{\partial D_\rho} \Vert\nabla^\Sigma r\Vert\,d\mu_{\partial D_\rho}\\
\leq & \,C(\rho)\,\frac{{\C}^h(B^w_\rho, B^w_R)\,{\rm Vol}(\partial D_\rho)}{{\rm Vol}_h(\partial B^w_\rho)},
\end{aligned}
\end{equation}
where $C(\rho)={\rm max}\{\kappa(p):\,p\in \partial D_\rho\}$ (resp. $C(\rho)={\rm min}\{\mu(p):\,p\in \partial D_\rho\}$), ${\rm Cap}^h(B^w_\rho,B^w_R)$ is the weighted capacity of a metric annulus in the weighted model space $(M^m_w, g_w,e^{h(r_w)})$ with $h(t)=\int_{\rho}^t \theta(s)\,ds$ for any $t\geq \rho$, and ${\rm Vol}_h$ is the corresponding weighted volume in $M_w^m$.\\

Moreover, {since $\int_\rho^\infty w^{1-m}(s)\,e^{-h(s)}ds =\infty$, then $\Sigma$ is $\mathcal{W}$-parabolic (resp. if $\int_\rho^\infty w^{1-m}(s)\,e^{-h(s)}ds <\infty$, then $\Sigma$ is $\mathcal{W}$-hyperbolic)}.

\end{theorem}

\begin{proof}
Let $\phi_{\rho,R,h,m}$ be the function given in \eqref{fsolmodel}, with $n=m$. Consider the radial function $v:=\phi_{\rho,R,h,m}\circ r$ in $\widebar{D}_R- D_\rho$. For $p\in \Sigma-D_\rho$, since $\phi_{\rho,R,h,m}'\leq 0$ and 
$$\phi_{\rho,R,h,m}''(r)-\phi_{\rho,R,h,m}'(r)\frac{w'(r)}{w(r)}=-\phi_{\rho,R,h,m}'(r)\left(m\frac{w'(r)}{w(r)}+\theta\right)\leq (\geq)\,\,0,$$
by the balance condition, taking into account that $w'\geq 0$, ${\rm tr}_\Sigma(\mathcal{W})\leq  m\, \kappa(p)$, $\langle\mathcal{W}(\nabla^\Sigma r),\nabla^\Sigma r\rangle \leq \kappa(p)\,\Vert\nabla^\Sigma r\Vert^2$, and $\Vert\nabla^\Sigma r\Vert \leq 1$, we obtain from Proposition \ref{mainteo} that

\begin{eqnarray*}
\Delta^\Sigma_\mathcal{W}\,v &\geq (\leq)& \kappa\left(\phi_{\rho,R,h,m}''(r)+(m-1) \phi_{\rho,R,h,m}'(r)\frac{w'(r)}{w(r)}\right) +\phi_{\rho,R,h,m}'(r)\, \theta(r)\,\kappa\\
&=& \kappa\left(\phi_{\rho,R,h,n}''(r)+(m-1) \phi_{\rho,R,h,m}'(r)\,\frac{w'(r)}{w(r)} + \phi_{\rho,R,h,m}'(r)\, \theta(r)\right)\\
&=& \kappa \,\Delta^h_w \,\phi_{\rho,R,h,m}=0=\Delta^\Sigma_\mathcal{W}\,u,
\end{eqnarray*}
where $u$ is the solution to the Dirichlet problem \eqref{eqWLap} in $D_R- \widebar{D}_\rho$. The proof follows the lines of the proof of Theorems \ref{thm:extrinsic2} and \ref{propIntrinsicA2}, taking into account Remark \ref{qN}. \\

\end{proof}
\begin{remark} Under the hypothesis corresponding to $
{\rm sec}(\sigma_p) \geq - \frac{w''(r(p))}{w(r(p))}$ we can deduce that
\begin{equation*}
\frac{{\rm Cap}^\Sigma_\mathcal{W}(D_\rho,D_R)}{C(\rho)\,{\rm Vol}(\partial D_\rho)}\leq \frac{{\C}^h(B^w_\rho, B^w_R)\,}{{\rm Vol}_h(\partial B^w_\rho)}.
\end{equation*}

On the other hand, under the hypothesis corresponding to $
{\rm sec}(\sigma_p) \leq - \frac{w''(r(p))}{w(r(p))}$, if $\int_\rho^\infty w(t)\,dt$ is finite, then $\int_\rho^\infty w^{1-m}(s)\,{\rm e}^{-h(s)}ds$ is also finite and $\Sigma$ is $\W$-hyperbolic.  
\end{remark}

\begin{remark} For the isotropic case, $\W=e^f\cdot{\rm Id}$, the functions $\mu$ and $\kappa$ in Definition \ref{defConducRiem} are given by $\mu(p)=\kappa(p)= e^{f(p)}$, for all $p\in M$, and 

\begin{equation*}\label{rel:HW_weighted}
m\,\Vec{H}_\W= e^f(\Vec{H}_f+\nabla f),
\end{equation*}
where $\Vec{H}_f=m \Vec{H}-(\nabla f)^\bot$ is the weighted mean curvature of $\Sigma$ in the weighted manifold $(M,g,{\rm e}^f)$. In this situation, the hypothesis on the sign of $w'$ is not needed, and we recover Theorem 5.7 in \cite{hpr2020a}.
\end{remark}
\section{Hyperbolicity from conductivities with bounded variation} \label{sec:BV}
{
In the previous results, we assume bounds on the trace of the conductivity $\W$, that is, bounds on the mean value of its eigenvalues. In the same vein, we can obtain a $\W$-parabolicity criterion using also the coefficient of variation ${\rm cv}(\W)$ of these eigenvalues.
We first recall the definitions of these statistical notions mentioned above when applied to the eigenvalues of the matrix $\mathcal{W}$:}
\begin{definition}
    Let $(M^n,g,\mathcal{W})$ be a complete conductive Riemannian manifold, $p$ a point in $M$, and $\{e_i\}_{i=1}^n$ an orthonormal basis of $T_pM$. Then \emph{the mean value $\overline{\lambda}_p(\mathcal{W})$ of $\mathcal{W}$ at $p$} is given by 
    $$
\overline{\lambda}_p(\mathcal{W}):=\frac{1}{n}\sum_{i=1}^n \langle e_i , \mathcal{W}(e_i)\rangle=\frac{1}{n}{\rm tr}(\mathcal{W}).
    $$
    The \emph{standard deviation ${\rm sd}_p(\mathcal{W})$ of $\mathcal{W}$ at $p$} is thence
    $$
{\rm sd}_p(\mathcal{W}):=\left(\frac{1}{n}{\rm tr}(\mathcal{W}^2)-\left(\frac{1}{n}{\rm tr}(\mathcal{W})\right)^2\right)^{\frac{1}{2}}.
    $$
    \emph{The coefficient of variation  ${\rm cv}_p(\mathcal{W})$ of $\mathcal{W}$ at $p$} is correspondingly
    $$
{\rm cv}_p(\mathcal{W}):=\frac{{\rm sd}_p(\mathcal{W})}{\overline{\lambda}_p(\mathcal{W})}.
    $$
\end{definition}
\begin{remark}
    Observe that, if $\{e_i\}_{i=1}^n$ is an orthonormal basis of $T_p M$ that diagonalizes $\W$, namely
    $$
\mathcal{W}(e_i)=\lambda_i e_i,
    $$
    then the mean value $\overline{\lambda}_p$ of $\mathcal{W}$ at $p$ and standard deviation ${\rm sd}_p$ of $\mathcal{W}$ at $p$ are nothing else than the classical mean and the classical standard deviation of the list of eigenvalues $\{\lambda_1, \cdots,\lambda_n\}$. Indeed,
    $$
 \overline{\lambda}_p(\mathcal{W})=\frac{1}{n}\sum_{i=1}^n\lambda_i,\quad {\rm sd}_p(\mathcal{W})=\sqrt{\frac{1}{n}\sum_{i=1}^n\left(\lambda_i-\overline{\lambda}_p\right)^2}.
    $$

    For the isotropic case, $\W=e^f\cdot {\rm Id}$, $\lambda_i=e^f$ for all $i=1,\ldots,n$. Then, $\overline{\lambda}_p(\mathcal{W})={\rm e}^f$ and ${\rm sd}_p(\mathcal{W})=0$. {So the coefficient of variation of $\W$ tells us how far the conductivity is from being isotropic}.
\end{remark}

\begin{theorem}\label{teo:93}{
Let $(M^n,g,\mathcal{W})$ be a conductive simply-connected complete Riemannian manifold with sectional curvatures bounded from above by ${\rm sec}\leq 0$ (namely, a conductive Cartan-Hadamard manifold). Let $\Sigma\subset M$ be a $\mathcal{W}$-compatible immersed $m$-dimensional submanifold. Suppose moreover that
$$
{\rm CV}{(\Sigma)}:= \sup_{p\in {\Sigma}}{\rm cv}_p\left({\mathcal{W}^\Sigma}\right)< \frac{m-2}{2\sqrt{m}},$$
and that
$
\begin{aligned}
  \langle\nabla r, \vec{H}_\mathcal{W}\rangle \leq 0
\end{aligned}
$,
\noindent where $r: \Sigma \rightarrow \mathbb{R}$ is the distance function in $\Sigma$ to some point in $M$. Then $\Sigma$ is $\W^{\Sigma}$-hyperbolic.}
\end{theorem}
{
\begin{remark}
  Observe that in the isotropic case $\W=Id$,  this is a natural generalization of the results in \cite{GAFA} when we consider proper and minimal immersions $\vec{H}=0$. Moreover, this theorem can be applied to self-expanders $X:S\to \mathbf{R}^n$ of the Mean Curvature flow because thence
  $$
\langle X, \vec{H}\rangle=\langle r\nabla r, \vec{H}\rangle=\lambda \Vert X^\perp\Vert^2\leq 0.
  $$
  Hence Theorem \ref{teo:93} is an extension of the results of \cite{Soliton2021}. Observe that moreover by \cite{alarcon} we know that at least for minimal surfaces in Euclidean space non-properness is a generic property.
\end{remark}}

\begin{proof} Suppose on the contrary that $\Sigma$ is $\W$-parabolic and let us consider a point $p\in \Sigma$. Since $M$ is a Cartan-Hadamard manifold then $p$ is a pole of $M$. We shall apply the $\mathcal{W}_\Sigma$-minimum parabolic principle ({see \cite{Alias2016}}) to a smooth function $u \in C^\infty(\Sigma)$ such that $u(q)\neq 0\,\forall q \in \Sigma$ and
$$ q\mapsto u(q):=r^{-k}(q)
$$
outside a compact set of $\Sigma$ containing $p$. Here $r$ is the distance function to the pole $p$
and $k$ is some positive number such that
$$
0<k<\frac{m}{\sqrt{m}{\rm CV}(\Sigma)+1}-2.
$$
Notice that the right hand side is positive by assumption, so that such a $k$ exists.
The construction of this smooth function on $\Sigma$ will be constructed in a similar way as in  \cite{Soliton2021}:

\begin{lemma} Given a point $p \in \Sigma$, and $r: \Sigma \rightarrow \mathbb{R}$ the distance function to the point $p$, there is $u: \Sigma \rightarrow \mathbb{R}$ such that $u \in C^\infty(\Sigma)$ and such that $u(q) \neq 0\,\,\forall q \in \Sigma$ and
$$
 q\mapsto u(q):=r^{-k}(q)
$$
outside a compact set of $\Sigma$ containing $p$. 
\end{lemma}
\begin{proof}
Let us assume first that $\text{sup}_\Sigma r =\infty$. As $r :\Sigma \rightarrow \mathbb{R}$ is a continuous function, then the set
$$A:=\{q \in \Sigma:r(q) \leq 1\} \cup \{q \in \Sigma :2 \leq r(q) \leq 3\}$$
is a closed set in such a way that, for a fixed $\epsilon >0$, the function 
$$v_1: A \rightarrow \left(\frac{1}{3^k}-\epsilon, 1+\epsilon\right)$$
\noindent defined as
$$v_1(q):=\begin{cases}
    1\,\,\,\, \text{if}\,\,\, r(q) \leq 1\\
    \frac{1}{r^k(q)}\,\,\,\text{if}\,\,\, 2 \leq r(q) \leq 3 
\end{cases}$$

is a smooth function $v_1\in C^\infty(A)$.

Let us consider now, for a fixed $\delta >0$, the open set in $\Sigma$ defined as 
$$N:=\{q \in \Sigma : r(q) <3+\delta\}.$$

It is obvious that $A \subseteq N$. Then, we define the function $v_2: N \rightarrow M$ as:
$$v_2(q):=\begin{cases}
    1\,\,\,\, \text{if}\,\,\, r(q) \leq 1\\
    2-\frac{1}{2^k}+(\frac{1}{2^k}-1)r(q) \,\,\,\, \text{if}\,\,\, 1 \leq r(q) \leq 2\\
    \frac{1}{r^k(q)}\,\,\,\text{if}\,\,\, 2 \leq r(q) \leq 3 
\end{cases}.$$

This function is a continuous extension of $v_1$, namely, $v_2\vert_A=v_1$ and hence, applying the Extension lemma, (see Corollary 6.27 in \cite{Lee}), we have that there is 
$$v_3: N \rightarrow M$$ a smooth extension of $v_2$, so we can define
$$u:\Sigma \rightarrow \mathbb{R}$$ as
$$u(q):=\begin{cases}
    v_3(q)\,\,\,\, \text{if}\,\,\, r(q) \leq 3\\
    \frac{1}{r^k(q)}\,\,\,\text{if}\,\,\, r(q)> 3 
\end{cases},$$
obtaining finally a smooth function such that $u(q) \neq 0\,\,\forall q \in \Sigma$ and that
$$u:\Sigma\to \mathbb{R}, \quad q\mapsto u(q):=r^{-k}(q)
$$
outside a compact set of $\Sigma$ containing $p$. 

If we assume that $\text{sup}_\Sigma r =R <\infty$,
then we follow exactly the same procedure, starting with the closed set
$$A:=\left\{q \in \Sigma:r(q) \leq \frac{R}{4} \right\} \cup\left \{q \in \Sigma :\frac{R}{3} \leq r(q) \leq \frac{R}{2}\right\}$$
and, given a fixed $\epsilon >0$, with the smooth function defined in $A$
$$w_1: A \rightarrow \left(\frac{3^k}{R^k}-\epsilon, 1+\epsilon\right)$$
\noindent defined as
$$w_1(q):=\begin{cases}
    1\,\,\,\, \text{if}\,\,\, r(q) \leq \frac{R}{4}\\
    \frac{1}{r^k(q)}\,\,\,\text{if}\,\,\, \frac{R}{2} \leq r(q) \leq \frac{R}{3} 
\end{cases}.$$

\end{proof}

Now, given $p \notin \Sigma$, we consider the smooth function $u: \Sigma\to \mathbb{R}$ that we obtain applying the lemma above. Therefore, since 
$$
\inf_\Sigma u=:u_*\geq 0>-\infty,
$$
there exists a sequence of points $\{x_l\}_{l\in \mathbb{N}} \subset \Sigma -\{p\}$ such that
$$
u(x_l)<u_*+\frac{1}{l},\quad \Delta^\Sigma_{\mathcal{W}}u(x_l)>0.
$$
By using Theorem \ref{mainteo},
$$
\begin{aligned}
    0<& r^{k+2}(x_l)\,\Delta^\Sigma_{\mathcal{W}}u(x_l)\leq(k^2+2k)\langle \mathcal{W}(\nabla^\Sigma r),\nabla^\Sigma r\rangle-k\, {\rm tr}_\Sigma (\mathcal{W}^\Sigma)-m\,k\, r(x_l)\langle \nabla r,\Vec{H}_\mathcal{W}(x_l)\rangle\\
    \leq &(k^2+2k)\,\lambda^*\, \Vert\nabla^\Sigma r\Vert^2-mk\, \overline{\lambda}_{x_l}(\mathcal{W}^\Sigma)-m\,k\, r(x_l)\langle \nabla r,\Vec{H}_\mathcal{W}(x_l)\rangle,  
\end{aligned}
$$
where $\lambda^*$ is the greatest of the eigenvalues of $\mathcal{W}^\Sigma$. Taking into account that $\Vert \nabla^\Sigma r\Vert^2 \leq 1 $,
$$
\begin{aligned}
    0<&(k^2+2k)\,\lambda^* -k\, m\,\overline{\lambda}_{x_l}(\mathcal{W}^\Sigma)-m\,k\, r(x_l)\langle \nabla r,\Vec{H}_\mathcal{W}(x_l)\rangle\\
    =& (k^2+2k)(\lambda^*-\overline{\lambda}_{x_l}(\mathcal{W}^\Sigma))+(k^2-(m-2)k)\overline{\lambda}_{x_l}(\mathcal{W}^\Sigma)-m\,k\, r(x_l)\langle \nabla r,\Vec{H}_\mathcal{W}(x_l)\rangle\\
    \leq& (k^2+2k)\sqrt{m}\cdot {\rm sd}_{x_l}(\mathcal{W}^\Sigma)+(k^2-(m-2)k)\overline{\lambda}_{x_l}(\mathcal{W}^\Sigma)-m\,k\, r(x_l)\langle \nabla r,\Vec{H}_\mathcal{W}(x_l)\rangle\\
    =& k\cdot \overline{\lambda}_{x_l}(\mathcal{W}^\Sigma)\left( (k+2)\left(\sqrt{m}\cdot {\rm cv}_{x_l}(\mathcal{W}^\Sigma)+1\right)-m\right)-m\,k\, r(x_l)\langle \nabla r,\Vec{H}_\mathcal{W}(x_l)\rangle\\
    \leq & k\cdot \overline{\lambda}_{x_l}(\mathcal{W}^\Sigma)\left( (k+2)\left(\sqrt{m}\cdot {\rm CV}(\Sigma)+1\right)-m\right)-m\,k\, r(x_l)\langle \nabla r,\Vec{H}_\mathcal{W}(x_l)\rangle.
\end{aligned}
$$
Therefore,
$$
\begin{aligned}
    m\,k\, r(x_l)\langle \nabla r,\Vec{H}_\mathcal{W}(x_l)\rangle<&\overline{\lambda}_{x_l}(\mathcal{W}^\Sigma)\left(m- (k+2)\left(\sqrt{m}\cdot {\rm CV}(\Sigma)+1\right)\right)\\
    =&{\overline{\lambda}_{x_l}(\mathcal{W}^\Sigma)}{\left(\sqrt{m}\cdot {\rm CV}(\Sigma)+1\right)}\left(\frac{m}{\sqrt{m}\cdot {\rm CV}(\Sigma)+1}- 2-k\right).
\end{aligned}
$$
But since we are assuming $
\frac{m}{\sqrt{m}{\rm CV}(\Sigma)+1}-2>k>0
$, we can conclude that 
$$
\begin{aligned}
    \langle r(x_l)\nabla r,\Vec{H}_\mathcal{W}(x_l)\rangle<&0,
\end{aligned}
$$
{and we arrive therefore to a contradiction .}

\end{proof}

As a consequence of the above Theorem \ref{teo:93}, for the intrinsic case we have

\begin{corollary}\label{cor:95}
    Let $(M,g,\mathcal{W})$ be conductive  $n$-dimensional Cartan Hadamard manifold with 
$$
{\rm CV}:= \sup_{p\in M}{\rm cv}_p\left(\mathcal{W}\right)< \frac{n-2}{2\sqrt{n}} .
$$
and
$
\begin{aligned}
{\rm div} \left(\mathcal{W}\right)=0
\end{aligned}
$.
Then $M$ is $\W$-hyperbolic
\end{corollary}

\section{Applications of the main results} \label{sec:Applications}
\subsection{New criteria  for weighted manifolds}

Since the weighted case is a particular case of a conductive manifold (more precisely, an isotropic conductive manifold) we can state the following corollary of Theorem \ref{thm:extrinsic2} for isotropic conductivity:

\begin{corollary}\label{extisotropic2}
Let $(M^n,g,\mathcal{W})$ be a conductive Riemannian manifold, with $\mathcal{W}=e^f\cdot \Id$, $\Sigma^m$ a noncompact $\mathcal{W}$-compatible submanifold properly immersed in $M$, $r:M\to \mathbb{R}$ the distance function from the pole $o$ in $M$, and $w(s)$ a smooth function such that $w(0)=0$, $w'(0)=1$ and $w(s)>0$ for all $s>0$.

Suppose that,

\begin{itemize}

\item[(a) ] For any $p\in M- \{o\}$ and any plane $\sigma_p \subset T_p M$ containing $(\nabla r)_p$,
\begin{displaymath}
{\rm sec}(\sigma_p) \leq (\geq) - \frac{w''(r(p))}{w(r(p))}.
\end{displaymath}

\item[(b) ] There exist $\rho>0$, a real number $q>0$ and a  continuous function $\theta:[\rho,+\infty)\to \mathbb{R}$ such that $\partial D_\rho$ is smooth and in $\Sigma - D_\rho$:

\begin{enumerate}
\item  {{$w'(r)(m-q\Vert \nabla^\Sigma r\Vert^2) \geq (\leq)\,\,0$,}} 
\item $ \langle \Vec{H}_f+\nabla f,\nabla r\rangle \geq \theta(r)\,\Vert\nabla^\Sigma r\Vert^2$,
\noindent where $\Vec{H}_f$ is the weighted mean curvature of $\Sigma$ in the weighted manifold $(M,g,{\rm e}^f)$.
\end{enumerate}
\end{itemize}
Then for all $R> \rho$ with $\partial D_R$ smooth,

\begin{equation*}
{\rm Cap}^\Sigma_\mathcal{W}(D_\rho, D_R)={\rm Cap}^f_\Sigma(D_\rho,D_R)\geq (\leq) -\phi'_{\rho,R,h,q}(\rho)\int_{\partial D_\rho}e^f\, \Vert\nabla^\Sigma r\Vert^2\,d\mu_{\partial D_\rho},
\end{equation*}

\noindent where $\phi_{\rho,R,h,q}$ is the function given  in \eqref{fsolmodel} for $n=q$:
\begin{equation}
\phi_{\rho,R,h,q}(r):=\left(\int_r^R w^{1-q}
(s)\,e^{-h(s)}\,ds\right)\,\left(\int_\rho^R w^{1-q}
(s)\,e^{-h(s)}\,ds\right)^{-1},
\end{equation}
with $h(t)=\int_{\rho}^t\theta(s)\,ds$ for any $t\geq \rho$.

Moreover, if

\begin{equation}\label{cond03}
\int_\rho^\infty w^{1-q}(s)\,e^{-h(s)}ds < (=)\, \infty,
\end{equation}

then $\Sigma$ is  $f$-hyperbolic (resp. $f$-parabolic).

\end{corollary}

\begin{proof}
It suffices to note that if $\W=e^f\cdot\Id$, then ${\rm tr}_\Sigma (\W)=m\,e^f$, $\langle \mathcal{W}(\nabla^\Sigma r), \nabla^\Sigma r\rangle=e^f\, \Vert \nabla^\Sigma r \Vert^2$, ${\rm tr}_\Sigma( \nabla \mathcal{W})=e^f\, \nabla^\Sigma f$, and $m\,\Vec{H}_\W=  e^f\,(\Vec{H}_f+\nabla f)$. 
\end{proof}

\subsection{Equivalent metrics} \label{subsec:EquivMetrics}

The objective of this subsection is to endow the smooth manifold $M$ with a Riemannian metric $G$ in such a way that $(M,g,\mathcal{W})$ is $\mathcal{W}$-hyperbolic if and only if $(M,G)$ is hyperbolic (in the usual homogeneous Riemannian sense).

\begin{theorem}\label{teo:81}{
    Let $(M,g, \mathcal{W})$ be a conductive Riemannian manifold with dimension $n={\rm dim}(M)$. Assume either $n=2$ and that there exists $c_1\in \mathbb{R}$ such that
    $$
\frac{1}{c_1}\leq {\rm det}(\W)\leq c_1,
    $$
    or just that $n>2$.\\
    Then $M$ is $\mathcal{W}$-hyperbolic if and only if the Riemannian manifold $(M,G)$ is hyperbolic, where the new metric tensor $G$ is constructed as 
    $$
G(u,v)=\left\lbrace\begin{array}{lcc}
     e^h  g(u,\mathcal{W}^{-1}(v))&{\rm if}&n=2,\\
     f\cdot ({\rm det}\mathcal{W})^\frac{1}{n-2}\cdot g(u,\mathcal{W}^{-1}(v))&{\rm if}&n>2 ,
\end{array}\right.
    $$
    for any smooth function $h:M\to\mathbb{R}$ and any bounded function $f:M\to \mathbb{R}$,
    $$
    \frac{1}{c_2}\leq f(p)\leq c_2,\quad \forall c_2>0,\, \forall p\in M.
    $$}
\end{theorem}

\begin{proof}
    we show that the new metric tensor $G$ has the property that for any open $\Omega\subset M$ and any compact $K\subset \Omega$ {there exists $c>0$ such that}
\begin{equation}\label{eq:diffcap}
    \frac{1}{c}{\rm Cap}^G(K,\Omega)\leq {\rm Cap}_\mathcal{W}(K,\Omega)\leq c \, {\rm Cap}^G(K,\Omega). 
\end{equation} {where ${\rm Cap}^G(K,\Omega)$ is the Riemannian capacity in  $(M,G)$} ,
$$
{\rm Cap}^G(K,\Omega)=\inf_{\phi\in \mathcal{L}(K,\Omega)}\int_\Omega G(\nabla^G \phi,\nabla^G \phi)dV_G,
$$
where $\mathcal{L}(K,\Omega)$ is the set of locally Lipschitz functions on $M$ with compact support on $\overline{\Omega}$ such that $0\leq \phi\leq 1$ and $\phi\vert_K=1$, 
$$
{\rm Cap}_\mathcal{W}(K,\Omega)=\inf_{\phi\in \mathcal{L}(K,\Omega)}\int_\Omega g(\nabla^g \phi,\mathcal{W}(\nabla^g \phi))dV_g.
$$
To prove inequality \eqref{eq:diffcap}  we only need to check that

$$
\frac{1}{c}G(\nabla^G \phi,\nabla^G \phi)dV_G\leq g(\nabla^g \phi,\mathcal{W}(\nabla^g \phi))dV_g\leq c\, G(\nabla^G \phi,\nabla^G \phi)dV_G
$$
for some $c>0$. Indeed, take a point $p\in M$ and take local coordinates $\{x^i\}$ in a neighborhood of $p$, then the new metric tensor is given by
$$
{
G=F  g_{ij} \left(\W^{-1}\right)_k^jdx^i\otimes dx^k,}
$$
{with
$$
F=\left\lbrace\begin{array}{lcc}
    e^h& {\rm if}& n=2,\\
    f (\rm det \W)^\frac{1}{n-2}&{\rm if}& n>2.
\end{array}\right.
$$
}
Then
\begin{equation}
\begin{aligned}
    G(\nabla^G \phi,\nabla^G \phi)dV_G=&\frac{\partial \phi}{\partial x^i}G^{ij}\frac{\partial \phi}{\partial x^j}\sqrt{{\rm det}G}\, dx^1\wedge\cdots\wedge dx^n\\
    =& {F^{\frac{n}{2}-1}\left({\rm det \W}\right)^{-1/2}\frac{\partial \phi}{\partial x^i}\mathcal{W}^i_kg^{kj}\frac{\partial \phi}{\partial x^j}\sqrt{{\rm det}g}\, dx^1\wedge\cdots\wedge dx^n}\\
    =&{ F^{\frac{n-2}{2}}\left({\rm det \W}\right)^{-1/2}g(\nabla^g \phi,\mathcal{W}(\nabla^g \phi))dV_g.}
\end{aligned}
    \end{equation}
{For the case $n=2$
$$
F^{\frac{n-2}{2}}\left({\rm det \W}\right)^{-\frac{1}{2}}=\left({\rm det \W}\right)^{-\frac{1}{2}}.
$$
Thence the proposition follows by taking 
$c=c_1^{-\frac{1}{2}}$. For the case $n>2$
$$
F^{\frac{n-2}{2}}\left({\rm det \W}\right)^{-\frac{1}{2}}=f^{\frac{n-2}{2}},
$$
and the proposition follows by taking $c=c_2^{\frac{n-2}{2}}$.}
\end{proof}

\subsection{New examples of hyperbolic Riemannian manifolds} \label{examp:R3R6}
In this subsection we will provide new examples of hyperbolic Riemannian manifolds by using our results. Given Corollary \ref{cor:95} and Theorem \ref{teo:81} we can build up families of non-parabolic manifolds. Take for instance the Euclidean space $\mathbb{R}^n$ with its canonical metric tensor 
$
g_{\rm can}
$
and its canonical Levi-Civita connection $\nabla$. The coordinate vector fields
$
\{\frac{\partial}{\partial x_i}\}_{i=1}^n
$
parallelizes $\mathbb{R}^n$ and satisfy that
$$
\nabla_{\frac{\partial}{\partial x_i}}\frac{\partial}{\partial x_j}=0,
$$
for any $i,j$. We now endow $(\mathbb{R}^n,g_{\rm can})$ with the following conductivity: 
$$
\mathcal{W}= \left(\sum_{i=1}^{n-1} f_i(x_{n}) \frac{\partial}{\partial x^i}\otimes dx^i \right)+c \frac{\partial}{\partial x^n}\otimes dx^n,
$$
with $f_i: \mathbb{R}\to \mathbb{R}_+$ positive smooth functions and $c$ a positive constant. Observe that
$$
\begin{aligned}
{\rm div }( \mathcal{W})=&\sum_{i=1}^n\nabla\mathcal{W}\left(\frac{\partial}{\partial x^i},\frac{\partial}{\partial x^i}\right)=\sum_{i=1}^n\nabla_{\frac{\partial}{\partial x^i}}\mathcal{W}\left(\frac{\partial}{\partial x^i}\right) -\sum_{i=1}^n\mathcal{W}\left(\nabla_{\frac{\partial}{\partial x^i}}\frac{\partial}{\partial x^i}\right)\\
=& \sum_{i=1}^n\nabla_{\frac{\partial}{\partial x^i}}\mathcal{W}\left(\frac{\partial}{\partial x^i}\right)=\sum_{i=1}^{n-1}\nabla_{\frac{\partial}{\partial x^i}}\mathcal{W}\left(\frac{\partial}{\partial x^i}\right)+\nabla_{\frac{\partial}{\partial x^{n}}}\mathcal{W}\left(\frac{\partial}{\partial x^{n}}\right)\\
=&\sum_{i=1}^{n-1}\nabla_{\frac{\partial}{\partial x^i}}f_i(x_{n})\frac{\partial}{\partial x^i}+\nabla_{\frac{\partial}{\partial x^{n}}}c\frac{\partial}{\partial x^{n}}=0.
\end{aligned}
$$
Moreover,
$$
\overline{\lambda}(\mathcal{W})=\left(\frac{1}{n}\sum_{i=1}^{{n-1}}f_i(x_{n})\right)+\frac{c}{n},
$$
and
$$
\mathrm{sd}(\mathcal{W})=\left(\frac{1}{n}\sum_{i=1}^{{n-1}}\left(f_i(x_{n})-\overline{\lambda}(\mathcal{W})\right)^2+\frac{(c-\overline{\lambda}(\mathcal{W}))^2}{n}\right)^{1/2}.
$$
Our goal is to choose  functions $f_i$ in such a way that Corollary \ref{cor:95} also applies, i.e.:
$$
\rm CV < \frac{n-2}{2\sqrt{n}}.
$$
Then  
$(\mathbb{R}^n,g_{\rm can},\mathcal{W})$ is $\mathcal{W}$-hyperbolic and by using Theorem \ref{teo:81} we can therefore construct a new metric tensor $G$ on $\mathbb{R}^n$ ($n>2$) in such a way that $(\mathbb{R}^n,G)$ is a  hyperbolic manifold in the classical Riemannian sense:
\begin{example} \label{examp:R6}
    In $\mathbb{R}^6$ we choose the $\W$-associated matrix with respect to the  basis $\{\frac{\partial}{\partial x^i}\}_{i=1}^6$ as follows:
    $$
\mathcal{W}:=\left(
\begin{array}{cccccc}
 e^{x_6^2} & 0 & 0 & 0 & 0 & 0 \\
 0 & e^{-x_6^2} & 0 & 0 & 0 & 0 \\
 0 & 0 & e^{x_6^2} & 0 & 0 & 0 \\
 0 & 0 & 0 & e^{x_6^2} & 0 & 0 \\
 0 & 0 & 0 & 0 & e^{x_6^2} & 0 \\
 0 & 0 & 0 & 0 & 0 & 1 \\
\end{array}
\right).
    $$
    Then
    $$
{\rm cv}(\W)=\frac{ \left(e^{x_6^2}-1\right) \sqrt{5+8 e^{x_6^2}+8 e^{2x_6^2} }}{1+e^{x_6^2}+4 e^{2 x_6^2}}\leq \sqrt{\frac{1}{2}}={\rm CV} < \frac{n-2}{2\sqrt{n}} = \frac{4}{2\sqrt{6}}.
    $$
The associated Riemannian manifold $(\mathbb{R}^6,G)$ given by Theorem \ref{teo:81} has metric tensor
    $$
    \begin{aligned}
        G=&e^{-\frac{x_6^2}{4}}dx^1\otimes dx^1+e^{\frac{7 x_6^2}{4}}dx^2\otimes dx^2+ e^{-\frac{x_6^2}{4}}dx^3\otimes dx^3+e^{-\frac{x_6^2}{4}}dx^4\otimes dx^4\\
        & + e^{-\frac{x_6^2}{4}}dx^5\otimes dx^5+e^{\frac{3 x_6^2}{4}}dx^6\otimes dx^6.
    \end{aligned}
    $$
Then, $(\mathbb{R}^6,G)$ is a hyperbolic manifold but in this case $(\mathbb{R}^6,G)$ is not quasi-isometric to $(\mathbb{R}^6,g_{\rm can})$. Moreover, the Ricci tensor ${\rm Ric}_{ij}={\rm Ric}\left(\frac{\partial }{\partial x^i},\frac{\partial }{\partial x^j}\right)$ of $G$ is given by, see specific computations in \cite{gimeno_i_garcia_2025_15854432},
    $$
 \left(
\begin{array}{cccccc}
 \frac{e^{-x_6^2}}{4} & 0 & 0 & 0 & 0 & 0 \\
 0 & -\frac{1}{4} \left(7 e^{x_6^2}\right) & 0 & 0 & 0 & 0 \\
 0 & 0 & \frac{e^{-x_6^2}}{4} & 0 & 0 & 0 \\
 0 & 0 & 0 & \frac{e^{-x_6^2}}{4} & 0 & 0 \\
 0 & 0 & 0 & 0 & \frac{e^{-x_6^2}}{4} & 0 \\
 0 & 0 & 0 & 0 & 0 & \frac{1}{4} \left(-11 x_6^2-3\right) \\
\end{array}
\right).
$$
Then $(\mathbb{R}^6,G)$ has positive and negative sectional curvatures -- in particular it is not a Cartan-Hadamard manifold.
\end{example}

\section{Conductivity tensors with physical or geometric meaning} \label{sec:GeometricConduc}
We present in this Section some examples of conductive Riemannian manifolds $(M, g, \W)$ (and submanifolds $\Sigma$) with conductivities $\W$ that are inherited/extracted from the curvature tensors of $M$ (and the $\W^{\Sigma}$-mean curvatures of $\Sigma$).

\subsection{The Schouten tensor} \label{sec:Schouten}
Let $(M,g)$ be a Riemannian manifold and ${\rm Ric}$ the Ricci tensor of $M$. The scalar curvature ${\rm R}$ is the trace of the Ricci tensor. Associated with the Ricci tensor, we also have a metrically equivalent $(1,1)$-tensor by ${\rm Ric}(X,Y)=\langle {\rm Ric}(X),Y\rangle$, that we will also call ${\rm Ric}$.

When $n\geq 3$, the Schouten tensor is given by
\begin{displaymath}
{\rm S}={\rm Ric}-\frac{{\rm R}}{2(n-1)}\, g.
\end{displaymath}

The Schouten tensor is a symmetric tensor. Let us study when the Schouten tensor is also positive definite so it can be  applied as a conductivity. Let $\{\lambda_i\}_{i=1}^n$ and $\{e_i\}_{i=1}^n$ be the eigenvalues and eigenvectors of the Ricci operator, respectively. Then,
\begin{eqnarray*}
{\rm S}(e_i,e_i)=\lambda_i -\frac{R}{2(n-1)},
\end{eqnarray*}
and the eigenvalues of the Schouten $(1,1)$-tensor ${\rm S}={\rm Ric}-\frac{{\rm R}}{2(n-2)} {\rm Id}$ are 
$$\sigma_i=\lambda_i -\frac{R}{2(n-1)},$$
with the same eigenvectors.\\

As a consequence, the Schouten tensor is positive definite if and only if, $\lambda_i > \frac{{\rm R}}{2(n-1)}$ for all $i=1,\cdots,\,n$. Notice that this condition implies that 
\begin{displaymath}
{\rm R}=\sum_{i=1}^n \lambda_i > \frac{n {\rm R}}{2(n-1)},
\end{displaymath}
and $\frac{(n-2)\, {\rm R}}{2(n-1)}>0$. Then ${\rm R}>0$, $\lambda_i >0$ for all $i$ and consequently ${\rm Ric} >0$. Having a Ricci tensor positive definite is a necessary (but not sufficient) condition for the Schouten tensor to be positively defined.\\

 The trace of the Schouten tensor is given by ${\rm tr}(S)=\frac{n-2}{2(n-1)}{\rm R}$ and  the divergence by ${\rm div}(S)=\frac{n-2}{2(n-1)}\,\nabla {\rm R}=\nabla {\rm tr(S)}$ (see \cite{Alencar2015},\cite{Cheng-Xu}). Then, it is divergence-free if and only if $M$ has constant scalar curvature.\\

Applying Corollary \ref{corollaryC}, we can state the following:
\begin{corollary} Let $(M^n, g)$ be a Riemannian manifold with a pole, $n\geq 3$, and ${\rm sec}\geq 0$. If the Schouten tensor is positive definite, ${\rm Ric}(\nabla r,\nabla r) \geq \frac{n {\rm R}}{4(n-1)}$, and $\langle \nabla r, \nabla R \rangle \leq 0$ on $M-B_\rho$ for some $\rho>0$, then $M$ is ${\rm S}$-parabolic.
\end{corollary}

\begin{proof} Applying the bounding condition of ${\rm Ric}(\nabla r, \nabla r)$, we have that

\begin{eqnarray*}
 \langle {\rm S}(\nabla r),\nabla r\rangle&=&{\rm Ric}(\nabla r,\nabla r)- \frac{{\rm R}}{2(n-1)}\\
 &\geq& \frac{n {\rm R}}{4(n-1)}-\frac{{\rm R}}{2(n-1)}=\frac{(n-2){\rm R}}{4(n-1)}=\frac{{\rm tr(S)}}{2}.
\end{eqnarray*}

Moreover, 
$$\langle \nabla r, {\rm div}(S) \rangle=\frac{n-2}{2(n-1)}\langle \nabla r, \nabla R \rangle \leq 0.$$

Then, we can apply  Corollary \ref{corollaryC}  with $\mathcal{W}=S$, $w(r)=r$, and $q=2$ to deduce the ${\rm S}$-parabolicity of $M$ since
$$\int_\rho^\infty \frac{1}{s}\,ds=\infty.$$
\end{proof}

If we consider the conductivity $\mathcal{W}=-S$, reasoning in a similar way, we can state the following

\begin{corollary} Let $(M^n, g)$ be a Riemannian manifold with {a pole,} $n\geq 3$, and ${\rm sec}\leq 0$. If the Schouten tensor is negative definite, ${\rm Ric}(\nabla r,\nabla r) \geq \frac{ {\rm R}}{n}$, and $\langle \nabla r, \nabla R \rangle \leq 0$ on $M-B_\rho$ for some $\rho>0$, then $M$ is $(-{\rm S})$-hyperbolic.
\end{corollary}
\begin{remark} Notice that in this case, if the Schouten tensor is negative definite,    $\lambda_i <0$ for all $i$, so ${\rm Ric}<0$ and $R{<}0$ .
\end{remark}

\subsection{The Einstein tensor} \label{sec:Einstein}

Given a Riemannian manifold $(M,g)$ recall that the Einstein tensor is given by
$$
{E}={\rm Ric}-\frac{R}{2}g.
$$
Given an orthonormal basis of $T_pM$ made of eigenvectors of ${\rm Ric}$ this basis will automatically diagonalize $E$.  The point of this section is to use the Einstein tensor as a conductivity. We need hence a positive definite Einstein tensor. 
\begin{lemma}
    Let $(M,g)$ be a Riemannian manifold. Then 
    \begin{enumerate}
        \item If ${\rm sec}<0$ for any point and any tangent plane, $E$ is positive definite.
        \item If ${\rm sec}>0$ for any point and any tangent plane, $-E$ is positive definite.
    \end{enumerate}
\end{lemma}
\begin{proof}
    Let $X$ be a unit eigenvector of $E$, and hence of ${\rm Ric}$. Let us use the following orthonormal basis of eigenvalues of ${\rm Ric}$:
    $$
\{e_1=X, e_2, \cdots, e_n\}.
    $$
    Thence,
    $$
\begin{aligned}
E(X,X)=E(e_1,e_1)=&{\rm Ric}(e_1,e_1)-\frac{R}{2}=\sum_{i=2}^n{\rm sec}(e_1,e_i)-\frac{\sum_{j\neq k} {\rm sec}(e_j,e_k)}{2}\\
=&\sum_{i=2}^n{\rm sec}(e_1,e_i)-\sum_{j< k} {\rm sec}(e_j,e_k)\\
=&\sum_{k=2}^n{\rm sec}(e_1,e_k)-\sum_{j< k} {\rm sec}(e_j,e_k)\\
=&\sum_{k=2}^n{\rm sec}(e_1,e_k)-\sum_{j=1}^{n-1}\sum_{k>j}^n {\rm sec}(e_j,e_k)\\
=&-\sum_{j=2}^{n-1}\sum_{k>j}^n {\rm sec}(e_j,e_k).
\end{aligned}
    $$
    If ${\rm sec}<0$, we conclude that $E(X,X)>0$ and if ${\rm sec}>0$, we conclude that $-E(X,X)>0$ and the lemma is proved.
\end{proof}
By using our theorems we can state the following results concerning the hyperbolicity
\begin{theorem}
    Let $(M,g)$ be a simply connected Riemannian manifold . Let $E$ be the Einstein tensor. Then:
    \begin{enumerate}
    \item If ${\rm dim}(M)=3$ and 
    $
\frac{3}{2}b\leq {\rm sec}\leq b<0
    $ for some constant $b$, 
    $M$ is $E$-hyperbolic.
    \item If ${\rm dim}(M)=4$ and 
    $
{\rm sec}\leq b<0
    $,
    $M$ is $E$-hyperbolic.
\item If ${\rm dim}(M)>4$ and 
    $
{\rm sec}<0
    $,
    $M$ is $E$-hyperbolic.
    \end{enumerate}
\end{theorem}
\begin{proof}
    For the $3$-dimensional case let us denote as $\{e_1, e_2, e_3\}$ the orthonormal basis of eigenvectors of the Ricci tensor of $T_pM$ then,
    the eigenvalues $\{\sigma_1, \sigma_2, \sigma_3\}$ of the Einstein tensor are given by 
    $$
\begin{aligned}
    \sigma_1=& -{\rm sec}(e_2,e_3),\\
    \sigma_2=&-{\rm sec}(e_1,e_3),\\
    \sigma_3=&-{\rm sec}(e_1,e_2).
\end{aligned}
    $$
    By the hypothesis on the sectional curvatures therefore:
    $$
\sigma^*:=\max\{\sigma_1, \sigma_2, \sigma_3\}\leq -\frac{3}{2}b
    $$
    and
    $$
\sigma_*:=\min\{\sigma_1, \sigma_2, \sigma_3\}\geq -b.
    $$
    We are applying Corollary \ref{corollaryC} with $\mathcal{W}=E$, $w(t)=\sinh{\sqrt{-b}\,t}$ and ${q}=2$. We know that ${\rm div} (E)=0$ and 
    $$
\begin{aligned}
    \langle \mathcal{W}(\nabla r), \nabla r\rangle=&\sum_{i=1}^3 \sigma_i \langle \nabla r, e_i\rangle^2\leq \sigma^*\leq -\frac{3}{2}b\\
    \leq &  \frac{3\sigma_*}{{q}}\leq \frac{\sigma_1+\sigma_2+\sigma_3}{{q}}=\frac{{\rm tr}(\mathcal{W})}{{q}}\cdot
\end{aligned}
    $$
 Finally we conclude that $M$ is $E$-hyperbolic because $\int_1^\infty\sinh(\sqrt{-b}s)^{-1}ds<\infty$.\\
 
 For the $n$-dimensional case we know that
 $$
{\rm tr}(E)=-\left(\frac{n}{2}-1\right)R,
 $$
 Thence with an accurate choice of ${q}$ we can obtain
$$
\begin{aligned}
    \langle E(\nabla r), \nabla r\rangle&={\rm Ric}(\nabla r,\nabla r)-\frac{R}{2}\leq -\frac{R}{2}
    \leq &  \frac{-(\frac{n}{2}-1)R}{{q}}=\frac{{\rm tr}(E)}{{q}}\cdot
\end{aligned}
    $$

For $n=4$ we can choose ${q}=2$ and the theorem follows by using $w(t)=\sinh{\sqrt{-b}t}$ (to ensure that $\int_1^\infty\sinh(\sqrt{-b}s)^{-1}ds<\infty$). For $n>4$ we can choose ${q}=3$ and $w(t)=t$ (to ensure that $\int_1^\infty s^{{-2}}ds<\infty$) and the theorem follows.

\end{proof}

For the parabolicity we can state that
\begin{theorem}
    Let $(M,g)$ be a $3$-dimensional Riemannian manifold with a pole $o\in M$, let $r:M\to \mathbb{R}$ the distance function from $o$. Suppose that  any tangent plane has positive sectional curvature and
    $$
{\rm Ric}(\nabla r,\nabla r)\leq \frac{1}{4}R.
    $$
    Then, $M$ is $(-E)$-parabolic.
\end{theorem}
\begin{proof} We have that
    $$
    \begin{aligned}
        &\langle -E(\nabla r),\nabla r\rangle=\frac{R}{2}-{\rm Ric}(\nabla\,r,\nabla\,r), 
    \end{aligned}
    $$
    and 
    $$
{\rm tr}(-E)=\frac{1}{2}R.
    $$
    Thence
    $$
\langle -E(\nabla\, r),\nabla\, r\rangle\geq \frac{1}{4}R.
    $$
We can apply Corollary \ref{corollaryC}  with ${q=2}$ and $w(t)={t}$  and the theorem follows since
$$
\int_1^\infty s^{-1} ds=\infty.
$$
\end{proof}

\subsection{The first Newton transformation} \label{sec:Newton}
Let $M^n$ be an $n$-dimensional Riemannian manifold and $x: M^n \to \widetilde M^{n+1}$ an isometric immersion of $M$ into an $(n+1)$-dimensional Riemannian manifold. Denote by $A=-\nabla N$ the shape operator and by $\lambda_i$ the eigenvalues of $A$. The first Newton transformation $P_1: TM \to TM$ is given by
$$P_1= n\,H I-A,$$
where $H=\displaystyle\frac{1}{n}\sum_{i=1}^n \lambda_i$ is the mean curvature of $M$ with respect the unit normal $N$. If $A>0$, then $P_1$ is positive definite. Moreover, it was shown by Reilly \cite{Reilly} that if $\widetilde M^{n+1}$ is a space of constant sectional curvature then $P_1$ is divergence-free. This conclusion also holds true  if $\widetilde M^{n+1}$ is an Einstein manifold, see \cite{Alencar2015}. 
\begin{corollary}
Let $M^n$ be a manifold with a pole immersed as hypersurface in an Einstein manifold such that ${\rm sec}_M \leq\, 0$ for radial planes. If $A>0$ and $\langle A(\nabla r),\nabla r \rangle \geq \frac{3-n+1}{3}n\, H$, then $M$ is $P_1$-hyperbolic.
\end{corollary}
\begin{proof}
Since ${\rm tr}(P_1)=(n-1)\,n\, H$ and ${\rm div} (P_1)=0$,
\begin{eqnarray*}
\langle P_1(\nabla r), \nabla r \rangle &=& n\,H-\langle A(\nabla r),\nabla r \rangle \leq n\, H- \frac{3-n+1}{3}n\, H\\
&=&\frac{n-1}{3}n\,H=\frac{{\rm tr}(P_1)}{3}.
\end{eqnarray*}
Then, applying Corollary \ref{corollaryC}  with $w(r)=r$, $W=P_1$ and  $q=3$, the result follows.
\end{proof}
{{
\begin{remark} The result is also true for all $q>2$. Since $A>0$ and then $H>0$, using $2<q\leq 3$, we obtain that
$$
\frac{q-n+1}{q}n\, H\leq \frac{3-n+1}{3}n\, H.
$$

\end{remark}
}}
\subsection{Compressible gas} In $\mathbb{R}^n$, a gas is described by a mass density $\rho\geq 0$, a velocity $u$ and a pressure $p\geq \geq 0$ that obey the Euler equations. If the flow is steady (for instance, $\rho$, $u$ and $p$ are independent of time), then 
$$A=\rho u \otimes u+p I$$

is a divergence-free positive symmetric tensor (see \cite{serre}).
\begin{corollary} Let $\theta$ be the angle between the velocity $u$ of a steady compressible gas and the vector field $\nabla r$ in the Euclidean space $\mathbb{R}^n$. Then, 
\begin{itemize}
\item[a) ] If $n\geq 3$ and ${\rm cos}^2(\theta) \leq 1/3$ then $\mathbb{R}^n$ is $A$-hyperbolic. 
\item[b) ] If ${\rm cos}^2(\theta) \geq 1/2$, then $\mathbb{R}^2$ is $A$-parabolic. 
\end{itemize}

Here, $A=\rho\, u \otimes u+p I$, and $\rho,\, p$ are the mass density and the pressure of the gas, respectively.
\end{corollary}
\begin{proof} By the definition of $A$,
$$\langle A(\nabla r),\nabla r\rangle=\rho \langle u, \nabla r\rangle^2 +p\quad \textrm{and}\quad {\rm tr}(A)=\rho\, |u|^2+n\,p.$$\\
Then, if $n\geq 3$
\begin{eqnarray*}
\frac{{\rm tr}(A)}{3}&=&\frac{\rho\, |u|^2+n\,p}{3}\geq 
\rho \frac{|u|^2}{3}  +p\\
&\geq& \rho |u|^2 {\rm cos}^2(\theta) +p =\langle A(\nabla r),\nabla r \rangle. 
\end{eqnarray*}
So applying Corollary \ref{corollaryC}   with $w(r)=r$, $\mathcal{W}=A$, and $q=3$ , we obtain a). Reasoning in a similar way, part b) is a consequence of Corollary \ref{corollaryC}  for $n=q=2$.
\end{proof}

\subsection{Extrinsic examples} \label{subsec:ExtrinExamps}

{In this section we shall see how to apply our results from an extrinsic point of view.}

\begin{example}[Example with zero Gaussian curvature]
    Let us consider the Euclidean space $\mathbb{R}^3$ with the canonical metric tensor $g=dx\otimes dx+dy\otimes dy+dz\otimes dz$. The following vector fields on $\mathbb{R}^3-\{z-\rm{axis}\}$,
    $$
    \left\lbrace
    \begin{array}{rl}
    E_1=&\frac{x}{\sqrt{x^2+y^2}}\partial x-\frac{y}{\sqrt{x^2+y^2}}\partial y       \\
     E_2=&\frac{y}{\sqrt{x^2+y^2}}\partial x+\frac{x}{\sqrt{x^2+y^2}}\partial y     \\
     E_3=&\partial z,
    \end{array}
    \right.
    $$
and the associated $1$-forms
$$ \Theta^1=\frac{x}{\sqrt{x^2+y^2}}dx-\frac{y}{\sqrt{x^2+y^2}}dy,\quad \Theta^2=\frac{y}{\sqrt{x^2+y^2}}dx+\frac{x}{\sqrt{x^2+y^2}}dy,\quad \Theta^3=dz,
$$
satisfy that
$$
\begin{aligned}
    \Theta^i(E_j)=&\delta^i_j,\\
g=&\Theta^1\otimes \Theta^1+\Theta^2\otimes \Theta^2+\Theta^3\otimes \Theta^3,\\
\nabla_{E_2}E_2=&\frac{x^2-y^2}{\left(x^2+y^2\right)^{3/2}}E_1, \\
\nabla_{E_3}E_3=&0.
\end{aligned}
$$
The one-parametric family of surfaces 
$$
\Sigma_\sigma:=\left\lbrace(x,y,z)\in \mathbb{R}^3\, :\, x^2-y^2=\sigma\right\rbrace
$$
is a family of embedded smooth surfaces for $\sigma\neq 0$ and contains a smooth immersed surface for $\sigma=0$ . We will focus in the case $\sigma\neq 0$ where $\Sigma_\sigma$ does not meet the $z$-axis.  In such a case, we can choose as unit normal vector $\nu$ to $\Sigma_\sigma$ the vector field $E_1$ and $\{E_2,E_3\}$ as an orthonormal basis of the tangent space. The mean curvature vector field is given by
$$
\vec{H}=\frac{1}{2}\frac{x^2-y^2}{\left(x^2+y^2\right)^{3/2}}E_1=\frac{1}{2}\frac{\sigma}{\left(x^2+y^2\right)^{3/2}}E_1,
$$
and the tangential component of the gradient of the distance function is given by 
$$
\begin{aligned}
    \nabla^{\Sigma}r=&\langle \nabla^{\Sigma} r,E_2\rangle E_2+\langle \nabla^{\Sigma} r,E_3\rangle E_3\\
    =& E_2(r)E_2+E_3(r)E_3=\frac{2 x y}{\sqrt{x^2+y^2} \sqrt{x^2+y^2+z^2}}E_2+\frac{z}{\sqrt{x^2+y^2+z^2}}E_3.
\end{aligned}
$$
Choosing  positive functions $\lambda_1,\lambda_2$ and $\lambda_3$ the following $(1,1)$ tensor field
$$
\mathcal{W}=\lambda_1(x,y,z) E_1\otimes\theta^1+\lambda_2(x,y,z) E_2\otimes\theta^2+\lambda_3(x,y,z) E_3\otimes\theta^3
$$
endows $(\mathbb{R}^3- \{z-{\rm axis}\},g,\mathcal{W})$ as a conductive Riemannian manifold and $\Sigma_\sigma$ as a non-compact $\mathcal{W}$-compatible submanifold properly immersed in  $\mathbb{R}^3- \{z-{\rm axis}\}$. In order to apply Theorem \ref{thm:extrinsic2} we need to compute the following:
$$
\begin{aligned}
    {\rm tr}_\Sigma(\mathcal{W})=&\lambda_2+\lambda_3,\\
    \langle \mathcal{W}(\nabla^\Sigma r), \nabla^\Sigma r\rangle=&\frac{4 x^2 y^2 \lambda_2}{(x^2+y^2) (x^2+y^2+z^2)}+\frac{z^2\lambda_3}{x^2+y^2+z^2},\\
    {\rm tr}_\Sigma (\nabla \mathcal{W})=&\nabla_{E_2}\mathcal{W}(E_2)-\mathcal{W}(\nabla_{E_2}E_2)+\nabla_{E_3}\mathcal{W}(E_3)-\mathcal{W}({\nabla_{E_3}E_3})\\
    =&\nabla_{E_2}(\lambda_2 E_2)-\lambda_1\frac{x^2-y^2}{\left(x^2+y^2\right)^{3/2}}E_1+\frac{\partial \lambda_3}{\partial z}E_3\\
    =&(\lambda_2-\lambda_1)\frac{x^2-y^2}{\left(x^2+y^2\right)^{3/2}}E_1+\left(\frac{y}{\sqrt{x^2+y^2}}\frac{\partial \lambda_2}{\partial x}+\frac{x}{\sqrt{x^2+y^2}}\frac{\partial \lambda_2}{\partial y}\right)E_2
    \\&+\frac{\partial \lambda_3}{\partial z}E_3.
    \end{aligned}
    $$
   Then, 
   $$
    \begin{aligned}
2\vec{H}_{\mathcal{W}}=& \frac{\sigma\lambda_2}{\left(x^2+y^2\right)^{3/2}}E_1+\frac{1}{\sqrt{x^2+y^2}}\left(y\frac{\partial \lambda_2}{\partial x}+x\frac{\partial \lambda_2}{\partial y}\right)E_2 +\frac{\partial \lambda_3}{\partial z}E_3,\\
\langle2\vec{H}_{\mathcal{W}}, \nabla r\rangle=&\frac{\sigma\lambda_2}{\left(x^2+y^2\right)^{3/2}}E_1(r)+\frac{1}{\sqrt{x^2+y^2}}\left(y\frac{\partial \lambda_2}{\partial x}+x\frac{\partial \lambda_2}{\partial y}\right)E_2(r) +\frac{\partial \lambda_3}{\partial z}E_3(r)\\
=&\frac{\sigma^2\lambda_2}{r\left(x^2+y^2\right)^{2}}+\frac{2xy}{(x^2+y^2)r}\left(y\frac{\partial \lambda_2}{\partial x}+x\frac{\partial \lambda_2}{\partial y}\right) +\frac{\partial \lambda_3}{\partial z}\frac{z}{r},
\end{aligned}
$$
and
$$
\begin{aligned}
    \langle \mathcal{W}(\nabla^\Sigma r), \nabla^\Sigma r\rangle=&\frac{4 x^2 y^2 \lambda_2}{(x^2+y^2) (x^2+y^2+z^2)}+\frac{z^2\lambda_3}{ (x^2+y^2+z^2)}\\
    \leq& \frac{1}{ (x^2+y^2+z^2)}\left( (x^2+ y^2) \lambda_2+z^2\lambda_3\right)\leq \lambda_2+\lambda_3=    {\rm tr}_\Sigma (\mathcal{W}).
\end{aligned}
$$
By choosing $q=1$ condition b.1) of Theorem \ref{thm:extrinsic2} is fulfilled. 
Now we choose $\lambda_2=e^{\frac{r}{2}}$  and $\lambda_3=e^{r}$, then 
$$
\begin{aligned}
\langle2\vec{H}_{\mathcal{W}}, \nabla r\rangle=& \frac{\sigma^2 e^{\frac{r}{2}}}{\rho^4r}+\frac{2x^2y^2e^{\frac{r^2}{2}}}{(x^2+y^2)r^2}+\frac{z^2e^{r^2}}{r^2}\geq \frac{2x^2y^2e^{\frac{r^2}{2}}}{(x^2+y^2)r^2}+\frac{z^2e^{r^2}}{2r^2}\\
=&\frac{1}{2}\langle \mathcal{W}(\nabla^\Sigma r), \nabla^\Sigma r\rangle.
\end{aligned}
$$ Hence we can use the function 
$$
\theta(r)=\frac{1}{2}
$$
for condition b.2) in the above theorem and thence for $\rho=1$,
$$
h(t)=\frac{{t-1}}{2}.
$$
We conclude that $\Sigma_\sigma$ is $\mathcal{W}$-hyperbolic since
$$
\int_1^\infty e^{{\frac{1-s}{2}}}ds={2}<\infty.
$$
\end{example}

\begin{example}[Example with positive Gaussian curvature ] \label{examp:Paraboloid}
In this example we will consider a paraboloid (with positive Gaussian curvature, and thus parabolic in the Riemannian sense) immersed in the $3$-dimensional Euclidean space. We will endow $\mathbb{R}^3$ with a specific conductivity and use our results to show that the paraboloid is $\W$-hyperbolic with the inherited conductivity. This is an example that shows how to use the conductivity on non-compact surfaces to produce, (in diffusion terms), a "fast driven escape to infinity" effect on the ends of the surface that allows a transformation of the surface into a $\W$-hyperbolic surface.

Our example will be a surface immersed in the conductive $3$-dimensional Euclidean space $(M^{n}, g, \W) = (\mathbb{R}^{3}, g_{\rm can}, \W)$ with the following conductivity tensor field 
\begin{equation}
    \W = \sum_{i=1}^{3}\lambda_{i}\cdot E_{i} \otimes \theta^{i} \quad ,
\end{equation}
where the eigenvectors of $\W$, $\lambda_{i}$, are positive functions on $\mathbb{R}^{3}$ and $\theta^{i}$ are the dual one-forms metrically associated with the following orthonormal vector fields
\begin{equation}
\begin{aligned}
E_1:=&\frac{x}{\sqrt{\rho^2+1}}\partial_x+\frac{y}{\sqrt{\rho^2+1}}\partial_y-\frac{1}{\sqrt{\rho^2+1}}\partial_z\\
E_2:=&\frac{x}{\rho\sqrt{\rho^2+1}}\partial_x+\frac{y}{\rho\sqrt{\rho^2+1}}\partial_y+\frac{\rho}{\sqrt{\rho^2+1}}\partial_z\\
E_3:=&E_1\wedge E_2
=\frac{y}{\rho}\partial_x-\frac{x}{\rho}\partial_y,
\end{aligned}
\end{equation}
where $\rho=\sqrt{x^2+y^2}$.

Observe that despite the fact that $E_1,E_2,E_3$ are not well defined on $\rho=0$, if we choose $\lambda_1=\lambda_2=\lambda_3=1$, then $\W=Id$ and it admits an obvious extension to $\rho=0$. Thence we will always impose  that $\lambda_1=\lambda_2=\lambda_3=1$ close to the points where $\rho=0$, namely, close to the $z$-axis. For the points of $\mathbb{R}^3$ with $\rho>1$ we will chose
$$
\left\lbrace
\begin{array}{ccl}
\lambda_1&=&1,  \\
     \lambda_2&=&r^{\frac{16}{9}} e^{\frac{9}{8}r},\\
     \lambda_3&=&\frac{r^{\frac{16}{9}} e^{\frac{9}{8}r}}{2(1+4r^2)}.
\end{array}
\right.
$$
{By using the comparison for the capacities as it is done in Section  \ref{sec:com:cap}, taking into account that $\mathbb{R}^3$ is an hyperbolic manifold (in the Riemannian sense), and taking into account that outside a geodesic of radius large enough and some positive $\epsilon\in (0,1)$ 
$$
\lambda_*\geq 1-\epsilon,
$$
we can state that $(\mathbb{R}^3,g_{\rm can},\W)$ is $\W$-hyperbolic.
}

We consider the following paraboloid as the surface of $\mathbb{R}^3$ given by
$$
\Sigma:=\left\{(x,y,z)\in \mathbb{R}^3\, :\, z=\frac{x^2+y^2}{2}\right\}.
$$
Thence $\Sigma$ is the $0$-level set of 
$$
(x,y,z)\mapsto F(x,y,z)=x^2+y^2-2z.
$$
By using the gradient of $F$  and normalize we can found a unit normal vector field as
$$
(x,y,z)\mapsto N:=\frac{x}{\sqrt{\rho^2+1}}\partial_x+\frac{y}{\sqrt{\rho^2+1}}\partial_y-\frac{1}{\sqrt{\rho^2+1}}\partial_z=E_1.
$$
Then $(\Sigma, g_{|_{\Sigma}},\W_{\Sigma})$ is a conductive hypersurface  because the given ambient conductivity tensor $\W$ induces a well defined conductivity tensor field $\W^{\Sigma}$ on $\Sigma$. Indeed for any $p\in \Sigma$
$$
\W(T_p\Sigma)=\W\left({\rm span}_\mathbb{R}\{E_2(p),E_3(p)\}\right)=T_p\Sigma,
$$
and
$$
\W((T_p\Sigma)^\perp)=\W\left({\rm span}_\mathbb{R}\{E_1(p)\}\right)=(T_p\Sigma)^\perp.
$$
We are now applying Theorem \ref{thm:extrinsic2} with $\theta(r)=1$ because 
$$
\langle \mathcal{W}({\nabla^\Sigma}r), \nabla^\Sigma r\rangle=\lambda_2\left(1-\frac{\rho^4}{\rho^6+5 \rho^4+4 \rho^2}\right)\leq \lambda_2<\lambda_2+\lambda_3=1\cdot{\rm tr}_\Sigma \mathcal{W}
$$
and moreover can be checked that
\begin{align*}
    \langle 2\Vec{H}_\mathcal{W},\nabla r\rangle\geq&r^\frac{16}{9}e^{\frac{9}{8}r}\geq \langle \mathcal{W}({\nabla^\Sigma}r), \nabla^\Sigma r\rangle.
\end{align*}
Taking finally $h(t)=t$ we conclude that
$$
\int_1^\infty e^{-t}dt<\infty,
$$
and hence, $\Sigma$ is $\W$-hyperbolic.

\begin{remark}
   With reference to the setting in the above example: For any $p\neq 0$ in $\Sigma$ the tangent space at $p$ is spanned by the eigenvectors $E_2,E_3$ of $\W$ 
    $$
T_p\Sigma={\rm span}_{\mathbb{R}}\{E_2(p),E_3(p)\}
    $$
    and 
      $$
(T_p\Sigma)^\perp={\rm span}_{\mathbb{R}}\{E_1(p)\}.
    $$
    Moreover $E_2$ is proportional to the tangential gradient of the distance function to $0$. In the case $\lambda_1=\lambda_2=\lambda_3=1$ we had the classical paraboloid (and hence a parabolic surface). But since $\lambda_2=r^\frac{16}{9}e^{\frac{9}{8}h(r)}$ far away from $0$ there is a ``fast-driven-escape effect towards infinity'' in the end of the paraboloid. In Figure \ref{fig:paraboloid} we illustrate this behavior when $h(r)=r$.
\end{remark}
\begin{figure}
    \centering
    \includegraphics[scale=0.5]{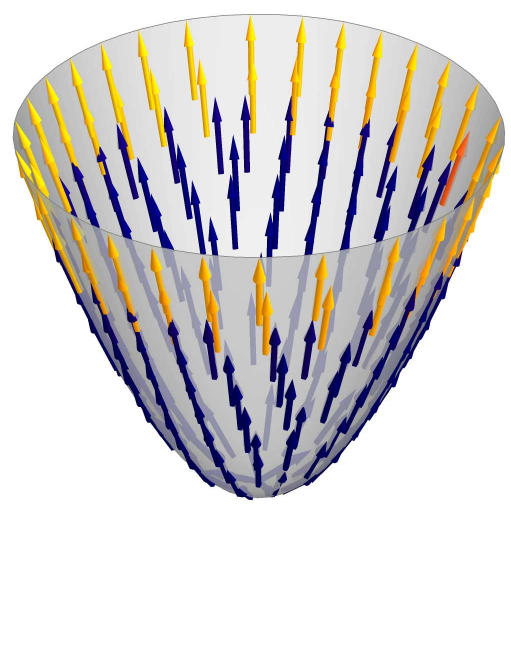}\quad\includegraphics[scale=0.5]{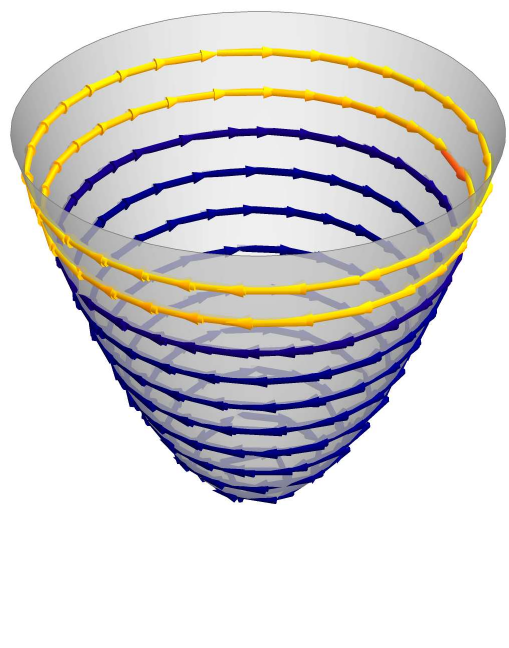}
    \caption{The vector field $\lambda_2E_2$ and $\lambda_3E_3$ on the paraboloid. The vector field $E_2$ is ``aligned" with the end of the paraboloid and $\lambda_2$ is growing as $r\to \infty$.}
    \label{fig:paraboloid}
\end{figure}
\end{example}

\subsection*{Acknowledgement}\label{subsecAcknowledgement}
The authors wish to express their gratitude to the Department of Applied Mathematics and Computer Science at the Technical University of Denmark and to the Department of Mathematics at Universitat Jaume I for their generous hospitality during the course of this work.
\bibliographystyle{plain}

\bibliography{POTENTIAL}

\end{document}